\documentclass[a4paper, 11pt]{article}
\usepackage[margin=3cm]{geometry}
\usepackage[onehalfspacing]{setspace}
\usepackage{amsmath}
\usepackage{amsfonts}
\usepackage{amsthm}
\usepackage{mathtools}
\usepackage{enumitem}
\usepackage{hyperref}
\numberwithin{equation}{section}
\usepackage{graphicx}

\usepackage{tikz}
\usetikzlibrary{positioning,decorations.pathreplacing, bending}

\usepackage{biblatex}
\newtheorem{lemma}{Lemma}

\newtheorem{theorem}[lemma]{Theorem}
\newtheorem{corollary}[lemma]{Corollary}

\theoremstyle{definition}
\newtheorem{definition}{Definition}

\newtheorem{remark}{Remark}
\theoremstyle{remark}
\newtheorem*{example}{Example}

\DeclareMathOperator{\sgn}{sgn}
\DeclareMathOperator{\prob}{\mathbb P}
\DeclareMathOperator{\mean}{\mathbb E}
\DeclareMathOperator{\supp}{supp}
\newcommand{\SAT}{\mathsf{SAT}}

\title{On the Regularity of Random 2-SAT and 3-SAT}
\author{Andreas Basse-O'Connor, Tobias~L.\ Overgaard, Mette Skj{\o}tt}
\date{\today}

\begin{document}
\maketitle

\begin{abstract}
We consider the random $k$-{SAT} problem with $n$ variables, $m=m(n)$ clauses, and clause density $\alpha\coloneqq\lim_{n\to\infty}m/n$ for $k=2,3$. It is known that if~$\alpha$ is small enough, then the random $k$-SAT problem admits a solution with high probability, which we interpret as the problem being under-constrained. In this paper, we quantify exactly how under-constrained the random $k$-SAT problems are by determining their degrees of freedom, which we define as the threshold for the number of variables we can fix to an arbitrary value before the problem no longer is solvable with high probability. Our main result shows that the random $2$-SAT and $3$-SAT problems have $n/m^{1/2}$ and $n/m^{1/3}$ degrees of freedom, respectively. We also explicitly compute the corresponding threshold functions. Our result shows that the threshold function for the random $2$-SAT problem is regular, while it is non-regular for the random $3$-SAT problem. By regular, we mean continuous and analytic on the interior of its support. This result shows that the random $3$-SAT problem is more sensitive to small changes in the clause density~$\alpha$ than the random $2$-SAT problem.
\end{abstract}

\newpage
\tableofcontents

\newpage
\section{Introduction}

\subsection{Background}
For more than half a century, the Boolean $k$-satisfiability ($k$-SAT) problem has enjoyed continued interest in the field of computer science. The $3$-SAT problem was one of the five original $NP$-complete problems from Cooks seminal paper~\cite{cook71}, viz.\ $P=NP$ is equivalent to the existence of a polynomial-time algorithm for solving $3$-SAT problem instances. The $3$-SAT problem thus lies at the heart of the famous $P$~vs.~$NP$ problem. Similarly, the $2$-SAT problem lies at the heart of the important $L$ vs. $NL$ problem, as $L=NL$ is equivalent to the existence of a logarithmic-space algorithm for solving $2$-SAT problem instances (Thm.~16.3 in~\cite{papadimitriou94}), that is, the $2$-SAT problem is $NL$-complete. The $k$-SAT problem has also turned out to have great practical significance in artificial intelligence, bioinformatics, hardware verification, and more (see~\cite{GGW06,marques08}).

In applications, there seemed to be a major gap between the theoretical bounds for the runtime of SAT solving algorithms and the comparatively fast performance of SAT solvers (see~\cite{goldberg79,FP83,CKT91,SML96}). Thus, the study of the \emph{random} $k$-SAT problem arose, motivated by an interest in understanding the typical structure of a SAT problem instance (compared to the hitherto worst case analysis) and understanding where to find the ``hard'' SAT problems. The random $k$-SAT problem has for the last three decades remained highly relevant in the fields of probability theory and statistical physics (see~\cite{knuth15,DSS22} and references therein).

The random $k$-SAT problem~$\varphi$---also called a random $k$-CNF formula---with $m$ clauses and $n$ variables is obtained by taking $m$ i.i.d.\ random disjunctive clauses $C_1,\dots,C_m$ of length~$k$, i.e.\ of the form $C(x)=\sigma_1 x_{v_1}\lor\dots\lor\sigma_k x_{v_k}$, where $\sigma_1,\dots,\sigma_k\in\{-1,1\}$, and where $v_1,\dots,v_k\in[n]\coloneqq\{ 1,\dots,n\}$ are pairwise distinct. Each $C_j$ is chosen uniformly at random among the $2^k\binom{n}{k}$ possibilities. Then $\varphi$ is the conjunction/minimum of these, $\varphi\coloneqq C_1\land\dots\land C_m$. The problem is to determine whether there exists an assignment $x=(x_1,\dots,x_n)\in\{-1,1\}^n$ such that $\varphi(x)=1$, i.e.\ such that all the clauses are simultaneously ``satisfied''. We call such an~$x$ a \emph{solution} to~$\varphi$, and in the affirmative case, we say that $\varphi$ is \emph{satisfiable} and write $\varphi\in\SAT$. We will consider the random $k$-SAT problem asymptotically, so in the following, the random $k$-SAT problem/a random $k$-CNF formula with $m=m(n)$ clauses and $n$ variables refers to a \emph{sequence} of random $k$-CNF formulas, where the $n$'th term is a random $k$-CNF formula with $m(n)$ clauses and $n$ variables as described above.

In 1992 it was conjectured in~\cite{CR92} that for every $k\geq 2$ there exists a critical value $\alpha_\textup{c}(k)>0$, such that if $\varphi$ is a  random $k$-CNF formula with $m=m(n)$ clauses and $n$ variables, where $m/n\to\alpha$ as $n\to\infty$, then
\begin{equation}
\label{satconj}
    \lim_{n\to\infty}\prob(\varphi\in\SAT)=
    \begin{cases}
        1,&\text{if $\alpha<\alpha_\textup{c}(k)$}, \\
        0,&\text{if $\alpha>\alpha_\textup{c}(k)$},
    \end{cases}
\end{equation}
that is, the satisfiability of the random $k$-SAT problem is conjectured to undergo a phase transition as $\alpha$, the \emph{clause density}, crosses a critical threshold $\alpha_\textup{c}(k)$. In the regime $\alpha<\alpha_\textup{c}(k)$ in~\eqref{satconj}, where the problem is satisfiable with high probability (w.h.p.), we will call the random $k$-SAT problem \emph{under-constrained}.

In the present paper, we focus on the random $2$-SAT and $3$-SAT problems. In regards to the former, the satisfiability conjecture~\eqref{satconj} was proved independently in~\cite{CR92} and~\cite{goerdt96}, where it was shown that $\alpha_\textup{c}(2)=1$. Much more has since been discovered about the random $2$-SAT problem, and overall it remains an interesting and actively researched problem; see for instance~\cite{goerdt99,verhoeven99,BBCKW01,ACHLMPZ21,CCMRRZZ24}.

Conjecture~\eqref{satconj} has recently been proved for all $k\geq k_0$ in the monumental work~\cite{DSS22}, where $k_0$ is a large unknown constant, but for $k=3$ the problem remains elusive. Thus, the sharp satisfiability threshold for the random $3$-SAT problem is still an important open question. However, much research activity has been directed towards overcoming this challenge, and increasingly tighter bounds on which $\alpha$'s result in asymptotic satisfiability or unsatisfiability have been produced throughout the years (see \cite{FP83,franco84,BFU93,ED95,KMSP95,FS96,DB97,KKKS98,zito99,friedgut99,achlioptas00,JSW00,AS00,DBM02,HS03,KKL06,KKSVZ05,DKMP09}), showing for example that the random $3$-SAT problem is under-constrained when $\alpha<3.52$.

\subsection{Main result}
Let $\varphi$ be an under-constrained random $2$- or $3$-CNF formula. We pose and aim to answer the following question: \emph{how} under-constrained is~$\varphi$? Our proposed solution is based on the following idea: there are too many degrees of freedom in the variables $x_1,\dots,x_n$ compared to the number of clauses in~$\varphi$, so fix a number, say $f=f(n)$, of the variables $x_1,\dots,x_n$, each to either $1$ or $-1$. Does there still w.h.p.\ exist a solution to~$\varphi$ in this restricted search space? Presumably, if~$f$ is large enough, the answer will be no. We will represent the ``under-constrainedness'' of~$\varphi$ as the critical mass of the number~$f$ of variables needed to be fixed before~$\varphi$ becomes unsatisfiable w.h.p., and we call this number the \emph{degrees of freedom} in~$\varphi$. This is an analogue concept to the nullity, i.e.\ the size of the kernel/null space, of a (random) matrix; it is the number of variables we can ``freely'' fix before the corresponding system of equations no longer has a solution~(w.h.p.)

We say that a subset $\mathcal L$ of $\pm[n]\coloneqq\{-n,\dots,-1,1,\dots,n\}$ is \emph{consistent} if, for all $v\in[n]$, at most one of $v$ and $-v$ is in~$\mathcal L$. Here, $v\in\mathcal L$ represents fixing $x_v$ to~$1$ and $-v\in\mathcal L$ represents fixing $x_v$ to~$-1$ for each $v\in[n]$. Let $\varphi_\mathcal L$ denote the formula $\varphi$ with these restrictions to its domain.

\begin{definition}[Degrees of freedom in the random $k$-SAT problem]
    The random $k$-SAT problem~$\varphi$ with $m=m(n)$ clauses and $n$ variables has $f_\ast=f_\ast(n)$ \emph{degrees of freedom} if, for all $f=f(n)$ and all consistent (non-random) subsets $\mathcal L\subseteq\pm[n]$ with $\lvert\mathcal L\rvert=f$, it holds that $f_\ast$ is a threshold in~$f$ for the satisfiability of~$\varphi_\mathcal L$, that is, if
    \begin{equation}
    \label{varexcess}
        \lim_{n\to\infty}\prob(\varphi_\mathcal L\in\SAT)=
        \begin{cases}
            1,&\text{when $f/f_\ast\to 0$}, \\
            0,&\text{when $f/f_\ast\to\infty$}.
        \end{cases}
    \end{equation}
\end{definition}

The concept of a threshold hearkens back to the seminal paper~\cite{ER60} by Erd\H{o}s and R\'{e}nyi. In Theorem~\ref{mainthm} below, we find a threshold $f_\ast$, which we also show is unique up to asymptotic order, in the cases $k=2,3$. Furthermore, we find an explicit expression for the \emph{threshold function}
\begin{equation}
\label{defthreshfunc}
    \lim_{n\to\infty}\prob(\varphi_\mathcal L\in\SAT)
\end{equation}
when $0<\lim_{n\to\infty}(f/f_\ast)<\infty$. In the following, we interpret $e^{-\infty}\coloneqq 0$.

\begin{theorem}
\label{mainthm}
    Let $\varphi$ be a random $k$-CNF formula with $m=m(n)$ clauses and $n$ variables such that $m\to\infty$ and $m/n\to\alpha$. Let $\mathcal L\subseteq\pm[n]$ be a consistent set with $f=f(n)$ elements. Let finally $0\leq\gamma\leq\infty$.
    \begin{enumerate}
        \item[{$k=2:$}] If $0\leq\alpha<1$ and $f\sqrt{m}/n\to\gamma$, then
        \begin{equation*}
            \lim_{n\to\infty}\prob(\varphi_\mathcal L\in\SAT)=e^{-(\gamma/2)^2(1-\alpha)^{-1}}.
        \end{equation*}
        In particular, the random $2$-SAT problem with $m$ clauses and $n$ variables has $n/\sqrt{m}$ degrees of freedom for clause density $\alpha<1$.
        \item[{$k=3:$}] If $0\leq\alpha<3.145$ and $fm^{1/3}/n\to\gamma$, then
        \begin{equation*}
            \lim_{n\to\infty}\prob(\varphi_\mathcal L\in\SAT)=e^{-(\gamma/2)^3}.
        \end{equation*}
        In particular, the random $3$-SAT problem with $m$ clauses and $n$ variables has $n/m^{1/3}$ degrees of freedom for clause density $\alpha<3.145$.
    \end{enumerate}
\end{theorem}

\begin{remark}
\label{rem1}
    In both cases $k=2,3$ in Theorem~\ref{mainthm}, there are two distinct scenarios:
    \begin{enumerate}
        \item[{$\alpha=0:$}] Theorem~\ref{mainthm} in particular applies to the ``ultra'' under-constrained case $\alpha=0$ not traditionally studied in the literature, showing precisely how the number of constraints in a random $2$- or $3$-CNF formula affects the number of degrees of freedom. As an example, the random $2$-SAT problem with $n$ variables and only $\log(n)$ clauses has $n/\sqrt{\log(n)}$ degrees of freedom.
        \item[{$\alpha>0:$}] For $\alpha>0$, the random $k$-SAT problem with $m\sim\alpha n$ clauses and $n$ variables has $n^{1-1/k}$ degrees of freedom when $0<\alpha<1$ respectively $0<\alpha<3.145$ for $k=2,3$. Furthermore, the threshold functions are given by $e^{-(\beta/2)^2\alpha (1-\alpha)^{-1}}$ and $e^{-(\beta/2)^3\alpha}$ when $f/n^{1-1/k}\to\beta$ for some $0\leq\beta\leq\infty$. This statement is equivalent to Theorem~\ref{mainthm} and follows directly from the asymptotic equivalence $m\sim\alpha n$. An advantage of the parameter~$\beta$ is that it only depends on~$f$, where $\gamma$ depends on both~$m$ and~$f$; indeed, $\gamma=\beta\alpha^{1/k}$.
    \end{enumerate}
\end{remark}

\begin{figure}[htb]
    \centering
    \includegraphics[width=0.8\linewidth]{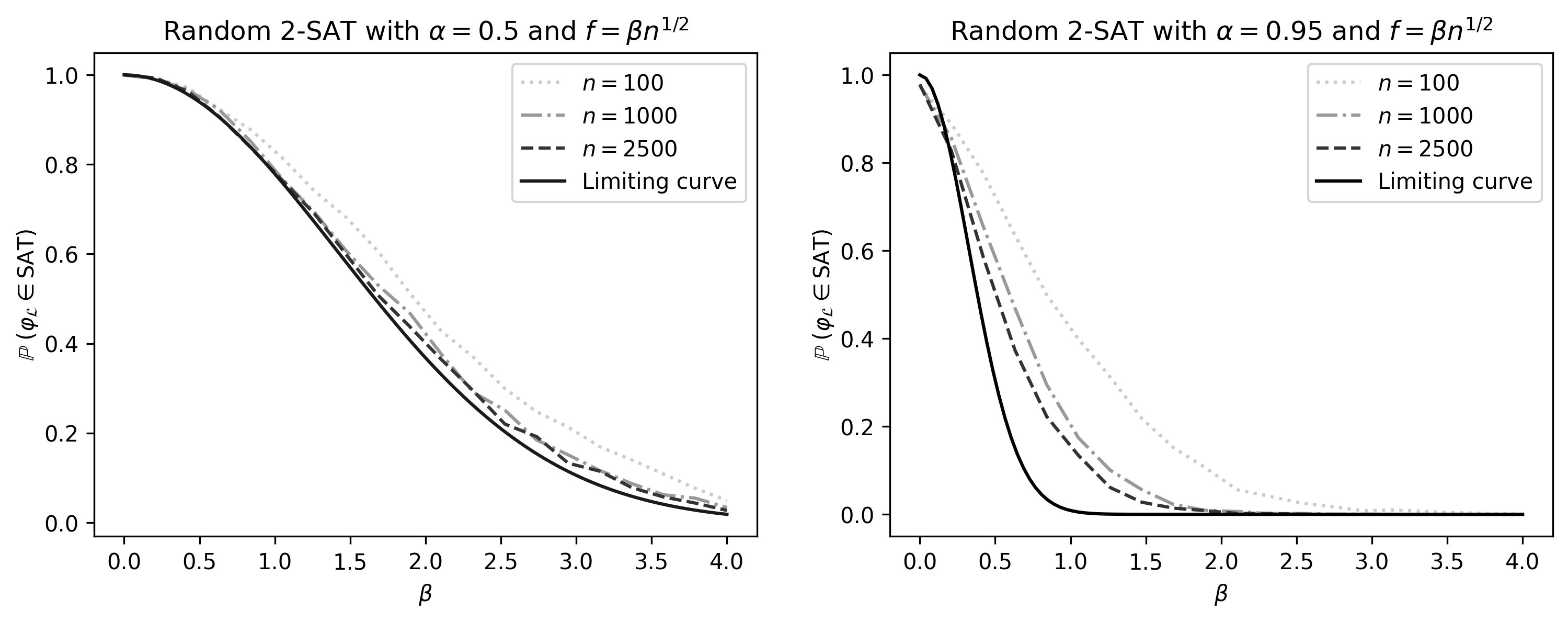}
    \includegraphics[width=0.8\linewidth]{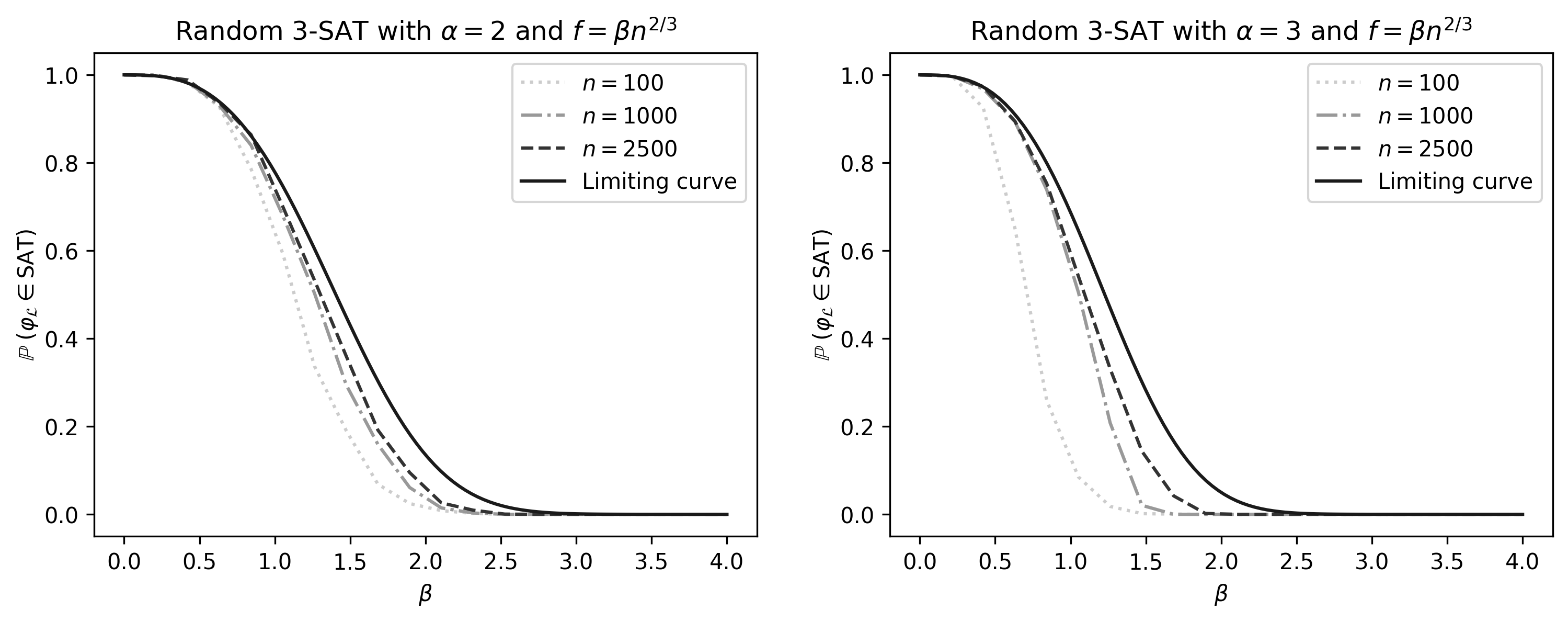}
    \caption{Finite size sampling of the threshold functions corresponding to the degrees of freedom in the random $2$- and $3$-SAT problems. Each datapoint (40 for each curve) is comprised of 2000 simulations.}
    \label{figsatsim}
\end{figure}

Figure~\ref{figsatsim} shows finite size sampling of the threshold functions from Theorem~\ref{mainthm} for fixed values of~$\alpha$. The simulations suggest that the convergence is faster for smaller values of~$\alpha$.

\begin{remark}
\label{existence}
    We now prove uniqueness modulo asymptotic order of the degrees of freedom for the random $k$-SAT problem, $k=2,3$. Let $\varphi$ be a random $k$-CNF formula with $m=m(n)$ clauses and $n$ variables, $m/n\to\alpha$, where $\alpha<1$ if $k=2$ or $\alpha<3.145$ if $k=3$, and let $\mathcal L\subseteq\pm[n]$ be consistent. We have seen in Theorem~\ref{mainthm} that $f_\ast=n/m^{1/k}$ is a threshold in $\lvert\mathcal L\rvert$ for the satisfiability of $\varphi_\mathcal L$, cf.~\eqref{varexcess}. Assume that $g_\ast=g_\ast(n)$ is another such threshold. Then, taking $\lvert\mathcal L\rvert=f_\ast$, we get $0<\lim_{n\to\infty}\prob(\varphi_\mathcal L\in\SAT)<1$ from Theorem~\ref{mainthm}. Hence, we cannot have $f_\ast/g_\ast\to 0$ or $f_\ast/g_\ast\to\infty$ by~\eqref{varexcess}, as this would imply $\lim_{n\to\infty}\prob(\varphi_\mathcal L\in\SAT)=1$ or $\lim_{n\to\infty}\prob(\varphi_\mathcal L\in\SAT)=0$, respectively. Furthermore, by considering sub-sequences, we cannot even have $\liminf_{n\to\infty}(f_\ast/g_\ast)=0$ or $\limsup_{n\to\infty}(f_\ast/g_\ast)=\infty$, that is, $f_\ast$ and $g_\ast$ are of the same asymptotic order.
\end{remark}

We see that the threshold functions $e^{-(\gamma/2)^2(1-\alpha)^{-1}}$ and $e^{-(\gamma/2)^3}$ are real analytic functions in~$\gamma$ that are bounded away from~$0$ and~$1$ (when $0<\gamma<\infty$), and thus the threshold~$f_\ast$ in~\eqref{varexcess} is \emph{not} sharp. A sharp threshold would correspond to a step function from~$1$ to~$0$ at some critical value~$\gamma_\textup{c}$, which is not what we find. Hence, this is a markedly different type of threshold compared to the sharp threshold seen for example in the Boolean satisfiability conjecture~\eqref{satconj} (still only conjecture for $k=3$ but proved for $k=2$).

Returning again to our original question, we now have a way to quantify under-constraindness in the random $k$-SAT problem. Let us compare the random $2$-SAT and $3$-SAT problems: intuitively, clauses of length $3$ (ternary clauses) are ``less constraining'' than clauses of length $2$ (binary clauses). Indeed, the degrees of freedom also reflect this; let $\varphi$ be a random $3$-CNF formula with $n$ variables and $m\sim\alpha n$ clauses, where $0<\alpha<3.145$. Then from Theorem~\ref{mainthm} it follows that $\varphi$ has $n^{2/3}$ degrees of freedom, which is the same as a random $2$-CNF formula with $n$ variables and only $\Theta(n^{2/3})$ clauses!

To the best of our knowledge, Theorem~\ref{mainthm} is the first result of its kind. We believe that degrees of freedom will become an important concept in the analysis of random constraint satisfaction problems.

\subsection{The search tree}
When searching for a solution to a given $k$-SAT instance, a common approach is to fix the value of~$x_1$ (e.g., set $x_1=1$). If this assignment does not violate any clauses, we then proceed to fix~$x_2$ (e.g., set $x_2=-1$). By continuing this process, we generate a \emph{search tree}, as illustrated in Figure~\ref{figfixintree}. The degrees of freedom~$f_\ast$ for the random $k$-SAT problem can be thought of as the maximum depth in the search tree that can be reached without violating the $k$-SAT instance, thus avoiding backtracking. According to Theorem~\ref{mainthm}, for the random $2$-SAT problem with $\alpha<1$, this maximum depth is~$n/\sqrt{m}$, while for the random $3$-SAT problem with $\alpha<3.145$, the maximum depth is~$n/m^{1/3}$. To the best of our knowledge, this is one of the first results quantifying the additional flexibility in solving the random $3$-SAT problem compared to the random $2$-SAT problem.

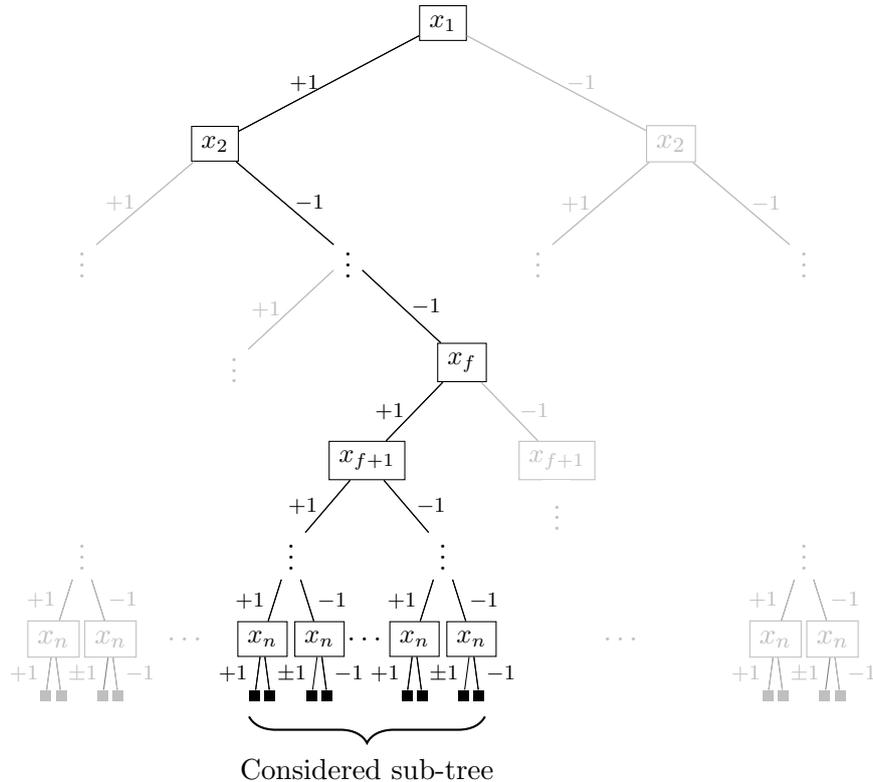
\begin{figure}[htb]
\centering
\begin{tikzpicture}
    [
    level 1/.style={sibling distance=6 cm, level distance = 1.6 cm},
    level 2/.style={sibling distance=3.5 cm, level distance = 1.5 cm},
    level 3/.style={sibling distance=3 cm, level distance = 1.4 cm},
    level 4/.style={sibling distance=2.5 cm, level distance = 1.3 cm},
    level 5/.style={sibling distance=2 cm, level distance = 1.177 cm},
    level 6/.style={sibling distance=0.75 cm, level distance = 1.19 cm},
    level 7/.style={sibling distance=0.2 cm, level distance = 0.75 cm},
    ]
    \node [draw, fill=white, text=black] (1) {\small{$x_1$}}
        child {node [draw, fill=white, text=black] (2_1) {\small{$x_2$}}
            child [lightgray] {node [text=lightgray] (3_1) {$\vdots$}
                child[white] {node (4_1) {!}
                    child {node (5_1) {!}
                        child {node [text=lightgray] (6_1) {$\vdots$}
                            child [gray]{node[draw=lightgray, text=lightgray] (7_1) {$x_n$}
                                child {node [fill=lightgray,inner sep=0.07 cm] (8_1) {}}
                                child {node [fill=lightgray,inner sep=0.07 cm] (8_2) {}}
                            } 
                            child [gray]{node[draw=lightgray, text=lightgray] (7_2) {$x_n$}
                                child {node [fill=lightgray,inner sep=0.07 cm] (8_3) {}}
                                child {node [fill=lightgray,inner sep=0.07 cm] (8_4) {}}
                            } 
                        } 
                    } 
                }
            }
            child {node (3_2) {$\vdots$}
                child [lightgray] {node[ text=lightgray] (4_3) {$\vdots$}
                }
                child {node[draw, fill=white, text=black] (4_4) {\small{$x_f$}}
                    child {node[draw, fill=white, text=black] (5_2) {\small{$x_{f+1}$}}
                        child {node (6_2) {$\vdots$}
                            child {node[draw, fill=white, text=black] (7_3) {\small{$x_n$}}
                                child {node [fill=black,inner sep=0.07 cm] (8_5) {}}
                                child {node [fill=black,inner sep=0.07 cm] (8_6) {}}
                            }
                            child {node[draw, fill=white, text=black] (7_4) {\small{$x_n$}}
                                child {node [fill=black,inner sep=0.07 cm] (8_7) {}}
                                child {node [fill=black,inner sep=0.07 cm] (8_8) {}}
                            }
                        }
                        child {node (6_3) {$\vdots$}
                            child {node[draw, fill=white, text=black] (7_5) {\small{$x_n$}}
                                child {node [fill=black,inner sep=0.07 cm] (8_9) {}}
                                child {node [fill=black,inner sep=0.07 cm] (8_10) {}}
                            }
                            child {node[draw, fill=white, text=black] (7_6) {\small{$x_n$}}
                                child {node [fill=black,inner sep=0.07 cm] (8_11) {}}
                                child {node [fill=black,inner sep=0.07 cm] (8_12) {}}
                            }
                        }
                    }
                    child[lightgray] {node[draw] (5_3) {\small{$x_{f+1}$}}
                        child {node (6_4) {}
                        }
                    }
                }
            }
        }
        child[lightgray] {node[draw=lightgray, text=lightgray] (2_2) {$x_2$}
            child {node[text=lightgray] (3_3) {$\vdots$}
            }
            child {node[text=lightgray] (3_4) {$\vdots$}
                child[white] {node (4_6) {!}
                    child {node (5_4) {!}
                        child {node [text=lightgray] (6_6) {$\vdots$}
                            child [gray]{node[draw=lightgray, text=lightgray] (7_7) {$x_n$}
                                child {node [ fill=lightgray,inner sep=0.07 cm] (8_13) {}}
                                child {node [ fill=lightgray,inner sep=0.07 cm] (8_14) {}}
                            } 
                            child [gray]{node[draw=lightgray, text=lightgray] (7_8) {$x_n$}
                                child {node [fill=lightgray,inner sep=0.07 cm] (8_15) {}}
                                child {node [ fill=lightgray,inner sep=0.07 cm] (8_16) {}}
                            } 
                        } 
                    } 
                }
            }
        }
    ;

    \path (7_2) -- (7_3) node [midway, text=lightgray] {$\dotsc$};
    \path (7_6) -- (7_7) node [midway, text=lightgray] {$\dotsc$};
    \path (7_4) -- (7_5) node [midway, text=black] {$\dotsc$};
    
    \draw (1)--node[midway, left]{\scriptsize $+1$}(2_1);
    \draw[lightgray] (1)--node[midway, right]{\scriptsize $-1$}(2_2);
    \draw[lightgray] (2_1)--node[midway, left]{\scriptsize $+1$}(3_1);
    \draw (2_1)--node[midway, right]{\scriptsize $-1$}(3_2);
    \draw[lightgray] (2_2)--node[midway, left]{\scriptsize $+1$}(3_3);
    \draw[lightgray] (2_2)--node[midway, right]{\scriptsize $-1$}(3_4);
    \draw[lightgray] (3_2)--node[midway, left]{\scriptsize $+1$}(4_3);
    \draw (3_2)--node[midway, right]{\scriptsize $-1$}(4_4);
    \draw (4_4)--node[midway, left]{\scriptsize $+1$}(5_2);
    \draw[lightgray] (4_4)--node[midway, right]{\scriptsize $-1$}(5_3);
    \draw (5_2)--node[midway, left]{\scriptsize $+1$}(6_2);
    \draw (5_2)--node[midway, right]{\scriptsize $-1$}(6_3);
    \draw[lightgray] (5_3)--node[midway, fill=white]{$\vdots$}(6_4);
    \draw[lightgray] (6_1)--node[midway, left]{\scriptsize $+1$}(7_1);
    \draw[lightgray] (6_1)--node[midway, right]{\scriptsize $-1$}(7_2);
    \draw (6_2)--node[midway, left]{\scriptsize $+1$}(7_3);
    \draw (6_2)--node[midway, right]{\scriptsize $-1$}(7_4);
    \draw (6_3)--node[midway, left]{\scriptsize $+1$}(7_5);
    \draw (6_3)--node[midway, right]{\scriptsize $-1$}(7_6);
    \draw[lightgray] (6_6)--node[midway, left]{\scriptsize $+1$}(7_7);
    \draw[lightgray] (6_6)--node[midway, right]{\scriptsize $-1$}(7_8);
    \draw[lightgray] (7_1)--node[midway, left]{\scriptsize $+1$}(8_1);
    \draw[lightgray] (7_1)--node[midway, right]{\scriptsize $\pm1$}(8_2);
    \draw[lightgray] (7_2)--node[midway, right]{\scriptsize $-1$}(8_4);
    \draw (7_3)--node[midway, left]{\scriptsize $+1$}(8_5);
    \draw (7_3)--node[midway, right]{\scriptsize $\pm1$}(8_6);
    \draw (7_4)--node[midway, right]{\scriptsize $-1$}(8_8);
    \draw (7_5)--node[midway, left]{\scriptsize $+1$}(8_9);
    \draw (7_5)--node[midway, right]{\scriptsize $\pm1$}(8_10);
    \draw (7_6)--node[midway, right]{\scriptsize $-1$}(8_12);
    \draw[lightgray] (7_7)--node[midway, left]{\scriptsize $+1$}(8_13);
    \draw[lightgray] (7_7)--node[midway, right]{\scriptsize $\pm1$}(8_14);
    \draw[lightgray] (7_8)--node[midway, right]{\scriptsize $-1$}(8_16);

    \draw[black!100!black,thick,decorate,decoration={brace,amplitude=10pt,mirror}] ([yshift=-0.2 cm] 8_5.south west) --  node[below=12pt] {Considered sub-tree} ([yshift=-0.2 cm] 8_12.south east);
\end{tikzpicture}
\caption{Going down $f$ steps into the search tree and thus only considering assignments $x\in\{-1,1\}^n$ on one of the sub-trees branching out at $x_{f+1}$.}
\label{figfixintree}
\end{figure}

\subsection{Mixed {SAT} problems}
As an almost immediate consequence of Theorem~\ref{mainthm}, we get the following characterization of asymptotic satisfiability in the random mixed $1$- and $k$-SAT problem when $k=2,3$.

\begin{theorem}
\label{mainthm2}
    Let $\varphi$ be a random $k$-CNF formula with $m=m(n)$ clauses and $n$ variables, where $m/n\to\alpha$, and let $\lambda$ be a random $1$-CNF formula with $f=f(n)$ clauses and $n$ variables, where $f/\sqrt{n}\to\beta$ for some $0\leq\beta\leq\infty$, such that $\varphi$ and $\lambda$ are independent.
    \begin{enumerate}
        \item[{$k=2:$}] If $0\leq\alpha<1$, then $\lim_{n\to\infty}\prob(\varphi\land\lambda\in\SAT)=e^{-(\beta/2)^2(1-\alpha)^{-1}}$.
        \item[{$k=3:$}] If $0\leq\alpha<3.145$, then $\lim_{n\to\infty}\prob(\varphi\land\lambda\in\SAT)=e^{-(\beta/2)^2}$.
    \end{enumerate}
\end{theorem}

In Knuth's book on Boolean satisfiability~\cite{knuth15}, he writes: ``unit clauses aren’t rare: Far from it. Experience shows that they’re almost ubiquitous in practice, so that the actual search trees often involve only dozens of branch nodes instead of thousands or millions.'' In Theorem~\ref{mainthm2}, we analyze the effect of adding these ubiquitous unit clauses to the random $2$-SAT and $3$-SAT problems.

For the random $2$-SAT problem, the ``missing'' factor between the threshold functions in Theorem~\ref{mainthm2} and Theorem~\ref{mainthm} is exactly the probability that the random unit clauses comprising~$\lambda$ cause a contradiction, i.e.\ the probability that there exists some $v\in[n]$ such that both~$x_v$ and~$-x_v$ are clauses in~$\lambda$. If we condition on this event not occurring, then we reclaim the original limit $e^{-(\beta/2)^2\alpha (1-\alpha)^{-1}}$ from Theorem~\ref{mainthm} (see Remark~\ref{rem1}). And indeed, taking $\alpha=0$ in Theorem~\ref{mainthm2} gives the following result: if $\lambda$ is a random $1$-CNF formula with $n$ variables and $f$~clauses where $f/\sqrt{n}\to\beta$, then
\begin{equation}
\label{1satproblem}
    \lim_{n\to\infty}\prob(\lambda\in\SAT)=e^{-(\beta/2)^2}.
\end{equation}
This limit is again an analytic function in~$\beta$ (a Gaussian, even), and thus the random $1$-SAT problem does not exhibit a phase transition/sharp threshold in satisfiability. We also see that the relevant parameter to consider in this problem is not the usual clause density~$\alpha$, but rather~$\beta$, the ratio of clauses to the \emph{square root} of the number of variables. That $\alpha$ is the wrong parameter is no surprise; Chv\'{a}tal and Reed had naturally realized this and excluded the random $1$-SAT problem from Conjecture~\eqref{satconj}, and Gent and Walsh even outright state in~\cite{GW94} that a random $1$-CNF formula with $n$ variables and clause density~$\alpha>0$ is always unsatisfiable with high probability. Despite this, we believe that this paper is the first to identify the correct parameter~$\beta$ and with it completely characterize the random mixed $1$- and $2$-SAT problem in Theorem~\ref{mainthm2}.

Theorem~\ref{mainthm2} shows that the random mixed $1$- and $2$-SAT problem behaves like the random $1$-SAT problem, in that the limiting probability of satisfiability does not have a sharp threshold, but instead varies smoothly in the parameters~$\alpha$ and~$\beta$, and even the slightest change to either of these affects this limit. The addition of $\beta\sqrt{n}$~random unit clauses has thus ``smoothed out'' the phase transition of the random $2$-SAT problem seen in~\eqref{satconj}.

Another property of the random mixed $1$- and $2$-SAT problem is that we are able to ``exchange'' unit clauses for binary clauses---or the other way around---without changing the limiting probability of the mixed formula being satisfiable. Indeed, if $\varphi$ is again a random CNF~formula with $\beta\sqrt{n}$ unit clauses and $\alpha n$ binary clauses, and we want to exchange, say, half of the unit clauses, so that $\varphi$ is left with $(\beta/2)\sqrt{n}$ of these, then to compensate we need to add $3n$ random binary clauses and then remove three-quarters of this new total, so that $\varphi$ ends up with $((3+\alpha)/4)n$ binary clauses. Doing this exchange, the limiting probability that $\varphi$~is satisfiable will remain the same. More generally, if we want to exchange from~$\varphi$ to a random mixed CNF-formula $\varphi^\prime$ with $\beta^\prime\sqrt{n}$ unit clauses and $\alpha^\prime n$ binary clauses, then $\varphi$ and $\varphi^\prime$ have the same asymptotic satisfiability as long as
\begin{equation}
    \biggl(\frac{\beta^\prime}{\beta}\biggr)^2=\frac{1-\alpha^\prime}{1-\alpha}.
\end{equation}

This is a very distinct situation from the random mixed $1$- and $3$-SAT problem, detailed in the second part of Theorem~\ref{mainthm2}. Indeed, we see from Theorem~\ref{mainthm2} that one can add up to $3.145n$ random $3$-clauses to \emph{any} random $1$-CNF formula without changing the asymptotic probability of satisfiability (compare Theorem~\ref{mainthm2} case $k=3$ with~\eqref{1satproblem}). Hence, the ``exchange rate'' between $1$- and $3$-clauses is infinite, since we are able to add these $<3.145n$ random $3$-clauses ``for free''.

In regards to exchange rate, the random mixed $1$- and $3$-SAT problem shows more similarity with the random mixed $2$- and $3$-SAT problem, which was first studied in the physics literature~\cite{MZKST96,MZ98,MZKST99b,BMW00} and later given a rigorous treatment in~\cite{AKKK01} (slightly improved upon in~\cite{ZG13}). Here it was established that one can add at least $(2/3)n$ random $3$-clauses, but no more than $2.17n$, to any under-constrained random $2$-CNF formula and have the resulting formula remain under-constrained. Again, the exchange rate between $2$- and $3$-clauses is infinite, and the smaller clauses completely dominate the problem. That is, the asymptotic satisfiability is determined completely by the random $2$-CNF sub-formula, just as we saw that the asymptotic satisfiability of a random mixed $1$- and $3$-CNF formula is determined by the random $1$-CNF sub-formula (see Chapter~10 of~\cite{handbook} for a further discussion on exchange rate).

Returning finally to the random mixed $1$- and $2$-SAT problem again, we note that this problem is not only interesting in its own right, but it is also sometimes useful for proofs. Indeed, in the paper~\cite{achlioptas00}, Achlioptas uses the following result: if $\varphi$ is a random CNF formula with $n$ variables, $\alpha n$ binary clauses, and any polynomial in $\log(n)$ unit clauses, then $\varphi$~is under-constrained in the usual region $\alpha<1$. Our Theorem~\ref{mainthm2} is obviously a direct strengthening of this, since we show that one can take~$n^q$ for any $q<1/2$ (actually anything in~$o(\sqrt{n})$) in place of the polynomial in $\log(n)$ unit clauses and still get the same result. Theorem~\ref{mainthm2} also shows that this is \emph{optimal}, in the sense that adding $n^q$ independent random unit clauses to an under-constrained random $2$-CNF formula spoils the asymptotic satisfiability when $q\geq 1/2$.

\subsection{The threshold functions}
For $k=2,3$, $\alpha>0$ and $0\leq\beta<\infty$, let $\varphi$ be a random $k$-CNF formula with $m\sim\alpha n$ clauses and $n$ variables. We recall that $\varphi$ has $n^{1-1/k}$ degress of freedom. Let $\mathcal L\subseteq\pm[n]$ be consistent with $\lvert\mathcal L\rvert=f=f(n)$, where $f/n^{1-1/k}\to\beta$. We denote the threshold function from~\eqref{defthreshfunc} by
\begin{equation*}
    \pi_k(\alpha,\beta)\coloneqq\lim_{n\to\infty}\prob(\varphi_\mathcal L\in\SAT),
\end{equation*}
when this limit exists (and only depends on~$\alpha$ and~$\beta$), and define $\pi_k(\alpha,\beta)\coloneqq 0$ otherwise. From Theorem~\ref{mainthm}, we have that $\pi_2(\alpha,\beta)=e^{-(\beta/2)^2\alpha (1-\alpha)^{-1}}$ when $0<\alpha<1$ and $\pi_3(\alpha,\beta)=e^{-(\beta/2)^3\alpha}$ when $0<\alpha<3.145$. The graphs of $\pi_2$ and $\pi_3$ are plottet in Figure~\ref{figthreshfunc}.

\begin{figure}[htb]
    \centering
    \includegraphics[width=0.9\linewidth]{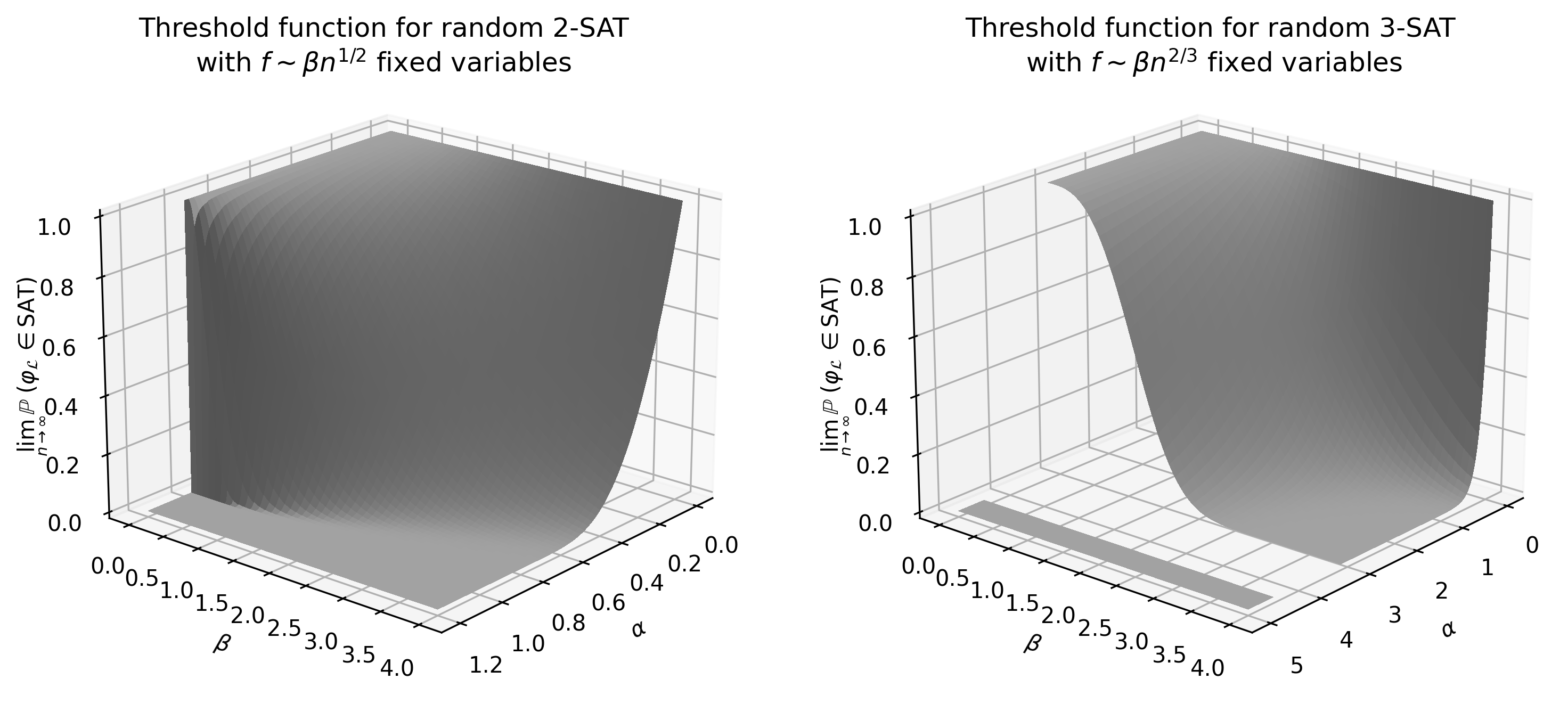}
    \caption{The threshold functions corresponding to the degrees of freedom in the random $2$-SAT and $3$-SAT problems.}
    \label{figthreshfunc}
\end{figure}

We define $\alpha_\textup{supp}(k)\coloneqq\inf\{\alpha\colon\pi_k(\alpha,\beta)=0\}$, the right endpoint of $\supp(\pi_k(\cdot,\beta))$. Now, if~$\varphi$ is unsatisfiable w.h.p., then $\pi_k(\alpha,\beta)=0$ for all $\beta\geq 0$, so we get $\alpha_\textup{supp}(3)\leq 4.4898$ from~\cite{DKMP09}, as this is the current best upper bound on~$\alpha_\textup{c}(3)$. Therefore, the value of the threshold function~$\pi_3(\alpha,\beta)$ is known for $\alpha<3.145$ (from our Theorem~\ref{mainthm}) and for $\alpha>4.4898$, which explains the ``gap'' in the graph in Figure~\ref{figthreshfunc} r.h.s. For the random $2$-SAT problem, we have $\alpha_\textup{supp}(2)=\alpha_\textup{c}(2)=1$ from~\cite{CR92} and Theorem~\ref{mainthm}.

What happens in this gap $\alpha\in[3.145,4.4898]$ remains an open problem with several potential scenarios possible. For the random $2$-SAT problem, we see that fixing $f\sim\beta\sqrt{n}$ variables ($\beta>0$) has ``smoothed out'' the sharp satisfiability threshold at $\alpha_\textup{c}(2)=1$ (this sharp threshold is visible at $\beta=0$). That is, the threshold function $\pi_2(\alpha,\beta)$ is continuous, and even $C^\infty$ (smooth), in~$\alpha$ on the entire domain $\alpha>0$, since $e^{-(\beta/2)^2\alpha (1-\alpha)^{-1}}$, and its derivatives of all orders in~$\alpha$, tend towards~$0$ as $\alpha\to 1$. On the other hand, $e^{-(\beta/2)^3\alpha}$ does not tend towards~$0$ as $\alpha\to\alpha_\textup{supp}(3)$ (whatever $\alpha_\textup{supp}(3)$ might be). From these observations and Theorem~\ref{mainthm}, we immediately get the following conclusion.

\begin{corollary}
\label{maincor}
    Let $\beta>0$ be given. For the random $2$-SAT problem, the threshold function $\pi_2(\alpha,\beta)$ is analytic for $0<\alpha<\alpha_\textup{supp}(2)=1$ and continuous (smooth, even) for all $\alpha>0$. For the random $3$-SAT problem, the threshold function $\pi_3(\alpha,\beta)$ is analytic for $0<\alpha<3.145$, but it \emph{cannot} be both analytic for $0<\alpha<\alpha_\textup{supp}(3)$ and continuous for all $\alpha>0$.
\end{corollary}
\begin{proof}
Because of the identity theorem for analytic functions, if $\pi_3(\alpha,\beta)$ is real analytic for $0<\alpha<\alpha_\textup{supp}(3)$, then necessarily $\pi_3(\alpha,\beta)=e^{-(\beta/2)^3\alpha}$ for all $0<\alpha<\alpha_\textup{supp}(3)$, which precludes continuity at~$\alpha_\textup{supp}(3)$.
\end{proof}

Corollary~\ref{maincor} shows a significant distinction between the random $2$-SAT and $3$-SAT problems. While the threshold function for random $2$-SAT is regular---continuous and analytic on the interior of its support---the threshold function for random $3$-SAT is non-regular. Consequently, the random $3$-SAT problem is particularly sensitive to variations in the clause density~$\alpha$ compared to the random $2$-SAT problem.

\subsection{Structure}
The structure of the remainder of the paper is as follows: in Section~\ref{secsketchofproof} we give a sketch of the proof of Theorem~\ref{mainthm}, where we explain the main steps and difficulties. In Section~\ref{secpreliminaries} we review the necessary preliminaries for a complete proof of Theorems~\ref{mainthm} and~\ref{mainthm2}, and in Section~\ref{sectheproof} we give the full proof, starting with a number of smaller results required for later.

\newpage
\section{Sketch of proof}
\label{secsketchofproof}
We begin by outlining a proof of Theorem~\ref{mainthm} in the case $k=2$ when $\alpha>0$. The proof relies on the following two special cases: namely~\eqref{1satproblem}, the random $1$-SAT problem, a special case of Theorem~\ref{mainthm2}, and Theorem~\ref{mainthm} in the ``subcritical'' case where $\beta=0$. The latter says that fixing $o(\sqrt{n})$ variables in an under-constrained random $2$-CNF formula still yields a w.h.p.\ satisfiable formula, and we prove this using a novel variation on the classic ``snakes and snares'' proof of $\alpha_\textup{c}(2)=1$ from~\cite{CR92,goerdt96,knuth15}. The former, i.e.~\eqref{1satproblem}, is proved through a direct counting argument.

The idea of the proof is now as follows: let $\varphi=C_1\land C_2\land\dots\land C_m$ be a random $2$-CNF formula with $n$ variables and $m=m(n)$ clauses, where $m/n\to\alpha$, $0<\alpha<1$, and let $\mathcal L\subseteq\pm[n]$ be consistent with $\lvert\mathcal L\rvert=f=f(n)$, where $f/\sqrt{n}\to\beta$, and we assume without loss of generality that $0<\beta<\infty$. Now, a critical observation is the following: each of the clauses~$C_j$ experiences one of four fates when the variables dictated by~$\mathcal L$ are fixed:
\begin{enumerate}
    \item[0.] it becomes a $0$-clause, i.e.\ is unsatisfied, and thus precludes the satisfiability of~$\varphi_\mathcal L$,
    \item it becomes a $1$-clause and thus dictates the value of a new variable,
    \item it remains unaltered as a $2$-clause,
    \item[$\star$.] it becomes a $\star$-clause, i.e.\ is satisfied, and thus no longer affects the satisfiability of~$\varphi_\mathcal L$.
\end{enumerate}
The core of the proof is understanding and carefully controlling how many clauses land in each of the four cases above, and especially important are cases 0 and 1. Denote by $\mathcal C^{(1)}_1$ the set of $j\in [m]$ for which $C_j$ falls into case 1. The unit clauses corresponding to $\mathcal C^{(1)}_1$ form a $1$-CNF formula $\lambda^{(1)}=\min\{ C_j\colon j\in\mathcal C^{(1)}_1\}$, which must be satisfiable for $\varphi_\mathcal L$ to be satisfiable, and furthermore, they fix the value of $\lvert\mathcal C^{(1)}_1\rvert$ new variables. Thus, the ``remaining'' $2$-CNF formula (i.e.\ the conjunction of the clauses in case 2), call it $\varphi^{(1)}$, undergoes the same ``splitting'' into $0$-, $1$-, $2$-, and $\star$-clauses when fixing the variables dictated by~$\lambda^{(1)}$.

Let $\mathcal C^{(r)}_k$ denote the set of $j\in [m]$ for which $C_j$ falls into case $k\in\{ 0,1,2,\star\}$ in the $r$'th round of this iterative process. We denote by $\lambda^{(r)}$ and $\varphi^{(r)}$ respectively the random $1$-CNF and random $2$-CNF formula generated in this round, and let $\mathcal L_r$ denote the set of literals forming $\lambda^{(r)}$. For each~$r$ we rely on the fact that $\varphi^{(r-1)}_{\mathcal L_{r-1}}$ (or $\varphi_\mathcal L$ for $r=1$) is satisfiable if and only if
\begin{equation*}
    \lvert\mathcal C^{(r)}_0\rvert=0,\quad\lambda^{(r)}\in\SAT,\quad\text{and}\quad\varphi^{(r)}_{\mathcal L_r}\in\SAT.
\end{equation*}
It turns out that these three events are almost independent (the error vanishes as $n\to\infty$), so
\begin{equation}
\label{loosedecomp}
    \prob\bigl(\varphi^{(r-1)}_{\mathcal L_{r-1}}\in\SAT\bigr)\approx\prob\bigl(\lvert\mathcal C^{(r)}_0\rvert=0\bigr)\cdot\prob\bigl(\lambda^{(r)}\in\SAT\bigr)\cdot\prob\bigl(\varphi^{(r)}_{\mathcal L_{r}}\in\SAT\bigr)
\end{equation}
for each $r\in\mathbb N$. Using this approach, the important issue becomes: how large are $\lvert\mathcal C^{(r)}_0\rvert$ and~$\lvert\mathcal C^{(r)}_1\rvert$? This is obviously heavily dependent on the previous ``rounds'', but on average, $\lvert\mathcal C^{(r)}_1\rvert$ is of the order $\alpha^r\beta\sqrt{n}$. A core part of the proof is showing that $\lvert\mathcal C^{(r)}_1\rvert$ concentrates around this mean. Now, we have suppressed the dependence on~$n$ in the notation, but the iterative decomposition~\eqref{loosedecomp} is done for each $n\in\mathbb N$. The next step of the proof is to \emph{diagonalize}, taking $r$ dependent on~$n$; specifically, when $r$ becomes larger than $\log(n)$, then $\alpha^r\beta\sqrt{n}$, and thus $\lvert\mathcal C^{(r)}_1\rvert$, becomes \emph{smaller} than $\beta n^{1/2+\log(\alpha)}$ (where of course $\log(\alpha)<0$), and at this point the problem shifts into the subcritical case where the number of variables being fixed is $o(\sqrt{n})$, which is one of the special cases which we already proved. Indeed, taking $R\coloneqq c\log(n)$, where $c>0$ is some appropriate constant, we use the decomposition~\eqref{loosedecomp} $R$~times successively to get
\begin{equation*}
    \prob(\varphi_\mathcal L\in\SAT)\approx\prob\bigl(\varphi^{(R)}_{\mathcal L_{R}}\in\SAT\bigr)\prod_{r=1}^R\prob\bigl(\lvert\mathcal C^{(r)}_0\rvert=0\bigr)\prod_{r=1}^R\prob\bigl(\lambda^{(r)}\in\SAT\bigr),
\end{equation*}
and the special case gives us $\prob(\varphi^{(R)}_{\mathcal L_{R}}\in\SAT)\to 1$, and the other special case~\eqref{1satproblem} gives us (because of the concentration phenomenon)
\begin{equation*}
    \prob\bigl(\lambda^{(r)}\in\SAT\bigr)\longrightarrow e^{-\tfrac{1}{4}\beta^2\alpha^{2r}}
\end{equation*}
for each $r\in\mathbb N$. Finally, we are able to show directly that
\begin{equation*}
    \prob\bigl(\lvert\mathcal C^{(r)}_0\rvert=0\bigr)\longrightarrow e^{-\tfrac{1}{4}\beta^2\alpha^{2r-1}}
\end{equation*}
for each $r\in\mathbb N$, which will then imply (with some additional arguments) that
\begin{equation*}
    \prob(\varphi_\mathcal L\in\SAT)\longrightarrow\exp\biggl(\frac{-\beta^2}{4}\sum_{r=1}^\infty\alpha^r\biggr)=\exp\biggl(\frac{-\beta^2\alpha}{4(1-\alpha)}\biggr).
\end{equation*}
We see that each even term in the infinite series comes from the probability that the appearing unit clauses cause a contradiction among themselves, and each odd term comes from the probability of the appearance of a $0$-clause.

Besides some lemmas and the special cases, the main technical difficulty in converting the idea above into a proof is the concentration of~$\lvert\mathcal C^{(r)}_k\rvert$, in particular for $k=1$. We will find bounds $f_r^-,f_r^+$ such that $\lvert\mathcal C^{(r)}_1\rvert$ stays within $[f_r^-,f_r^+]$ w.h.p., and such that both $f_r^-,f_r^+\sim\alpha^r\beta\sqrt{n}$. In place of~$\lambda^{(r)}$ we insert an independent random $1$-CNF formula with $f_r^-$ resp.\ $f_r^+$ clauses, which will lead to an asymptotic upper resp.\ lower bound on $\prob(\varphi_\mathcal L\in\SAT)$. Namely, the idea will be to choose
\begin{equation*}
    f_r^\pm\approx\alpha^r f\pm r\alpha^r n^{3/8},
\end{equation*}
which exactly does the job, as we will see.

Now, if $\alpha=0$ (and $\varphi$ is still a random $2$-CNF formula), then already in the first ``round'' $r=1$ in the decomposition~\eqref{loosedecomp}, the number of $1$-clauses generated $\lvert\mathcal L_1\rvert$ is $o(\sqrt{n})$, so the subcritical special case may immediately be invoked.

For the case $k=3$ in Theorem~\ref{mainthm}, the same strategy may be employed. The decomposition~\eqref{loosedecomp} still holds, although $\varphi^{(r)}$ is now a random mixed $2$- and $3$-CNF formula. Thus, there are more things to keep track of in this case, and we will not have a proof of the subcritical case handy, making the whole thing slightly more tedious. On the other hand, we show that w.h.p.\ there are no more unit clauses appearing after three rounds, so we avoid infinite series. For this reason, we are also able to prove Theorem~\ref{mainthm} case $k=3$ for all $0\leq\alpha<3.145$ instead of having $\alpha=0$ separate. The threshold function comes from the fact that
\begin{equation*}
    \prob\bigl(\lvert\mathcal C^{(1)}_0\rvert=0\bigr)\longrightarrow e^{-(\gamma/2)^3},
\end{equation*}
which comes from a direct calculation. Thus, it remains to show that the other terms of the decomposition~\eqref{loosedecomp} tend towards $1$. This is again achieved via the concentration of the sets $\lvert\mathcal C^{(r)}_k\rvert$. We rely on existing literature to show that the final error term involving a random mixed $2$- and $3$-CNF formula vanishes.

\newpage
\section{Preliminaries}
\label{secpreliminaries}
In this section, we will review the necessary preliminary results and notations for Boolean functions $h:\{-1,1\}^n\to\{-1,1\}$ of $n\in\mathbb N$ variables needed for the formulation and proof of this papers results.

\subsection{Boolean functions}
For any $A\subseteq\mathbb Z$ we define:
\begin{equation*}
    A_\textup{abs}\coloneqq\{\lvert a\rvert\colon a\in A\},\quad-A\coloneqq\{-a\colon a\in A\},\quad\pm A\coloneqq A\cup(-A),
\end{equation*}
and $\lvert A\rvert$ for the number of elements in $A$. For $n\in\mathbb N$ we use the notation
\begin{equation*}
    [n]\coloneqq\{ 1,2,\dots,n\}.
\end{equation*}

We let $\SAT$ denote the set of all functions $h:B\to\{-1,1\}$ with the element~$1$ contained in the range of~$h$, where $B$ is a non-empty subset of $\{-1,1\}^n$ for any~$n$, called the \emph{satisfiable} functions. If $x\in B$ is such that $h(x)=1$, then $h$ is said to be \emph{satisfied} at~$x$. The collection of such~$x$ is called the \emph{solution space} for~$h$ and is denoted
\begin{equation*}
    \mathsf{SOL}(h)\coloneqq h^{-1}(\{ 1\})=\{ x\in B\colon h(x)=1\}\subseteq\{-1,1\}^n.
\end{equation*}

If $\mathcal B$ is a set of Boolean functions of $n$ variables, then we will say that $\mathcal B$ is \emph{consistent} if it is possible to satisfy all functions in $\mathcal B$ simultaneously, i.e.\ if there exists an $x\in\{-1,1\}^n$ such that $h(x)=1$ for all $h\in\mathcal B$. If we define a new Boolean function $\min(\mathcal B):\{-1,1\}^n\to\{-1,1\}$ by
\begin{equation}
\label{literalstocnf}
    \min(\mathcal B)(x)\coloneqq\min_{h\in\mathcal B}h(x),\quad (x\in\{-1,1\}^n),
\end{equation}
then $\mathcal B$ is consistent if and only if $\min(\mathcal B)\in\SAT$.

\subsection{Literals and fixing variables}
A function $l:\{-1,1\}^n\to\{-1,1\}$ is called a \emph{literal} if, for some $v\in[n]$,
\begin{equation*}
    l(x)=x_v\text{ for all $x$},\quad\text{or}\quad l(x)=-x_v\text{ for all $x$}.
\end{equation*}
Each literal (of $n$ variables) corresponds to an element of $\pm[n]$ in the following way: the literal $x\mapsto x_v$ corresponds to~$v$ and $x\mapsto-x_v$ corresponds to~$-v$. This correspondence is bijective since taking an element $l\in\pm[n]$ to the literal $x\mapsto\sgn(l)x_{\lvert l\rvert}$ is the inverse operation. Thus, the space of literals in $n$ variables is identified with the set $\pm[n]$, and we shall use these two sets interchangeably in the following. So considering e.g.\ the number~$-5$ as a literal we have $-5(x)=-x_5$.

Notice that a literal~$l$ is satisfied at~$x$ (i.e.\ $l(x)=1$) exactly when $x_{\lvert l\rvert}=\sgn(l)$. It follows for a set of literals $\mathcal L\subseteq\pm[n]$ that, for all $x\in\{-1,1\}^n$,
\begin{equation*}
    l(x)=1\text{ for all }l\in\mathcal L\iff x_{\lvert l\rvert}=\sgn(l)\text{ for all }l\in\mathcal L,
\end{equation*}
and hence that $\mathcal L$ is consistent if and only if either $v\notin\mathcal L$ or $-v\notin\mathcal L$ for every $v\in[n]$, or equivalently if at most one of~$v$ and~$-v$ is a member of $\mathcal L$ for every $v\in[n]$. We see from this that for a consistent set of literals $\mathcal L$, satisfying $\min(\mathcal L)$ (at~$x$) amounts to fixing certain of the variables~$x_v$; more precisely fixing $x_{\lvert l\rvert}$ to~$\sgn(l)$ for every~$l\in\mathcal L$. Motivated by this we define for every consistent set of literals $\mathcal L\subseteq\pm[n]$ and every $v\in[n]$ the set
\begin{equation}
\label{Bdefn}
    B_v(\mathcal L)\coloneqq
    \begin{cases}
        \{ 1\},&\text{if $v\in\mathcal L$,} \\
        \{-1\},&\text{if $-v\in\mathcal L$,} \\
        \{-1,1\},&\text{if $v,-v\notin\mathcal L$,}
    \end{cases}
    \quad\text{and we define}\quad B(\mathcal L)\coloneqq\bigtimes_{v=1}^n B_v(\mathcal L).
\end{equation}
Then we exactly get $\mathsf{SOL}(\min(\mathcal L))=B(\mathcal L)$ by the preceding discussion (this identity will also hold for contradicting $\mathcal L$ if we extend the definition of $B_v(\mathcal L)$ to~$\emptyset$ when both $v,-v\in\mathcal L$). This again implies that for any Boolean function $h:\{-1,1\}^n\to\{-1,1\}$ it holds that
\begin{equation*}
    h\land\min(\mathcal L)\in\SAT\iff h\vert_{B(\mathcal L)}\in\SAT,
\end{equation*}
where $h\vert_{B(\mathcal L)}$ is the restriction of~$h$ to the set $B(\mathcal L)\subseteq\{-1,1\}^n$. Motivated by this, we define for any Boolean function $h:\{-1,1\}^n\to\{-1,1\}$ and any consistent set of literals $\mathcal L\subseteq\pm[n]$,
\begin{equation}
\label{fixvardefn}
    h_\mathcal L\coloneqq h\vert_{B(\mathcal L)}.
\end{equation}
Thus, $h_\mathcal L$ is the function~$h$ but with the variables dictated by $\mathcal L$ fixed. More generally we get for $\mathcal L\subseteq\pm[n]$ (not necessarily consistent) that
\begin{equation}
\label{splitoffunitclauses}
h\land\min(\mathcal L)\in\SAT\iff\min(\mathcal L)\in\SAT\quad\text{and}\quad h_\mathcal L\in\SAT.
\end{equation}

\begin{example}
    If $h:\{-1,1\}^n\to\{-1,1\}$ ($n\geq 5$ for this example) is again any Boolean function, and $\mathcal L=\{ 1,-2,4\}$, we see that $\mathcal L$ is dictating $x_1=1$, $x_2=-1$, and $x_4=1$, i.e.\
    \begin{equation*}
        B(\mathcal L)=\{ 1\}\times\{-1\}\times\{-1,1\}\times\{ 1\}\times\{-1,1\}^{n-4},
    \end{equation*}
    giving us
    \begin{equation*}
        h_\mathcal L(x)=h(1,-1,x_3,1,x_5,\dots,x_n),\quad (x\in B(\mathcal L)).
    \end{equation*}
\end{example}

\subsection{Clauses and CNF formulas}
A function $C:\{-1,1\}^n\to\{-1,1\}$ is called a \emph{(disjunctive) clause} if there exists $k\in\mathbb N$ and literals $l_1,l_2,\dots,l_k\in\pm[n]$ such that
\begin{equation*}
    C(x)=l_1(x)\lor l_2(x)\lor\dots\lor l_k(x)=\max_{i\in[k]}l_i(x),\quad (x\in\{-1,1\}^n).
\end{equation*}
The minimum such~$k$ is called the \emph{length} of~$C$, and we call~$C$ a $k$-clause. A $1$-clause is just a literal and is also called a \emph{unit clause}. Further, by convention we call the constant function $C\equiv-1$ a $0$-clause, i.e.\ a clause that is always unsatisfied, and the constant function $C\equiv 1$ a $\star$-clause, i.e.\ a clause that is always satisfied.

If \(C_1,C_2,\dots,C_m\) are clauses, then the function $\varphi:\{-1,1\}^n\to\{-1,1\}$ defined by
\begin{equation}
\label{cnfdef}
\varphi(x)=C_1(x)\land C_2(x)\land\dots\land C_m(x)=\min_{j\in[m]}C_j(x),\quad (x\in\{-1,1\}^n),
\end{equation}
is said to be a \emph{CNF formula}, and the form in~\eqref{cnfdef} is called a \emph{CNF representation} for~$\varphi$ (CNF is short for \emph{conjunctive normal form}). We allow $m=0$ which by convention gives $\varphi\equiv1$. If each of the clauses $C_1,C_2,\dots,C_m$ has length~$k$ (ignoring any $\star$-clauses), then $\varphi$ is called a $k$-CNF formula, and~\eqref{cnfdef} is called a $k$-CNF representation.

A $1$-CNF formula is just a conjunction/minimum of literals, and from any set of literals $\mathcal L$ it is of course possible to define a $1$-CNF formula~$\lambda$ by taking $\lambda=\min(\mathcal L)$ as defined in~\eqref{literalstocnf}. This operation sends all sets with contradicting literals to the same $1$-CNF formula (the constant function $x\mapsto -1$), but when restricted to the consistent sets of literals, it is one-to-one. Indeed, in this case $\mathsf{SOL}(\lambda)=B(\mathcal L)$, and $B(\mathcal L)$ determines~$\mathcal L$. Thus, for any \emph{satisfiable} $1$-CNF formula~$\lambda$ with corresponding set of literals~$\mathcal L$ (i.e.\ $\lambda=\min(\mathcal L)$) we may and do write~$h_\lambda$ instead of~$h_\mathcal L$ for any $h:\{-1,1\}^n\to\{-1,1\}$.

\subsection{Fixing variables in CNF formulas}
Consider first a $2$-CNF formula
\begin{equation*}
    \varphi=C_1\land C_2\land\dots\land C_m,
\end{equation*}
each clause given by the disjunction of two literals: $C_j=l_{j,1}\lor l_{j,2}$, where $l_{j,i}\in\pm[n]$ for all $j\in[m]$ and $i=1,2$, and consider in addition a consistent set of literals $\mathcal L\subseteq\pm[n]$. In the formula~$\varphi_\mathcal L$, each literal $l_{j,i}$ has a possibility of being fixed, and this happens exactly when $l_{j,i}\in\pm\mathcal L$. More precisely, if $l_{j,i}\in\mathcal L$, then $(l_{j,i})_\mathcal L\equiv 1$, and if $l_{j,i}\in -\mathcal L$, then $(l_{j,i})_\mathcal L\equiv-1$, and lastly if $l_{j,i}\notin\pm\mathcal L$, then $(l_{j,i})_\mathcal L=l_{j,i}$. This yields the following cases for $(C_j)_\mathcal L$ for each $j\in[m]$:
\begin{enumerate}
    \item[0.] If both literals get fixed to~$-1$, i.e.\ $l_{j,1},l_{j,2}\in-\mathcal L$, then $(C_j)_\mathcal L\equiv-1$ can never be satisfied.
    \item If one literal gets fixed to~$-1$, say $l_{j,1}\in-\mathcal L$, but the other is not fixed, i.e.\ $l_{j,2}\notin\pm\mathcal L$, then $(C_j)_\mathcal L=-1\lor l_{j,2}=l_{j,2}$, so $(C_j)_\mathcal L$ is the unit clause $l_{j,2}$. Similarly if $l_{j,2}\in-\mathcal L$ and $l_{j,1}\notin\pm\mathcal L$, then $(C_j)_\mathcal L=l_{j,1}$.
    \item If no literal gets fixed, i.e.\ both $l_{j,1},l_{j,2}\notin\pm\mathcal L$, then $(C_j)_\mathcal L=C_j$.
    \item[$\star$.] If one or both literals gets fixed to~$1$, i.e.\ $l_{j,1}\in\mathcal L$ or $l_{j,2}\in\mathcal L$, then $(C_j)_\mathcal L\equiv1$ is always satisfied.
\end{enumerate}
We see that each of the clauses $(C_j)_\mathcal L$ either becomes unsatisfiable (case~0), is transformed to a unit clause (case~1), remains intact (case~2), or is immediately satisfied (case~$\star$). Let for $k\in\{ 0,1,2,\star\}$
\begin{equation}
\label{defnmk}
    \mathcal C_k\coloneqq\{ j\in[m]\colon\text{$j$ is in case $k$}\},
\end{equation}
where we suppress the dependence on the~$l_{j,i}$'s and~$\mathcal L$ in the notation. That is, $\mathcal C_k$ consists of exactly those $j\in[m]$ where $C_j$ becomes a $k$-clause when fixing the variables dictated by~$\mathcal L$, i.e.\
\begin{equation*}
    j\in\mathcal C_k\iff\text{$(C_j)_\mathcal L$ is a $k$-clause}.
\end{equation*}
We can now split~$\varphi_\mathcal L$ into parts by defining for each $k\in\{ 0,1,2,\star\}$:
\begin{equation}
\label{defnpsik}
    \varphi_k\coloneqq\min_{j\in\mathcal C_k}(C_j)_\mathcal L,
\end{equation}
where $\min\emptyset=1$, so that $\varphi_k$ is a $k$-CNF formula, and $\varphi_\mathcal L=\varphi_0\land\varphi_1\land\varphi_2$. Using the latter and~\eqref{splitoffunitclauses}, we find that
\begin{equation}
\label{phiLsatiff}
    \varphi_\mathcal L\in\SAT\iff\lvert\mathcal C_0\rvert=0,\quad\varphi_1\in\SAT,\quad\text{and}\quad(\varphi_2)_{\varphi_1}\in\SAT.
\end{equation}
This method of ``splitting'' $\varphi_\mathcal L$ into its $0$-, $1$-, and $2$-CNF sub-formulas $\varphi_0,\varphi_1,\varphi_2$, along with the corresponding sets $\mathcal C_0,\mathcal C_1,\mathcal C_2$, will play a major role in the proof of the main theorem of this paper.

Consider now a $3$-CNF formula $\varphi=C_1\land\dots\land C_m$. By completely analogous considerations to the above, we see that each clause falls into one of of the following five cases when fixing the variables dictated by $\mathcal L$: becomes unsatisfiable (case~0), is transformed to a unit clause (case~1), is transformed to a binary clause (case~2), remains intact (case~3), or is immediately satisfied (case~$\star$). Again we may define $\mathcal C_k$ and $\varphi_k$ for $k\in\{ 0,1,2,3,\star\}$ as in~\eqref{defnmk} and~\eqref{defnpsik}, where then $\varphi_\mathcal L=\varphi_0\land\varphi_1\land\varphi_2\land\varphi_3$, and we get that
\begin{equation}
\label{phiLsatiff3}
    \varphi_\mathcal L\in\SAT\iff\lvert\mathcal C_0\rvert=0,\quad\varphi_1\in\SAT,\quad\text{and}\quad(\varphi_2\land\varphi_3)_{\varphi_1}\in\SAT.
\end{equation}

\subsection{Random CNF formulas}
A random $k$-CNF formula $\varphi$ with $m$~clauses and $n$~variables is obtained by sampling $m$ i.i.d.\ random $k$-clauses $C_1,\dots,C_m$ and taking their conjunction:
\begin{equation*}
    \varphi\coloneqq C_1\land\dots\land C_m=\min_{j\in[m]}C_j.
\end{equation*}
Each of the $C_j$ is obtained by sampling $V_j=(V_{j,1},\dots,V_{j,k})$ uniformly at random among all sequences of $k$ \emph{pairwise distinct} variables from $[n]$, and then sampling fair random signs (i.e.\ equal to $-1$ or $1$, each with probability $1/2$) $S_j=(S_{j,1},\dots,S_{j,k})$ independently of~$V_j$. This gives random literals $L_{j,i}\coloneqq S_{j,i}V_{j,i}$ and the random $k$-clause
\begin{equation*}
    C_j\coloneqq L_{j,1}\lor\dots\lor L_{j,k}=\max_{i\in [k]}L_{j,i}.
\end{equation*}
If the lengths of the clauses $C_1,\dots,C_j$ are not all chosen to be equal, we say that $\varphi$ is a random \emph{mixed} CNF formula.

\newpage
\section{The proof}
\label{sectheproof}

\subsection{Lemmas}
In this section, we establish a few technical lemmas needed for the proof of Theorem~\ref{mainthm}. We will first show that for a random $k$-CNF formula~$\varphi$ and consistent set of literals $\mathcal L\subseteq\pm[n]$, the probability that $\varphi_\mathcal L\in\SAT$ only depends on $\mathcal L$ through $\lvert\mathcal L\rvert$, the number of variables being fixed, and in a non-increasing way.

\begin{lemma}
\label{probdecreaseinc}
    Let $\varphi$ be a random (possibly mixed) CNF formula with $m$~clauses and $n$~variables, and let $\mathcal L,\mathcal L^\prime\subseteq\pm[n]$ be two consistent sets of literals.
    \begin{enumerate}[label=(\roman*)]
        \item If $\lvert\mathcal L\rvert=\lvert\mathcal L^\prime\rvert$, then $\prob(\varphi_\mathcal L\in\SAT)=\prob(\varphi_{\mathcal L^\prime}\in\SAT)$.
        \item If $\lvert\mathcal L\rvert\geq\lvert\mathcal L^\prime\rvert$, then $\prob(\varphi_\mathcal L\in\SAT)\leq\prob(\varphi_{\mathcal L^\prime}\in\SAT)$.
    \end{enumerate}
\end{lemma}
\begin{proof}
    Let $L_{j,i}=S_{j,i}V_{j,i}$ for $j\in[m]$ and $i\in[k_j]$ be the literals defining~$\varphi$. Assume that $\lvert\mathcal L\rvert=\lvert\mathcal L^\prime\rvert$ and let $\pi:[n]\to[n]$ be a permutation such that $\pi(\mathcal L_\textup{abs})=\mathcal L^\prime_\textup{abs}$. Define $\theta,\theta^\prime:[n]\to\{-1,1\}^n$ by putting
    \begin{equation*}
        \theta(v)\coloneqq
        \begin{cases}
            -1,&\text{if $-v\in\mathcal L$,} \\
            1,&\text{otherwise},
        \end{cases}
        \quad\text{and similarly}\quad\theta^\prime(v)\coloneqq
        \begin{cases}
            -1,&\text{if $-v\in\mathcal L^\prime$,} \\
            1,&\text{otherwise}.
        \end{cases}
    \end{equation*}
    This definition is such that if $l\in\mathcal L$ then $\theta(\lvert l\rvert)=\sgn(l)$ and similarly for~$\theta^\prime$. Now define
    \begin{equation*}
        V^\prime_{j,i}\coloneqq\pi(V_{j,i}),\quad S^\prime_{j,i}\coloneqq S_{j,i}\theta(V_{j,i})\theta^\prime(V^\prime_{j,i}),\quad\text{and}\quad L^\prime_{j,i}\coloneqq S^\prime_{j,i}V^\prime_{j,i}
    \end{equation*}
    for all $j\in[m]$ and $i\in[k]$. Then by direct calculation we find that the random vectors $(L^\prime_{j,i})_{j,i}$ and $(L_{j,i})_{j,i}$ have the same distribution, so $\varphi$ and $\varphi^\prime$ have the same distribution, where $\varphi^\prime$ is the random CNF formula defined from the random literals $L^\prime_{j,i}$ for $j\in[m]$ and $i\in[k_j]$. This implies
    \begin{equation*}
        \prob(\varphi_{\mathcal L^\prime}\in\SAT)=\prob(\varphi^\prime_{\mathcal L^\prime}\in\SAT).
    \end{equation*}
    We now show that
    \begin{equation*}
        \varphi^\prime_{\mathcal L^\prime}\in\SAT\iff\varphi_\mathcal L\in\SAT.
    \end{equation*}
    This will follow from the fact that $(L_{j,i})_\mathcal L(x)=(L^\prime_{j,i})_{\mathcal L^\prime}(x^\prime)$ for all $j\in[m]$, $i\in[k]$, and $x\in\{-1,1\}^n$, where $x^\prime\in\{-1,1\}^n$ is given by $x^\prime_v=x_{\pi^{-1}(v)}$. There are two cases:
    \begin{enumerate}[label=(\arabic*)]
        \item If $V_{j,i}\notin\mathcal L_\textup{abs}$, then $V^\prime_{j,i}\notin\mathcal L^\prime_\textup{abs}$, and in this case $\theta(V_{j,i})=\theta^\prime(V^\prime_{j,i})=1$, yielding
        \begin{equation*}
            (L^\prime_{j,i})_{\mathcal L^\prime}(x^\prime)=S^\prime_{j,i}x^\prime_{V^\prime_{j,i}}=S_{j,i}x_{\pi^{-1}(V^\prime_{j,i})}=S_{j,i}x_{V_{j,i}}=(L_{j,i})_\mathcal L(x).
        \end{equation*}
        \item If $V_{j,i}\in\mathcal L_\textup{abs}$, then $V^\prime_{j,i}\in\mathcal L^\prime_\textup{abs}$, and by construction we get $B_{V_{j,i}}(\mathcal L)=\{\theta(V_{j,i})\}$ and $B_{V^\prime_{j,i}}(\mathcal L^\prime)=\{\theta^\prime(V^\prime_{j,i})\}$ (see~\eqref{Bdefn}). This yields (see~\eqref{fixvardefn})
        \begin{equation*}
            (L^\prime_{j,i})_{\mathcal L^\prime}(x^\prime)=S^\prime_{j,i}\theta^\prime(V^\prime_{j,i})=S_{j,i}\theta(V_{j,i})\theta^\prime(V^\prime_{j,i})^2=S_{j,i}\theta(V_{j,i})=(L_{j,i})_\mathcal L(x).
        \end{equation*}
    \end{enumerate}
    All together we find that
    \begin{equation*}
        \prob(\varphi_{\mathcal L^\prime}\in\SAT)=\prob(\varphi^\prime_{\mathcal L^\prime}\in\SAT)=\prob(\varphi_\mathcal L\in\SAT),
    \end{equation*}
    proving (i).
    
    Now for (ii), if $\lvert\mathcal L\rvert\geq\lvert\mathcal L^\prime\rvert$, we consider an injection $\iota:\mathcal L^\prime\to\mathcal L$. Then it is clear that $\varphi_\mathcal L\in\SAT$ implies $\varphi_{\iota(\mathcal L^\prime)}\in\SAT$ (since $\iota(\mathcal L^\prime)\subseteq\mathcal L$), and the result follows from (i) since $\lvert\iota(\mathcal L^\prime)\rvert=\lvert\mathcal L^\prime\rvert$.
\end{proof}

The next lemma shows that adding more clauses to~$\varphi$ only decreases the probability of satisfiability for~$\varphi_\mathcal L$.

\begin{lemma}
\label{probdecreaseinm}
    Let $\varphi$ and $\varphi^\prime$ be random (possibly mixed) CNF formulas with $m_k$ resp.\ $m_k^\prime$ $k$-clauses for each~$k$ and $n$~variables, and let $\mathcal L\subseteq\pm[n]$ be a consistent set of literals. If $m_k\geq m^\prime_k$ for each $k$, then
    \begin{equation*}
        \prob(\varphi_\mathcal L\in\SAT)\leq\prob(\varphi^\prime_\mathcal L\in\SAT),
    \end{equation*}
    and in particular $\prob(\varphi\in\SAT)\leq\prob(\varphi^\prime\in\SAT)$.
\end{lemma}
\begin{proof}
    The result follows from a straightforward coupling argument; let $C^{(k)}_j$ be independent random $k$-clauses for $j,k\in\mathbb N$ (also independent across different values of~$k$), and define
    \begin{equation*}
        \psi\coloneqq\min_{k\in\mathbb N}\min_{j\in[m_k]}C^{(k)}_j,\quad\text{and}\quad\psi^\prime\coloneqq\min_{k\in\mathbb N}\min_{j\in[m^\prime_k]}C^{(k)}_j.
    \end{equation*}
    Then $\psi$ has the same distribution as~$\varphi$ and $\psi^\prime$ has the same distribution as~$\varphi^\prime$, but it holds by construction that, for every $x\in\{-1,1\}^n$, if $\psi(x)=1$, then $\psi^\prime(x)=1$, and thus also $\psi_\mathcal L(x)=1$ implies $\psi^\prime_\mathcal L(x)=1$. This yields
    \begin{equation*}
        \prob(\varphi_\mathcal L\in\SAT)=\prob(\psi_\mathcal L\in\SAT)\leq\prob(\psi^\prime_\mathcal L\in\SAT)=\prob(\varphi^\prime_\mathcal L\in\SAT).
    \end{equation*}
    The second assertion follows by taking $\mathcal L=\emptyset$.
\end{proof}

This next lemma describes the distribution of $(\varphi_0,\varphi_1,\varphi_2)$ and $(\varphi_0,\varphi_1,\varphi_2,\varphi_3)$, the $0$-, $1$-, $2$-CNF (and $3$-CNF) subformulas of $\varphi_\mathcal L$ defined in~\eqref{defnpsik}, when $\varphi$ is a random $2$- or $3$-CNF formula.

\begin{lemma}
\label{distofphiandm}
    Let $k^\prime=2$ or $k^\prime=3$ and let $K=\{ 0,1,\dots,k^\prime,\star\}$. For $m,n\in\mathbb N$ and $n^\prime\in\mathbb N_0$ with $2\leq n^\prime<n$, let $\varphi$ be a random $k^\prime$-CNF formula with $m$ clauses and $n$ variables, and let $\mathcal L=[n]\setminus[n-n^\prime]$. Let $(\mathcal C_k)_{k\in K}$ and $(\varphi_k)_{k\in K}$ be as defined in~\eqref{defnmk} and~\eqref{defnpsik}, so in particular $\varphi_\mathcal L=\varphi_0\land\dots\land\varphi_{k^\prime}$. Let furthermore $M=(M_k)_{k\in K}$ be a partition of~$[m]$.
    
    Then $(\lvert\mathcal C_k\rvert)_{k\in K}\sim\mathrm{Multinomial}(m,(p_k)_{k\in K})$, and in the conditional distribution given $(\mathcal C_k)_{k\in K}=M$, it holds that the $\varphi_k$'s are independent and that $\varphi_k$ is a random $k$-CNF formula with $\lvert M_k\rvert$ clauses and $n-n^\prime$ variables for all $k\in K$.
    
    Furthermore, if $k^\prime=2$, then
    \begin{equation*}
        p_0=\frac{n^\prime(n^\prime-1)}{4n(n-1)},\quad p_1=\frac{n^\prime (n-n^\prime)}{n(n-1)},\quad p_2=\frac{(n-n^\prime)(n-n^\prime-1)}{n(n-1)},
    \end{equation*}
    and if $k^\prime=3$, then
    \begin{align*}
        p_0&=\frac{n^\prime(n^\prime-1)(n^\prime-2)}{8n(n-1)(n-2)},& p_1&=\frac{3n^\prime (n^\prime-1)(n-n^\prime)}{4n(n-1)(n-2)}, \\
        p_2&=\frac{3n^\prime (n-n^\prime)(n-n^\prime-1)}{2n(n-1)(n-2)},& p_3&=\frac{(n-n^\prime)(n-n^\prime-1)(n-n^\prime-2)}{n(n-1)(n-2)}.
    \end{align*}
\end{lemma}
\begin{proof}
    Assume initially that $k^\prime=2$ and let $L=(L_{j,i})_{j\in[m],i\in[2]}$ be the random literals defining~$\varphi$. Notice that $[n]\setminus\mathcal L=[n-n^\prime]$, and define the sets
    \begin{gather*}
        A_0\coloneqq-\mathcal L\times-\mathcal L,\quad A_1\coloneqq(-\mathcal L\times\pm[n-n^\prime])\cup(\pm[n-n^\prime]\times-\mathcal L), \\
        A_2\coloneqq\pm[n-n^\prime]\times\pm[n-n^\prime],\quad A_\star\coloneqq(\mathcal L\times\pm[n])\cup(\pm[n]\times\mathcal L).
    \end{gather*}
    Notice for each $k\in K$ that
    \begin{equation*}
        \mathcal C_k=\big\{\, j\in[m]\colon (L_{j,1},L_{j,2})\in A_k\,\big\},
    \end{equation*}
    where the sets $\mathcal C_k$ are those defined in~\eqref{defnmk}. Since $(A_k)_{k\in K}$ is a partition of $\pm[n]\times\pm[n]$, it is clear that $(\lvert\mathcal C_k\rvert)_{k\in K}$ follows a multinomial distribution. Now,
    \begin{equation*}
        p_k=\frac{1}{m}\mean\bigl[\lvert\mathcal C_k\rvert\bigr]=\frac{1}{m}\mean\biggl[\sum_{j=1}^m 1_{\{ (L_{j,1},L_{j,2})\in A_k\}}\biggr]=\prob\bigl((L_{1,1},L_{1,2})\in A_k\bigr).
    \end{equation*}
    From the known distribution of the random literals $(L_{1,1},L_{1,2})$, we find by direct calculations the values of $p_0,p_1,p_2$ stipulated.
    
    Put $\mathcal C\coloneqq(\mathcal C_k)_{k\in K}$, and let $M=(M_k)_{k\in K}$ be a partition of $[m]$. For any $l_{j,i}\in\pm[n]$, $j\in[m]$ and $i=1,2$, we immediately find by independence of the $(L_{j,1},L_{j,2})$ pairs that
    \begin{equation*}
        \prob\bigl(L=(l_{j,i})_{j\in[m],i\in[2]}\bigm\vert\mathcal C=M\bigr)=\prod_{k\in K}\prod_{j\in M_k}\prob\bigl((L_{j,1},L_{j,2})=(l_{j,1},l_{j,2})\bigm\vert (L_{j,1},L_{j,2})\in A_k\bigr),
    \end{equation*}
    so the pairs $(L_{j,1},L_{j,2})$ for $j\in[m]$ are conditionally independent given $\mathcal C=M$. We further find by direct calculation that for $j\in M_2$ and $(l_{j,1},l_{j,2})\in A_2$ with $\lvert l_{j,1}\rvert\neq\lvert l_{j,2}\rvert$,
    \begin{equation*}
        \prob\bigl((L_{j,1},L_{j,2})=(l_{j,1},l_{j,2})\bigm\vert (L_{j,1},L_{j,2})\in A_2\bigr)=\frac{1}{4n(n-1)p_2}=\frac{1}{4(n-n^\prime)(n-n^\prime-1)},
    \end{equation*}
    showing that $(L_{j,1},L_{j,2})$ is uniformly distributed over all pairs of strictly distinct literals from $\pm[n-n^\prime]$ given $\mathcal C=M$. Also, $(L_{j,1},L_{j,2})\in A_2$ is ``case 2'' described above~\eqref{defnmk}, so $(L_{j,i})_\mathcal L=L_{j,i}$ for $i=1,2$ almost surely given $\mathcal C=M$. Recalling that
    \begin{equation*}
        \varphi_k(x)=\min_{j\in\mathcal C_k}\bigl((L_{j,1})_\mathcal L(x)\lor (L_{j,2})_\mathcal L(x)\bigr),\quad (x\in\{-1,1\}^n),
    \end{equation*}
    for $k=0,1,2$, it follows that in the conditional distribution given $\mathcal C=M$, $\varphi_2$ is indeed a random $2$-CNF formula (defined from the random literals $(L_{j,i})_{j\in M_k,i\in[2]}$). An analogous argument shows that given $\mathcal C=M$, $\varphi_1$ is a random $1$-CNF formula. By independence of the $(L_{j,1},L_{j,2})$ pairs in the conditional distribution, $\varphi_0$, $\varphi_1$ and $\varphi_2$ are independent given $\mathcal C=M$.

    Completely analogous considerations establish the result for $k^\prime=3$.
\end{proof}

The final lemma in this section concerns the amount of duplicate literals in a random $1$-CNF formula. We show that there are loosely speaking very few duplicates if the number of clauses/literals is proportional to the square root of the number of variables.

\begin{lemma}
\label{fewduplicateliterals}
    Let $n_0=n_0(n)$ and $f=f(n)$ be such that $f/\sqrt{n_0}\to\beta$, where $0\leq\beta<\infty$. Consider i.i.d.\ random literals $L_1,L_2,\dots,L_f$ uniformly distributed on $\pm[n_0]$. Put $\Lambda\coloneqq\{ L_1,\dots,L_f\}$. Then
    \begin{equation*}
        \prob\bigl(\lvert\Lambda\rvert<f-w\bigr)\leq\frac{4\beta^2}{w}
    \end{equation*}
    for all $w=w(n)$ when $n$ is large enough.
\end{lemma}
\begin{proof}
    Let $V_j\coloneqq\lvert L_j\rvert$ for each $j\in[f]$, so that $\Lambda_\textup{abs}=\{ V_1,\dots,V_f\}$. Define for each $v\in[n_0]$
    \begin{equation*}
        N_v\coloneqq\bigl\lvert\bigl\{ j\in[f]\colon V_j=v\bigr\}\bigr\rvert.
    \end{equation*}
    Then $(N_1,\dots,N_{n_0})\sim\mathrm{Multinomial}(f,(1/n_0,\dots,1/n_0))$, and we see that
    \begin{equation*}
        \lvert\Lambda_\textup{abs}\rvert=f-\sum_{v\in[n_0]}(N_v-1)1_{\{ N_v\geq 2\}}.
    \end{equation*}
    We split the sum as follows:
    \begin{equation*}
        \sum_{v\in[n_0]}(N_v-1)1_{\{ N_v\geq 2\}}=\sum_{v\in[n_0]}(N_v-1)1_{\{ N_v\geq 3\}}+\sum_{v\in[n_0]}1_{\{ N_v=2\}}.
    \end{equation*}
    Cauchy-Schwarz yields
    \begin{equation*}
        \mean[(N_v-1)1_{\{ N_v\geq 3\}}]\leq\mean[N_v 1_{\{ N_v\geq 3\}}]\leq\mean[N_v^2]^{1/2}\prob(N_v\geq 3)^{1/2},
    \end{equation*}
    where
    \begin{equation*}
        \mean[N_v^2]=\frac{f(f-1)}{n_0^2}+\frac{f}{n_0}\leq\frac{f^2}{n_0^2}+\frac{f}{n_0}\leq 2\frac{f}{n_0},
    \end{equation*}
    where we in the end use that $f/n_0\to 0$ as $n\to\infty$, so $f/n_0\leq 1$ when $n$ is large enough, and thus $\mean[N_v^2]^{1/2}\leq\sqrt{2f}/\sqrt{n_0}$ when $n$ is large enough. For the second factor we notice that $N_v\geq 3$ if and only if there exists $j_1<j_2<j_3$ such that $V_{j_1}=V_{j_2}=V_{j_3}=v$, so that
    \begin{equation*}
        \prob(N_v\geq 3)=\sum_{j_1<j_2<j_3}\prob(V_{j_1}=V_{j_2}=V_{j_3}=v)=\binom{f}{3}\frac{1}{n_0^3}\leq\frac{f^3}{n_0^3},
    \end{equation*}
    since the $V_1,\dots,V_f$ are i.i.d.\ uniformly distributed on $[n_0]$. We note for later that, in particular,
    \begin{equation}
    \label{drawatmosttwo}
        \prob(N_v\geq 3\text{ for some $v\in[n_0]$})\leq\sum_{v\in[n_0]}\prob(N_v\geq 3)\leq\frac{f^3}{n_0^2}=\biggl(\frac{f}{\sqrt{n_0}}\biggr)^3\frac{1}{\sqrt{n_0}}\longrightarrow 0
    \end{equation}
    as $n\to\infty$, so that $N_v\leq 2$ for all $v\in[n_0]$ w.h.p. Together this shows that
    \begin{equation*}
        \sum_{v\in[n_0]}\mean[(N_v-1)1_{\{ N_v\geq 3\}}]\leq\sqrt{2}\frac{f^2}{n_0}=\sqrt{2}\biggl(\frac{f}{\sqrt{n_0}}\biggr)^2\longrightarrow\sqrt{2}\beta^2\quad\text{as $n\to\infty$},
    \end{equation*}
    so this sum is smaller than~$2\beta^2$ when $n$ is large enough. For the other sum we note that, as before, $N_v\geq 2$ if and only if there exists $j_1<j_2$ such that $V_{j_1}=V_{j_2}=v$, so by a similar argument,
    \begin{equation*}
        \mean\biggl[\sum_{v\in[n_0]}1_{\{ N_v=2\}}\biggr]=\sum_{v\in[n_0]}\prob(N_v\geq 2)\leq\frac{f^2}{n_0}=\biggl(\frac{f}{\sqrt{n_0}}\biggr)^2\longrightarrow\beta^2\quad\text{as $n\to\infty$}
    \end{equation*}
    so this sum is also smaller than~$2\beta^2$ when $n$ is large enough. Taken together:
    \begin{equation*}
        \mean\biggl[\sum_{v\in[n_0]}(N_v-1)1_{\{ N_v\geq 2\}}\biggr]\leq 4\beta^2
    \end{equation*}
    when $n$ is large enough, and thus by Markov's inequality:
    \begin{equation*}
        \prob\bigl(\lvert\Lambda_\textup{abs}\rvert<f-w\bigr)=\prob\biggl(\sum_{v\in[n_0]}(N_v-1)1_{\{ N_v\geq 2\}}>w\biggr)\leq\frac{4\beta^2}{w}
    \end{equation*}
    when $n$ is large enough. As $\lvert\Lambda\rvert\geq\lvert\Lambda_\textup{abs}\rvert$, we are done.
\end{proof}

\subsection{Special cases}
We need to prove two special cases of our main result, which we will need for the proof itself. Now, we will actually prove our main result in a more general form than stated in Theorem~\ref{mainthm}. Namely, we let our random $2$- or $3$-CNF formula $\varphi$ have any number of variables $n_0=n_0(n)$, as long as $n_0(n)\to\infty$ when $n\to\infty$, as opposed to $\varphi$ having exactly $n$ variables. The assumption on the number of clauses~$m$ then becomes $m/n_0\to\alpha$ and the number of variables fixed~$f$ becomes $f\sqrt{m}/n_0\to\gamma$ or $fm^{1/3}/n_0\to\gamma$ respectively. We shall for convenience also employ this small increase in generality for the results of this section (as we have already done in Lemma~\ref{fewduplicateliterals}), but it does not in any way change the proofs (it simply amounts to exchanging some $n$'s for~$n_0$'s), and for this reason we shall abuse notation by writing $n\coloneqq n_0$.

The characterization of the random $1$-SAT problem~\eqref{1satproblem} is the first special case which shall actually play an important role in the proof of Theorem~\ref{mainthm}.

\begin{lemma}[The random $1$-SAT problem]
\label{random1sat}
    For $0\leq\beta\leq\infty$, let $\lambda$ be a random $1$-CNF formula with $f=f(n)$ clauses and $n_0=n_0(n)$ variables, where $n_0\to\infty$ and $f/\sqrt{n_0}\to\beta$. Then
    \begin{equation*}
        \lim_{n\to\infty}\prob(\lambda\in\SAT)=e^{-(\beta/2)^2}.
    \end{equation*}
\end{lemma}
\begin{proof}
    Assume initially that $\beta<\infty$. We abuse notation by writing $n\coloneqq n_0$. Let the random literals corresponding to $\lambda$ be $L_1,L_2,\dots,L_f$, where $S_j=\sgn(L_j)$ and $V_j=\lvert L_j\rvert$. For each $v\in[n]$ we put
    \begin{equation*}
        N_v\coloneqq\bigl\lvert\bigl\{ j\in[f]\colon V_j=v\bigr\}\bigr\rvert.
    \end{equation*}
    Then $N\coloneqq (N_1,\dots,N_n)\sim\mathrm{Multinomial}(f,(1/n,\dots,1/n))$. From~\eqref{drawatmosttwo} we have that $N_v\leq 2$ for all $v\in[n]$ w.h.p., so when writing
    \begin{equation*}
        \prob(\lambda\in\SAT)=\prob(\lambda\in\SAT,N_v\geq 3\text{ for some $v\in[n]$})+\prob(\lambda\in\SAT,N_v\leq 2\text{ for all $v\in[n]$}),
    \end{equation*}
    we see that the first term vanishes as $n\to\infty$, and in regards to the second term we have:
    \begin{equation*}
        \{ N_v\leq 2\text{ for all $v\in[n]$}\}=\bigcup_{\substack{\eta\in\{0,1,2\}^n\colon \\ \eta_1+\dots+\eta_n=f}}\{ N=\eta\}=\bigcup_{k=0}^{\lfloor f/2\rfloor}\bigcup_{\eta\in H_k}\{ N=\eta\},
    \end{equation*}
    where 
    \begin{equation*}
        H_k\coloneqq\bigl\{\,\eta\in\{0,1,2\}^n\colon\eta_1+\dots+\eta_n=f,\lvert\{ v\in[n]\colon\eta_v=2\}\rvert=k\,\bigr\},
    \end{equation*} 
    i.e.\ we have divided the union into sub-unions depending on how many entries of $\eta$ are equal to~$2$, so that $\eta\in H_k$ when there are $k$ entries of $\eta$ equal to~$2$. Since the events on the right-hand side are disjoint, we may write
    \begin{equation}
    \label{probphisatunit}
        \prob(\lambda\in\SAT,N_v\leq 2\text{ for all $v\in[n]$})=\sum_{k=0}^{\lfloor f/2\rfloor}\sum_{\eta\in H_k}\prob(\lambda\in\SAT,N=\eta).
    \end{equation}
    Notice that $\lvert H_k\rvert=\binom{n-k}{f-2k}\binom{n}{k}$, corresponding to firstly choosing $k$ entries of $\eta$ to equal~$2$ and then choosing $f-2k$ of the remaining $n-k$ entries to equal~$1$ (the remaining entries will equal~$0$).
    
    Now, for fixed $k\in\{ 0,1,\dots,\lfloor f/2\rfloor\}$ and $\eta\in H_k$ we find by definition of $N$ that $N=\eta$ if and only if $(V_1,\dots,V_f)=\underline v$ for some $\underline v\in[n]^f$ satisfying $\lvert\{ j\in[f]\colon\underline v_j=v\}\rvert=\eta_v$ for all $v\in[n]$. Hence,
    \begin{equation}
    \label{decompphisatx}
        \prob(\lambda\in\SAT,N=\eta)=\sum_{\substack{\underline v\in[n]^f\colon \\ \lvert\{ j\in[f]\colon\underline v_j=v\}\rvert=\eta_v \\ \text{for all $v\in[n]$}}}\prob(\lambda\in\SAT,V=\underline v),
    \end{equation}
    where $V\coloneqq (V_1,\dots,V_f)$. Notice now that $\lambda\in\SAT$ if and only if there are no contradictions in the literals $L_1,\dots,L_f$, i.e.\ we don't have $L_{j_1}=-L_{j_2}$ for any $j_1<j_2$, which is to say that $\lambda\in\SAT$ if and only if $V_{j_1}=V_{j_2}$ implies $S_{j_1}=S_{j_2}$. But since $V$ and $S\coloneqq (S_1,\dots,S_f)$ are independent, we find that:
    \begin{equation*}
        \prob(\lambda\in\SAT,V=\underline v)=\prob(\text{$S_{j_1}=S_{j_2}$ for all $j_1<j_2$ with $\underline v_{j_1}=\underline v_{j_2}$})\prob(V=\underline v)=\frac{1}{2^k}\prob(V=\underline v),
    \end{equation*}
    since there are exactly $k$ pairs $(j_1,j_2)$ satisfying $j_1<j_2$ and $\underline v_{j_1}=\underline v_{j_2}$ by choice of~$\eta$. Inserting this into~\eqref{decompphisatx} and using the known distribution of~$N$ together with the fact that exactly $k$ entries of $\eta$ are~$2$ and the rest are $1$ or~$0$, we get:
    \begin{equation*}
        \prob(\lambda\in\SAT,N=\eta)=\frac{1}{2^k}\prob(N=\eta)=\frac{f!}{4^k n^f},
    \end{equation*}
    and inserting this further into~\eqref{probphisatunit} yields:
    \begin{align}
        \MoveEqLeft\prob(\lambda\in\SAT,N_v\leq 2\text{ for all $v\in[n]$})=\sum_{k=0}^{\lfloor f/2\rfloor}\binom{n-k}{f-2k}\binom{n}{k}\frac{f!}{4^k n^f}\label{needforlater1sat} \\
        &=\sum_{k=0}^\infty\biggl(\prod_{d=0}^{2k-1}\frac{f-d}{\sqrt{n}}\biggr)\biggl(\prod_{d=0}^{f-k-1}\frac{n-d}{n}\biggr)\frac{1}{4^k k!}1_{[0,\lfloor f/2\rfloor]}(k),\label{finalsum1sat}
    \end{align}
    where we in the second equality simply have rearranged the factors in each term. For fixed~$k$ we see that
    \begin{equation*}
        \lim_{n\to\infty}\prod_{d=0}^{2k-1}\frac{f-d}{\sqrt{n}}=\prod_{d=0}^{2k-1}\lim_{n\to\infty}\biggl(\frac{f}{\sqrt{n}}-\frac{d}{\sqrt{n}}\biggr)=\beta^{2k},
    \end{equation*}
    and also for large enough~$n$ (independent of $k$) we have $f/\sqrt{n}\leq\beta+1$, so:
    \begin{equation}
    \label{boundfirstprod}
        \prod_{d=0}^{2k-1}\frac{f-d}{\sqrt{n}}\leq (\beta+1)^{2k}.
    \end{equation}
    In regards to the other product we use that $\log(1+t)/t\to 1$ as $t\to 0$, so that for any $0<\epsilon<1$ we have $(1-\epsilon)t\leq\log(1+t)\leq(1+\epsilon)t$ whenever $t$ is close enough to $0$. This yields
    \begin{equation*}
        (1-\epsilon)\sum_{d=0}^{f-k-1}\frac{-d}{n}\leq\sum_{d=0}^{f-k-1}\log\bigl(1-\tfrac{d}{n}\bigr)\leq(1+\epsilon)\sum_{d=0}^{f-k-1}\frac{-d}{n}
    \end{equation*}
    whenever $n$ is large enough, and since
    \begin{equation*}
        \sum_{d=0}^{f-k-1}\frac{-d}{n}=\frac{-1}{2}\cdot\frac{f-k-1}{\sqrt{n}}\cdot\frac{f-k}{\sqrt{n}}\longrightarrow\frac{-\beta^2}{2}\quad\text{as $n\to\infty$},
    \end{equation*}
    we get by taking $\epsilon\to 0$ that
    \begin{equation*}
        \lim_{n\to\infty}\prod_{d=0}^{f-k-1}\frac{n-d}{n}=\lim_{n\to\infty}\exp\biggl[\sum_{d=0}^{f-k-1}\log\bigl(1-\tfrac{d}{n}\bigr)\biggr]=e^{-\tfrac{1}{2}\beta^2}.
    \end{equation*}
    Of course, we also have for all $n\in\mathbb N$
    \begin{equation}
    \label{boundsecondprod}
        \prod_{d=0}^{f-k-1}\frac{n-d}{n}\leq 1
    \end{equation}
    Thus, by~\eqref{boundfirstprod} and~\eqref{boundsecondprod}, the $k$'th term in~\eqref{finalsum1sat} is bounded by $(\beta+1)^{2k}/(4^k k!)$ when $n$ is large enough (independent of $k$), so by appealing to dominated convergence we get
    \begin{equation*}
        \lim_{n\to\infty}\prob(\lambda\in\SAT)=e^{-\tfrac{1}{2}\beta^2}\sum_{k=0}^\infty\frac{\beta^{2k}}{4^k k!}=e^{-\tfrac{1}{4}\beta^2},
    \end{equation*}
    as claimed.

    Now, if $\beta=\infty$, we consider for an arbitrary $T>0$ a $1$-CNF formula $\psi$ with $\lfloor T\sqrt{n}\rfloor$ clauses and $n$ variables. Since $f/\sqrt{n}\to\infty$ as $n\to\infty$, we have for large enough~$n$ that $f\geq\lfloor T\sqrt{n}\rfloor$, so for such large~$n$ we also have by Lemma~\ref{probdecreaseinm}
    \begin{equation*}
        \prob(\lambda\in\SAT)\leq\prob(\psi\in\SAT),
    \end{equation*}
    and $\prob(\psi\in\SAT)\to e^{-(T/2)^2}$ as $n\to\infty$ by what we have proved thus far. Taking then $T\to\infty$ yields
    \begin{equation*}
        \lim_{n\to\infty}\prob(\lambda\in\SAT)=0
    \end{equation*}
    as desired.
\end{proof}

The second special case which we need to prove separately is the subcritical case $\beta=0$ in Theorem~\ref{mainthm} case $k=2$. To prove this, we employ a new variation of the classical ``snakes and snares'' proof of the sharp satisfiability threshold in the random $2$-SAT problem. In this variation, we show that a $2$-CNF formula that has its satisfiability spoiled by fixing some amount of variables must contain a ``cobra'', which we define shortly. We then show that the mean number of cobras in a random $2$-CNF formula vanishes if we fix asymptotically fewer variables than the square root of the total number of variables, i.e.\ if $\beta=0$.

Let $m,n,N\in\mathbb N$ and $\mathcal L\subseteq\pm[n]$ be a set of literals. Then an \emph{$\mathcal L$-cobra} (of size $N$) is a sequence $l_0,l_1,\dots,l_N\in\pm[n]$ of literals such that
\begin{enumerate}[label=(c\arabic*)]
    \item\label{c1} The variables $\lvert l_0\rvert,\lvert l_1\rvert,\dots,\lvert l_{N-1}\rvert$ are pairwise distinct,
    \item\label{c2} $\lvert l_0\rvert\in\mathcal L_\textup{abs}$,
    \item\label{c3} $\lvert l_N\rvert\in\mathcal L_\textup{abs}\cup\{\lvert l_0\rvert,\lvert l_1\rvert,\dots,\lvert l_{N-1}\rvert\}$.
\end{enumerate}

Let $\varphi$ be a (non-random) $2$-CNF formula with $n$ variables and $m$ clauses of the form $\ell_{j,1}\lor\ell_{j,2}$, $j\in[m]$. We say that $\varphi$ \emph{contains} the sequence $(l_0,l_1,\dots,l_N)$ if there for all $t\in[N]$ exists some $j\in[m]$ such that $\{-l_{t-1},l_t\}=\{\ell_{j,1},\ell_{j,2}\}$.

\begin{lemma}
\label{alphaisnull}
    For $0\leq\alpha<1$, let $\varphi$ be a random $2$-CNF formula with $m=m(n)$ clauses and $n_0=n_0(n)$ variables, where $n_0\to\infty$ and $m/n_0\to\alpha$, and let $\mathcal L\subseteq\pm[n_0]$ be a consistent set of literals with $\lvert\mathcal L\rvert=f=f(n)$, where $f/\sqrt{n_0}\to 0$ as $n\to\infty$. Then it holds that
    \begin{equation*}
        \lim_{n\to\infty}\prob(\varphi_\mathcal L\in\SAT)=1.
    \end{equation*}
\end{lemma}
\begin{proof}
    As before we abuse notation by suppressing the dependence on $n$, and we then abbreviate $n$ for $n_0$.
    
    Let $(L_{j,i})_{j\in[m],i\in[2]}$ denote the random literals defining $\varphi$, and let $C_1,C_2,\dots,C_m$ denote the random clauses of $\varphi$, i.e.\ $C_j=L_{j,1}\lor L_{j,2}$, and let $\mathcal C^{(1)}_k$ denote the set defined in~\eqref{defnmk} for each $k\in\{ 0,1,2,\star\}$. Define the corresponding random CNF formulas (cf.~\eqref{defnpsik})
    \begin{equation*}
        \lambda^{(1)}\coloneqq\min_{j\in\mathcal C^{(1)}_1}(C_j)_\mathcal L,\quad\text{and}\quad\varphi^{(1)}\coloneqq\min_{j\in\mathcal C^{(1)}_2}(C_j)_\mathcal L=\min_{j\in\mathcal C^{(1)}_2}C_j.
    \end{equation*}
    Define also $\Lambda_1\coloneqq\bigl\{ (C_j)_\mathcal L\colon j\in\mathcal C^{(1)}_1\bigr\}$, the set of random literals in $\lambda^{(1)}$, i.e.\ $\lambda^{(1)}=\min(\Lambda_1)$. Then by~\eqref{phiLsatiff},
    \begin{equation*}
        \varphi_\mathcal L\in\SAT\iff\lvert\mathcal C^{(1)}_0\rvert=0,\quad\lambda^{(1)}\in\SAT,\quad\text{and}\quad\varphi^{(1)}_{\Lambda_1}\in\SAT.
    \end{equation*}
    We now repeat this with $\varphi^{(1)}_{\Lambda_1}$ in place of $\varphi_\mathcal L$, giving us for each $r\in\mathbb N$ sets $\mathcal C^{(r)}_k$ for $k\in\{ 0,1,2,\star\}$ stemming from $\varphi^{(r-1)}_{\Lambda_{r-1}}$, and random CNF formulas
    \begin{equation*}
        \lambda^{(r)}\coloneqq\min_{j\in\mathcal C^{(r)}_1}(C_j)_{\Lambda_{r-1}},\quad\text{and}\quad\varphi^{(r)}\coloneqq\min_{j\in\mathcal C^{(r)}_2}(C_j)_{\Lambda_{r-1}}=\min_{j\in\mathcal C^{(r)}_2}C_j,
    \end{equation*}
    and we finally also define $\Lambda_r\coloneqq\bigl\{ (C_j)_{\Lambda_{r-1}}\colon j\in\mathcal C^{(r)}_1\bigr\}$ so that $\lambda^{(r)}=\min(\Lambda_r)$. We here take $\varphi^{(0)}\coloneqq\varphi$ and $\Lambda^{(0)}\coloneqq\mathcal L$. Now, notice that
    \begin{equation*}
        \mathcal C^{(r)}_2=\bigcup_{k\in\{ 0,1,2,\star\}}\mathcal C^{(r+1)}_k
    \end{equation*}
    for each $r\in\mathbb N$, and of course the sets $\mathcal C^{(r)}_0,\mathcal C^{(r)}_1,\mathcal C^{(r)}_2,\mathcal C^{(r)}_\star$ are pairwise disjoint. In particular, $\mathcal C^{(1)}_2,\mathcal C^{(2)}_2,\mathcal C^{(3)}_2,\dots$ is a decreasing sequence of finite sets, and thus it must at some point become constant, i.e.\ there exists an $R\in\mathbb N$ such that $\mathcal C^{(r)}_2=\mathcal C^{(r+1)}_2$ for all $r\geq R$ (this happens precisely the first time $\mathcal C^{(r)}_1=\emptyset$, but this is irrelevant). In particular we get for $r\geq R$ that $\varphi^{(r)}=\varphi^{(r+1)}$ and $\mathcal C^{(r+1)}_k=\emptyset$ for $k\in\{ 0,1,\star\}$, and thus also $\Lambda_{r+1}=\emptyset$. We are thus able to define
    \begin{equation*}
        \varphi^{(\infty)}(x)\coloneqq\lim_{r\to\infty}\varphi^{(r)}(x)=\min\Bigl\{ C_j(x)\colon j\in\bigcap_{r\in\mathbb N}\mathcal C^{(r)}_2\Bigr\},\quad (x\in\{-1,1\}^n),
    \end{equation*}
    which we immediately see has the property $\varphi^{(\infty)}(x)\geq\varphi(x)$ for all $x$, so $\varphi\in\SAT$ implies $\varphi^{(\infty)}\in\SAT$. Further, we see that when $r>R$,
    \begin{equation*}
        \varphi^{(r)}_{\Lambda_r}=\varphi^{(\infty)}_\emptyset=\varphi^{(\infty)},
    \end{equation*}
    so using~\eqref{phiLsatiff} iteratively on $\varphi^{(r)}_{\Lambda_r}$ for each $r\in\mathbb N$ yields
    \begin{equation*}
        \varphi_\mathcal L\in\SAT\iff\text{$\lvert\mathcal C^{(r)}_0\rvert=0$ and $\lambda^{(r)}\in\SAT$ for all $r\in\mathbb N$},\quad\text{and}\quad\varphi^{(\infty)}\in\SAT.
    \end{equation*}
    Now define the event
    \begin{equation*}
        F\coloneqq\bigcup_{r\in\mathbb N}\Bigl(\bigl\{\lvert\mathcal C^{(r)}_0\rvert>0\bigr\}\cup\bigl\{\lambda^{(r)}\notin\SAT\bigr\}\Bigr),
    \end{equation*}
    so that, by the above,
    \begin{equation}
    \label{checkpointsubcrit}
        \prob(\varphi_\mathcal L\in\SAT)=\prob\bigl(F^c\cap\{\varphi^{(\infty)}\in\SAT\}\bigr).
    \end{equation}
    Since $\varphi$ is satisfiable w.h.p.~\cite{CR92,goerdt96}, so is~$\varphi^{(\infty)}$, and it thus only remains to prove that $F^c$ also occurs w.h.p., which is the same as proving that $\prob(F)\to 0$ as $n\to\infty$.

    We first prove that, if $F$ occurs, then $\varphi$ contains an $\mathcal L$-cobra. We introduce the ``flipping'' map $\rho:[2]\to[2]$ given by
    \begin{equation*}
        \rho(i)=3-i=
        \begin{cases}
            2,&\text{if $i=1$,} \\
            1,&\text{if $i=2$,}
        \end{cases}
        \quad (i\in[2]),
    \end{equation*}
    so that for any $j\in[m]$ and $i\in[2]$ we have $C_j=L_{j,i}\lor L_{j,\rho(i)}$. Assume then that $F$ occurs, meaning there exists an $r\in\mathbb N$ such that $\lvert\mathcal C^{(r)}_0\rvert>0$ or $\lambda^{(r)}\notin\SAT$. Let $N^\prime$ be the smallest such $r$.
    
    Assume first that $\lambda^{(N^\prime)}\notin\SAT$. This means that there exists two clauses $C_j$ and $C_{j^\prime}$, which at ``timepoint'' $N^\prime-1$ are still $2$-clauses, but when fixing the variables dictated by $\Lambda^{(N^\prime-1)}$, they each become $1$-clauses that contradict each other. That is to say, there exists $j(1),j^\prime(1)\in[m]$ and $i(1),i^\prime(1)\in[2]$ such that
    \begin{equation}
    \label{recurvbegin}
        L_{j(1),i(1)}=-L_{j^\prime(1),i^\prime(1)},\quad\text{and}\quad -L_{j(1),\rho(i(1))},-L_{j^\prime(1),\rho(i^\prime(1))}\in\Lambda^{(N^\prime-1)}.
    \end{equation}
    Now, $-L_{j(1),\rho(i(1))}$ being an element in $\Lambda^{(N^\prime-1)}$ means that there exists $j(2)\in[m]$ and $i(2)\in[2]$ such that
    \begin{equation*}
        -L_{j(1),\rho(i(1))}=L_{j(2),i(2)},\quad\text{and}\quad -L_{j(2),\rho(i(2))}\in\Lambda^{(N^\prime-2)}.
    \end{equation*}
    We continue this recursively, giving sequences $j(1),j(2),\dots,j(N^\prime)\in[m]$ and $i(1),i(2),\dots,i(N^\prime)\in[2]$ such that
    \begin{equation}
    \label{recurvL}
        -L_{j(t),\rho(i(t))}=L_{j(t+1),i(t+1)},\quad\text{and}\quad -L_{j(t+1),\rho(i(t+1))}\in\Lambda^{(N^\prime-(t+1))}
    \end{equation}
    for all $t\in[N^\prime-1]$. In the same way we also get sequences $j^\prime(1),j^\prime(2),\dots,j^\prime(N^\prime)\in[m]$ and $i^\prime(1),i^\prime(2),\dots,i^\prime(N^\prime)\in[2]$ such that
    \begin{equation}
    \label{recurvLprime}
        -L_{j^\prime(t),\rho(i^\prime(t))}=L_{j^\prime(t+1),i^\prime(t+1)},\quad\text{and}\quad -L_{j^\prime(t+1),\rho(i^\prime(t+1))}\in\Lambda^{(N^\prime-(t+1))}
    \end{equation}
    for all $t\in[N^\prime-1]$. We now define for each $t\in\{ 0,1,\dots,N^\prime-1\}$:
    \begin{equation*}
        l_t\coloneqq -L_{j(N^\prime-t),\rho(i(N^\prime-t))}.
    \end{equation*}
    Then $l_0=-L_{j(N^\prime),\rho(i(N^\prime))}\in\Lambda^{(0)}=\mathcal L$ by~\eqref{recurvL}, so condition~\ref{c2} in the definition of an $\mathcal L$-cobra is satisfied. Now define further for $t\in\{ 0,1,\dots,N^\prime-1\}$:
    \begin{equation*}
        l_{N^\prime+t}\coloneqq -L_{j^\prime(t+1),i^\prime(t+1)}.
    \end{equation*}
    We now verify that the sequence $(l_0,l_1,\dots,l_{2N^\prime-1})$ is ``contained'' in $\varphi$ in the sense defined above the formulation of Lemma~\ref{alphaisnull}. For $t\in[N^\prime-1]$ we have
    \begin{align*}
        (-l_{t-1},l_t)&=\bigl(L_{j(N^\prime-t+1),\rho(i(N^\prime-t+1))},-L_{j(N^\prime-t),\rho(i(N^\prime-t))}\bigr) \\
        &=\bigl(L_{j(N^\prime-t+1),\rho(i(N^\prime-t+1))},L_{j(N^\prime-t+1),i(N^\prime-t+1)}\bigr),
    \end{align*}
    using~\eqref{recurvL} in the second equality, as desired. From~\eqref{recurvbegin} we get
    \begin{equation*}
        (-l_{N^\prime-1},l_{N^\prime})=\bigl(L_{j(1),\rho(i(1))},-L_{j^\prime(1),i^\prime(1)}\bigr)=\bigl(L_{j(1),\rho(i(1))},L_{j(1),i(1)}\bigr),
    \end{equation*}
    and lastly for $t\in[N^\prime-1]$:
    \begin{align*}
        (-l_{N^\prime+t-1},l_{N^\prime+t})&=\bigl(L_{j^\prime(t),i^\prime(t)},-L_{j^\prime(t+1),i^\prime(t+1)}\bigr) \\
        &=\bigl(L_{j^\prime(t),i^\prime(t)},L_{j^\prime(t),\rho(i^\prime(t))}\bigr),
    \end{align*}
    using this time~\eqref{recurvLprime}. Hence, $(l_0,l_1,\dots,l_{2N^\prime-1})$ is contained in~$\varphi$. Now, if there exists a $t\in[2N^\prime-1]$ such that $\lvert l_t\rvert\in\{\lvert l_0\rvert,\lvert l_1\rvert,\dots,\lvert l_{t-1}\rvert\}$, then let $N$ be the smallest such~$t$. Then $(l_0,l_1,\dots,l_N)$ is clearly an $\mathcal L$-cobra contained in~$\varphi$ as desired. If not, then the variables $\lvert l_0\rvert,\lvert l_1\rvert,\dots,\lvert l_{2N^\prime-1}\rvert$ are pairwise distinct, and we take $N\coloneqq 2N^\prime$ and define finally
    \begin{equation*}
        l_N\coloneqq L_{j^\prime(N^\prime),\rho(i^\prime(N^\prime))}.
    \end{equation*}
    It follows immediately from~\eqref{recurvLprime} that $-l_N\in\Lambda^{(0)}=\mathcal L$, so of course $\lvert l_N\rvert\in\mathcal L_\textup{abs}$, and the sequence $(l_0,l_1,\dots,l_N)$ is thus an $\mathcal L$-cobra. It is also contained in~$\varphi$ since
    \begin{equation*}
        (-l_{N-1},l_N)=\bigl(L_{j^\prime(N^\prime),i^\prime(N^\prime)},L_{j^\prime(N^\prime),\rho(i^\prime(N^\prime))}\bigr).
    \end{equation*}
    This concludes the case where $\lambda^{(N^\prime)}\notin\SAT$.
    
    Assume now instead that $\lvert\mathcal C^{(N^\prime)}_0\rvert>0$. That is, there exists an element in $\mathcal C^{(N^\prime)}_0$, meaning that there exists a clause~$C_j$ that at ``timepoint'' $N^\prime-1$ is still a $2$-clause, but when fixing the variables dictated by $\Lambda^{(N^\prime-1)}$, it becomes a $0$-clause. More precisely, there exists a $j\in[m]$ such that
    \begin{equation*}
        -L_{j,1},-L_{j,2}\in\Lambda^{(N^\prime-1)}.
    \end{equation*}
    Taking $j(1)\coloneqq j^\prime(1)\coloneqq j$ and $i(1)\coloneqq 1$, $i^\prime(1)\coloneqq 2$, this is the second statement in~\eqref{recurvbegin}, so again we can find $j(2),j^\prime(2),\dots,j(N^\prime),j^\prime(N^\prime)\in[m]$ and $i(2),i^\prime(2),\dots,i(N^\prime),i^\prime(N^\prime)\in[2]$ such that~\eqref{recurvL} and~\eqref{recurvLprime} hold. We then define for $t\in\{ 0,1,\dots,N^\prime-1\}$:
    \begin{equation*}
        l_t\coloneqq -L_{j(N^\prime-t),\rho(i(N^\prime-t))},\quad\text{and}\quad l_{N^\prime+t}\coloneqq L_{j^\prime(t+1),\rho(i^\prime(t+1))},
    \end{equation*}
    where we see that the first of the two definitions coincides with what we had in the case $\lambda^{(N^\prime)}\notin\SAT$, so we need only verify that
    \begin{equation*}
        (-l_{N^\prime-1},l_{N^\prime})=\bigl(L_{j(1),\rho(i(1))},L_{j^\prime(1),\rho(i^\prime(1))}\bigr)=(L_{j,2},L_{j,1}),
    \end{equation*}
    and, using~\eqref{recurvLprime},
    \begin{align*}
        (-l_{N^\prime+t-1},l_{N^\prime+t})&=\bigl(-L_{j^\prime(t),\rho(i^\prime(t))},L_{j^\prime(t+1),\rho(i^\prime(t+1))}\bigr) \\
        &=\bigl(L_{j^\prime(t+1),i^\prime(t+1)},L_{j^\prime(t+1),\rho(i^\prime(t+1))}\bigr),
    \end{align*}
    for all $t\in[N^\prime-1]$, so that $(l_0,l_1,\dots,l_{2N^\prime-1})$ is indeed contained in~$\varphi$. As before we also have $\lvert l_0\rvert\in\mathcal L_\textup{abs}$, so if there exists a $t\in[2N^\prime-1]$ such that $\lvert l_t\rvert\in\{\lvert l_0\rvert,\lvert l_1\rvert,\dots,\lvert l_{t-1}\rvert\}$, then we take $N$ to be the smallest such~$t$, and $(l_0,l_1,\dots,l_N)$ is an $\mathcal L$-cobra contained in~$\varphi$. If there does not exist such a~$t$, then $(l_0,l_1,\dots,l_{2N^\prime-1})$ is does the trick; indeed, we are only missing a verification of~\ref{c3}, but
    \begin{equation*}
        -l_{2N^\prime-1}=-L_{j^\prime(N^\prime),\rho(i^\prime(N^\prime))}\in\Lambda^{(0)}=\mathcal L
    \end{equation*}
    by~\eqref{recurvLprime}, so indeed $\lvert l_{2N^\prime-1}\rvert\in\mathcal L_\textup{abs}$ as required.

    Now let $\mathfrak C_N$ denote the set of $\mathcal L$-cobras of size~$N$ for each $N\in\mathbb N$, and put
    \begin{equation*}
        Z\coloneqq\sum_{N\in\mathbb N}\sum_{l\in\mathfrak C_N}1_{\{\varphi\text{ contains }l\}},
    \end{equation*}
    i.e.\ $Z$ is the number of $\mathcal L$-cobras contained in~$\varphi$. We have just shown that $F\subseteq\{ Z>0\}$, so if we can only show that $\mean[Z]\to 0$ as $n\to\infty$, then the first moment method (Markov's inequality) asserts that
    \begin{equation*}
        \prob(F)\leq\prob(Z>0)\leq\mean[Z]\longrightarrow 0\quad\text{as $n\to\infty$},
    \end{equation*}
    thus finishing the proof of Lemma~\ref{alphaisnull}. To calculate the large $n$ limit of $\mean[Z]$, we first give an upper bound on $\lvert\mathfrak C_N\rvert$, i.e.\ the number of $\mathcal L$-cobras of size~$N$. Since $\mathcal L$ is consistent, $\lvert\mathcal L_\textup{abs}\rvert=\lvert\mathcal L\rvert=f$, and so there are $f$ possible choices for~$\lvert l_0\rvert$ by~\ref{c2}, and taking the sign into consideration gives $2f$ possible choices for~$l_0$. Then each of $l_1,l_2,\dots,l_{N-1}$ can be chosen freely as long as $\lvert l_0\rvert,\lvert l_1\rvert,\dots,\lvert l_{N-1}\rvert$ are pairwise distinct by~\ref{c1}, so choosing them one after the other gives $2(n-t)$ choices for~$l_t$, where $t\in[N-1]$. Finally, $\lvert l_N\rvert$ must be chosen from $\mathcal L_\textup{abs}\cup\{\lvert l_0\rvert,\lvert l_1\rvert,\dots,\lvert l_{N-1}\rvert\}$ by~\ref{c3}, so this gives at the most $2(f+N)$ possible choices for~$l_N$. All in all, we have
    \begin{equation}
    \label{sizeofCN}
        \lvert\mathfrak C_N\rvert\leq 2^{N+1}f(n-1)(n-2)\dots (n-(N-1))(f+N)\leq 2^{N+1}n^{N-1}f(f+N).
    \end{equation}
    Now, given some $l\in\mathfrak C_N$, what is the probability that $\varphi$ contains~$l$? Notice that the sets
    \begin{equation*}
        \{-l_0,l_1\},\{-l_1,l_2\},\dots,\{-l_{N-1},l_N\}
    \end{equation*}
    are pairwise different by property~\ref{c1}, and again by~\ref{c1}, each of the sets contain exactly two elements (literals), \emph{except} for possibly the last set $\{-l_{N-1},l_N\}$, where we might have $-l_{N-1}=l_N$. For any $l,l^\prime\in\pm[n]$ and $j\in[m]$ we have
    \begin{equation*}
        \prob\bigl(\{ L_{j,1},L_{j,2}\}=\{ l,l^\prime\}\bigr)=
        \begin{cases}
            2/(2n)^2=1/(2n^2),&\text{if $l\neq l^\prime$,} \\
            1/(2n)^2=1/(4n^2),&\text{if $l=l^\prime$,}
        \end{cases}
    \end{equation*}
    but in all cases $\prob\bigl(\{ L_{j,1},L_{j,2}\}=\{ l,l^\prime\}\bigr)\leq 1/(2n^2)$. Next we note that, for pairwise different $j_1,j_2,\dots,j_N\in[m]$,
    \begin{equation*}
        \prob\bigl(\{ L_{j_t,1},L_{j_t,2}\}=\{-l_{t-1},l_t\}\text{ for all $t\in[N]$}\bigr)=\prod_{t\in[N]}\prob\bigl(\{ L_{j_t,1},L_{j_t,2}\}=\{-l_{t-1},l_t\}\bigr)\leq\frac{1}{(2n^2)^N},
    \end{equation*}
    and since there are $\binom{m}{N}$ choices for $j_1,j_2,\dots,j_N$ and $N!$ ways to arrange these, we get the following bound for the probability that $\varphi$ contains~$l$:
    \begin{equation}
    \label{proboflinphi}
        \prob(\text{$\varphi$ contains $l$})\leq\binom{m}{N}\frac{N!}{2^Nn^{2N}}.
    \end{equation}
    Putting~\eqref{sizeofCN} and~\eqref{proboflinphi} together yields
    \begin{equation*}
        \mean[Z]=\sum_{N\in\mathbb N}\sum_{l\in\mathfrak C_N}\prob(\text{$\varphi$ contains $l$})\leq\sum_{N\in\mathbb N}\frac{m!2f(f+N)}{(m-N)!n^{N+1}}\leq\frac{f}{\sqrt{n}}\sum_{N\in\mathbb N}2\bigg(\frac{m}{n}\bigg)^N\frac{f+N}{\sqrt{n}}.
    \end{equation*}
    We now choose $\epsilon>0$ small enough that $\alpha+\epsilon<1$, and since $m/n\to\alpha$ as $n\to\infty$, we get for large enough $n$ that $m/n\leq\alpha+\epsilon$. On the other hand, $f/\sqrt{n}\to 0$ as $n\to\infty$, so for $n$ large enough we have $f/\sqrt{n}\leq 1\leq N$ for all $N\in\mathbb N$, so
    \begin{equation*}
        \sum_{N\in\mathbb N}2\bigg(\frac{m}{n}\bigg)^N\frac{f+N}{\sqrt{n}}\leq\sum_{N\in\mathbb N}4(\alpha+\epsilon)^N N\eqqcolon\kappa<\infty,
    \end{equation*}
    where the first inequality holds when $n$ is large enough, and $\kappa$ is a fixed number which does not depend on~$n$, and thus
    \begin{equation*}
        \mean[Z]\leq\frac{f}{\sqrt{n}}\kappa\longrightarrow 0\quad\text{as $n\to\infty$},
    \end{equation*}
    which was the final thing missing.
\end{proof}

We now have all necessary results in place and are thus equipped to prove Theorem~\ref{mainthm}. We prove initially the case $k=2$ when $\alpha>0$, which is the most difficult part of the theorem. We proceed by showing that the claimed limit is both an asymptotic upper- and lower bound on the quantity under consideration.

\subsection{Lower bound}
We now begin our proof of Theorem~\ref{mainthm} in the case $k=2$ when $\alpha>0$. Just as with the special cases, we will prove the result in a slightly more general form, where $\varphi$ is a random $2$-CNF formula with $m=m(n)$ clauses and $n_0=n_0(n)$ variables, where $n_0\to\infty$ and $m/n_0\to\alpha$, $0<\alpha<1$. Also, $\mathcal L\subseteq\pm[n_0]$ has $\lvert\mathcal L\rvert=f=f(n)$, where $f/\sqrt{n_0}\to\beta$, $0\leq\beta\leq\infty$. The goal is to compute $\lim_{n\to\infty}\prob(\varphi_\mathcal L\in\SAT)$. In this section we give a proof of the lower bound
\begin{equation*}
    \liminf_{n\to\infty}\prob(\varphi_\mathcal L\in\SAT)\geq\exp\biggl(\frac{-\beta^2\alpha}{4(1-\alpha)}\biggr).
\end{equation*}

In the case that $\beta=\infty$, the lower bound $\liminf_{n\to\infty}\prob(\varphi_\mathcal L\in\SAT)\geq 0$ is trivial, so assume without loss of generality that $\beta<\infty$. The case $\beta=0$ is Lemma~\ref{alphaisnull}, so assume also that $\beta>0$. We will prove that
\begin{equation*}
    \liminf_{n\to\infty}\prob(\varphi_\mathcal L\in\SAT)\geq\exp\biggl(\frac{-\beta^2(1+\epsilon)\alpha}{4[1-(1+\epsilon)\alpha]}\biggr)
\end{equation*}
for all $\epsilon>0$ small enough so that $(1+\epsilon)\alpha<1$. Taking $\epsilon\to 0$ will then yield the correct lower bound. Thus, let $\epsilon>0$ be such that $(1+\epsilon)\alpha<1$, and define for convenience $\alpha_\epsilon\coloneqq (1+\epsilon)\alpha$.

Let $C_1,\dots,C_m$ be the random $2$-clauses defining~$\varphi$. By Lemma~\ref{probdecreaseinc} we may assume without loss of generality that $\mathcal L=[n_0]\setminus[n_0-f]$. Let as before $K\coloneqq\{ 0,1,2,\star\}$, and let for each $k\in K$ $\mathcal C^{(1)}_k$ denote the set $\mathcal C_k$ from~\eqref{defnmk}, i.e.\ $\mathcal C^{(1)}_k$ is the set of $j\in[m]$ for which the $j$'th clause of $\varphi$ has become a $k$-clause (a $\star$-clause meaning a satisfied clause) when fixing the variables dictated by $\mathcal L$. Put $\mathcal C\coloneqq (\mathcal C^{(1)}_k)_{k\in K}$. Let further $\varphi_1$ and $\varphi_2$ be the random functions defined in~\eqref{defnpsik}, i.e.\
\begin{equation*}
    \varphi_k\coloneqq\min_{j\in\mathcal C^{(1)}_k}(C_j)_\mathcal L,\quad (k=1,2).
\end{equation*}
Define finally
\begin{equation*}
    n_1\coloneqq n_0-f,\quad f_1\coloneqq\bigl\lfloor\alpha_\epsilon f+\alpha_\epsilon n_0^{3/8}\bigr\rfloor,\quad\mathcal L_1\coloneqq [n_1]\setminus[n_1-f_1],
\end{equation*}
(and note that all of $n_1$, $f_1$, and $\mathcal L_1$ depend on $n$), and let $\lambda^{(1)}$ denote a random $1$-CNF formula with $f_1$ clauses and $n_1$ variables, and let $\varphi^{(1)}$ denote a random $2$-CNF formula with $m$ clauses and $n_1$ variables. The exponent $3/8$ in the definition of $f_1$ could be any number strictly between $1/4$ and $1/2$, as the proof will show.

With all this in place, we get from~\eqref{phiLsatiff} that
\begin{align*}
    \prob(\varphi_\mathcal L\in\SAT)&=\prob\bigl(\lvert\mathcal C^{(1)}_0\rvert=0,\,\varphi_1\in\SAT,\, (\varphi_2)_{\varphi_1}\in\SAT\bigr) \\
    &\geq\prob\bigl(\lvert\mathcal C^{(1)}_0\rvert=0,\,\lvert\mathcal C^{(1)}_1\rvert\leq f_1,\,\varphi_1\in\SAT,\, (\varphi_2)_{\varphi_1}\in\SAT\bigr) \\
    &=\sum_{M\in\mathfrak M}\prob\bigr(\varphi_1\in\SAT,\, (\varphi_2)_{\varphi_1}\in\SAT\bigm\vert\mathcal C=M\bigr)\prob(\mathcal C=M),
\end{align*}
where $\mathfrak M$ is the set of partitions $(M_k)_{k\in K}$ of $[m]$ such that $\lvert M_0\rvert=0$ and $\lvert M_1\rvert\leq f_1$. For any $M\in\mathfrak M$ we now consider for $k=1,2$ an auxiliary random $k$-CNF formula $\psi_k$ with $\lvert M_k\rvert$ clauses and $n_1$ variables, such that $\psi_1$ and $\psi_2$ are independent. Then by Lemma~\ref{distofphiandm} the distribution of $(\psi_1,\psi_2)$ (under~$\prob$) is the same as that of $(\varphi_1,\varphi_2)$ in the conditional distribution given $\mathcal C=M$. Hence, letting $\mathcal B$ denote the set of all satisfiable $1$-CNF formulas with $\lvert M_1\rvert$ clauses and $n_1$ variables, we get
\begin{equation}
\label{transfermethod}
\begin{aligned}
    \prob\bigl(\varphi_1\in\SAT,\, (\varphi_2)_{\varphi_1}\in\SAT\bigm\vert\mathcal C=M\bigr)&=\prob\bigl(\psi_1\in\SAT,\, (\psi_2)_{\psi_1}\in\SAT\bigr) \\
    &=\sum_{h\in\mathcal B}\prob\bigl(\psi_1=h,\, (\psi_2)_h\in\SAT\bigr) \\
    &=\sum_{h\in\mathcal B}\prob(\psi_1=h)\prob\bigl((\psi_2)_h\in\SAT\bigr) \\
    &\geq\prob\bigl(\psi_1\in\SAT\bigr)\prob\bigl((\psi_2)_{\mathcal L_1}\in\SAT\bigr) \\
    &\geq\prob\bigl(\lambda^{(1)}\in\SAT\bigr)\prob\bigl(\varphi^{(1)}_{\mathcal L_1}\in\SAT\bigr),
\end{aligned}
\end{equation}
where we in the first inequality have used Lemma~\ref{probdecreaseinc}(ii) (since $\lvert\mathcal L_1\rvert=f_1\geq\lvert M_1\rvert$), and in the final inequality used Lemma~\ref{probdecreaseinm} with the fact that $\lambda^{(1)}$ has more clauses than~$\psi_1$ and $\varphi^{(1)}$ has more clauses than~$\psi_2$. Using this, we get
\begin{align*}
    \MoveEqLeft[2.5]\smash{\sum_{M\in\mathfrak M}}\prob\bigl(\varphi_1\in\SAT,\, (\varphi_2)_{\varphi_1}\in\SAT\bigm\vert\mathcal C=M\bigr)\prob(\mathcal C=M) \\
    &\geq\prob\bigl(\lambda^{(1)}\in\SAT\bigr)\prob\bigl(\varphi^{(1)}_{\mathcal L_1}\in\SAT\bigr)\prob\bigl(\lvert\mathcal C^{(1)}_0\rvert=0,\,\lvert\mathcal C^{(1)}_1\rvert\leq f_1\bigr) \\
    &=\prob\bigl(\lambda^{(1)}\in\SAT\bigr)\prob\bigl(\varphi^{(1)}_{\mathcal L_1}\in\SAT\bigr)\prob\bigl(\lvert\mathcal C^{(1)}_0\rvert=0\bigr)\prob\bigl(\lvert\mathcal C^{(1)}_1\rvert\leq f_1\bigm\vert\lvert\mathcal C^{(1)}_0\rvert=0\bigr),
\end{align*}
where the final factor should read: the probability that $\lvert\mathcal C^{(1)}_1\rvert\leq f_1$ given $\lvert\mathcal C^{(1)}_0\rvert=0$. All together we conclude that
\begin{equation}
\label{lowerbdfirst}
    \prob(\varphi_\mathcal L\in\SAT)\geq\prob\bigl(\lambda^{(1)}\in\SAT\bigr)\prob\bigl(\varphi^{(1)}_{\mathcal L_1}\in\SAT\bigr)\prob\bigl(\lvert\mathcal C^{(1)}_0\rvert=0\bigr)\prob\bigl(\lvert\mathcal C^{(1)}_1\rvert\leq f_1\bigm\vert\lvert\mathcal C^{(1)}_0\rvert=0\bigr).
\end{equation}

Notice now that we are in the same situation with $\varphi^{(1)}_{\mathcal L_1}$ as we were with~$\varphi_\mathcal L$, so we can repeat the above procedure. Put
\begin{equation*}
    R\coloneqq\bigl\lfloor c\log(n_0)\bigr\rfloor,\quad\text{where}\quad c\coloneqq\frac{1}{16\log(1/\alpha_\epsilon)}>0.
\end{equation*}
and define for all $r\in\mathbb N$ (where we take $f_0\coloneqq f$):
\begin{equation*}
    n_r\coloneqq n_{r-1}-f_{r-1}=n_0-\sum_{k=0}^{r-1}f_k,\quad f_r\coloneqq\bigl\lfloor\alpha_\epsilon^r f+r\alpha_\epsilon^r n_0^{3/8}\bigr\rfloor,\quad\mathcal L_r\coloneqq[n_r]\setminus[n_r-f_r],
\end{equation*}
and let $\mathcal C^{(r)}_k$ denote the set $\mathcal C_k$ from~\eqref{defnmk} with $\varphi^{(r-1)}$ in place of~$\varphi$ and $\mathcal L_{r-1}$ in place of~$\mathcal L$ for each $k\in K$, and denote finally by $\lambda^{(r)}$ a random $1$-CNF formula with $f_r$ clauses and $n_r$ variables, and by $\varphi^{(r)}$ a random $2$-CNF formula with $m$ clauses and $n_r$ variables. Redoing the argument above, we get for each $r\in\mathbb N$ the following:
\begin{equation*}
    \prob\bigl(\varphi^{(r-1)}_{\mathcal L_{r-1}}\in\SAT\bigr)\geq\prob\bigl(\lambda^{(r)}\in\SAT\bigr)\prob\bigl(\varphi^{(r)}_{\mathcal L_r}\in\SAT\bigr)\prob\bigl(\lvert\mathcal C^{(r)}_0\rvert=0\bigr)\prob\bigl(\lvert\mathcal C^{(r)}_1\rvert\leq f_r\bigm\vert\lvert\mathcal C^{(r)}_0\rvert=0\bigr),
\end{equation*}
and by successively applying this inequality $R$ times we obtain:
\begin{equation}
\begin{aligned}
    \prob(\varphi_\mathcal L\in\SAT)\geq{}&\prob\bigl(\varphi^{(R)}_{\mathcal L_R}\in\SAT\bigr)\prod_{r=1}^R\prob\bigl(\lambda^{(r)}\in\SAT\bigr) \\
    &\times\prod_{r=1}^R\prob\bigl(\lvert\mathcal C^{(r)}_0\rvert=0\bigr)\prod_{r=1}^R\prob\bigl(\lvert\mathcal C^{(r)}_1\rvert\leq f_r\bigm\vert\lvert\mathcal C^{(r)}_0\rvert=0\bigr).
\end{aligned}
\label{maindecomp}
\end{equation}
Hence, the stipulated asymptotic lower bound will follow if we establish the following points:
\begin{enumerate}[label=(L\arabic*)]
    \item\label{L1} $\lim_{n\to\infty}\prob\bigl(\varphi^{(R)}_{\mathcal L_R}\in\SAT\bigr)=1$,
    \item\label{L2} $\lim_{n\to\infty}\prod_{r=1}^R\prob\bigl(\lambda^{(r)}\in\SAT\bigr)=\exp\biggl(\frac{-\beta^2}{4}\sum_{r=1}^\infty\alpha_\epsilon^{2r}\biggr)$,
    \item\label{L3} $\lim_{n\to\infty}\prod_{r=1}^R\prob\bigl(\lvert\mathcal C^{(r)}_0\rvert=0\bigr)\geq\exp\biggl(\frac{-\beta^2}{4}\sum_{r=1}^\infty\alpha_\epsilon^{2r-1}\biggr)$,
    \item\label{L4} $\lim_{n\to\infty}\prod_{r=1}^R\prob\bigl(\lvert\mathcal C^{(r)}_1\rvert\leq f_r\bigm\vert\lvert\mathcal C^{(r)}_0\rvert=0\bigr)=1$.
\end{enumerate}
It will be useful to note that, according to Lemma~\ref{distofphiandm}, it holds for each $n,r\in\mathbb N$ that
\begin{equation}
\label{distofmr}
    \bigl(\lvert\mathcal C^{(r)}_0\rvert,\lvert\mathcal C^{(r)}_1\rvert,\lvert\mathcal C^{(r)}_2\rvert,\lvert\mathcal C^{(r)}_\star\rvert\bigr)\sim\mathrm{Multinomial}\bigl(m,(p^{(r)}_0,p^{(r)}_1,p^{(r)}_2,p^{(r)}_\star)\bigr),
\end{equation}
where
\begin{equation*}
    p^{(r)}_0=\frac{f_{r-1}(f_{r-1}-1)}{4n_{r-1}(n_{r-1}-1)},\quad p^{(r)}_1\coloneqq\frac{f_{r-1}n_r}{n_{r-1}(n_{r-1}-1)},\quad p^{(r)}_2\coloneqq\frac{n_r(n_r-1)}{n_{r-1}(n_{r-1}-1)},
\end{equation*}
and lastly $p^{(r)}_\star=1-p^{(r)}_0-p^{(r)}_1-p^{(r)}_2$. It will further be useful to note that, for any fixed $r\in\mathbb N$,
\begin{equation}
\label{computelimfr}
    \lim_{n\to\infty}\frac{f_r}{\sqrt{n_0}}=\alpha_\epsilon^r\beta.
\end{equation}
Putting then $S\coloneqq\sup_{n\in\mathbb N}(n_0^{3/8}/f)<\infty$, we get the upper bound
\begin{equation}
\label{boundfr}
    f_r\leq\alpha_\epsilon^r f+r\alpha_\epsilon^r n_0^{3/8}\leq\alpha_\epsilon^r f+Sr\alpha_\epsilon^r f=(1+Sr)\alpha_\epsilon^r f
\end{equation}
for all $r\in\mathbb N$. This yields further the following bound:
\begin{equation}
\label{boundsumfk}
    \sum_{k=0}^\infty f_k\leq\biggl[\sum_{k=0}^\infty\alpha_\epsilon^k+S\sum_{k=0}^\infty k\alpha_\epsilon^k\biggr]f=Cf,
\end{equation}
where $C=(1+(S-1)\alpha_\epsilon)/(1-\alpha_\epsilon)^2<\infty$. We get from the definition of $n_r$ and~\eqref{boundsumfk} that
\begin{equation}
\label{boundur}
    n_0-Cf\leq n_r\leq n_0
\end{equation}
for all $r\in\mathbb N$ and $n\in\mathbb N$, yielding in particular (noticing that also $R$ depends on~$n$)
\begin{equation}
\label{computelimuranduR}
    \lim_{n\to\infty}\frac{n_r}{n_0}=1,\quad\text{and}\quad\lim_{n\to\infty}\frac{n_R}{n_0}=1.
\end{equation}
Since $c\log(n_0)\leq R+1$, $\alpha_\epsilon<1$ and $c=1/(16\log(1/\alpha_\epsilon))$, we get the following bounds on~$\alpha_\epsilon^R$:
\begin{equation}
\label{betapowerR}
    \frac{1}{n_0^{1/16}}=\alpha_\epsilon^{c\log(n_0)}\leq\alpha_\epsilon^R\leq\alpha_\epsilon^{c\log(n_0)-1}=\frac{1}{\alpha_\epsilon n_0^{1/16}}.
\end{equation}
Using the lower bound, we find for $r\in[R]$ that $r\alpha_\epsilon^r n_0^{3/8}\geq\alpha_\epsilon^R n_0^{3/8}\geq n_0^{5/16}\geq 1$, so
\begin{equation}
\label{lowerboundfr}
    f_r=\bigl\lfloor\alpha_\epsilon^r f+r\alpha_\epsilon^r n_0^{3/8}\bigr\rfloor\geq\alpha_\epsilon^r f\geq\alpha_\epsilon^R f\geq\frac{f}{n_0^{1/16}}
\end{equation}
for all $r\in[R]$ and $n\in\mathbb N$.

Using now the upper bound on $\alpha_\epsilon^{R(n)}$ from~\eqref{betapowerR}, we conclude using initially~\eqref{boundfr} that
\begin{equation*}
    f_R\leq (1+SR)\alpha_\epsilon^R f\leq\frac{1+Sc\log(n_0)}{n_0^{1/16}}\cdot\frac{f}{\alpha_\epsilon},
\end{equation*}
so that $f_R/\sqrt{n_0}\to 0$. Hence, since $\lvert\mathcal L_R\rvert=f_R$ and since $\varphi^{(R)}$ is a random $2$-CNF formula with $m$ clauses and $n_R$ variables, where we have seen that $n_R/n_0\to 1$, we see that~\ref{L1} follows from Lemma~\ref{alphaisnull}.

We now establish~\ref{L2}. Since $\lambda^{(r)}$ is a random $1$-CNF formula with $f_r$ clauses and $n_r$ variables for every $r\in\mathbb N$, Lemma~\ref{random1sat} together with~\eqref{computelimfr} and~\eqref{computelimuranduR} shows that
\begin{equation*}
    \lim_{n\to\infty}\prob\bigl(\lambda^{(r)}\in\SAT\bigr)=e^{-\tfrac{1}{4}\beta^2\alpha_\epsilon^{2r}}.
\end{equation*}
Taking the logarithm, we want to conclude that
\begin{equation*}
    \lim_{n\to\infty}\sum_{r=1}^\infty\log\prob\bigl(\lambda^{(r)}\in\SAT\bigr)1_{[R]}(r)=-\frac{1}{4}\beta^2\sum_{r=1}^\infty\alpha_\epsilon^{2r},
\end{equation*}
which we will do by arguing that dominated convergence applies. This requires a uniform (over large enough~$n$) summable (in~$r$) lower bound on $\log\prob(\lambda^{(r)}\in\SAT)$. It follows from~\eqref{needforlater1sat} in the proof of Lemma~\ref{random1sat} that
\begin{equation*}
    \prob\bigl(\lambda^{(r)}\in\SAT\bigr)\geq\sum_{k=0}^{\lfloor f_r/2\rfloor}\binom{n_r-k}{f_r-2k}\binom{n_r}{k}\frac{f_r!}{4^k n_r^{f_r}}\geq\binom{n_r}{f_r}\frac{f_r!}{n_r^{f_r}}=\prod_{d=0}^{f_r-1}\frac{n_r-d}{n_r},
\end{equation*}
where we in the second inequality simply remove all but the first term. Now, from the classical inequality $\log(t)\leq t-1$ for all $t>0$, we multiply by~$-1$ and evaluate at $t=1/(1-s)$, yielding
\begin{equation}
\label{canonlogbound}
    \log(1-s)\geq\frac{s}{s-1}\quad\text{for all $s<1$}.
\end{equation}
Using this and our inequalities above, we get
\begin{equation*}
    \log\prob\bigl(\lambda^{(r)}\in\SAT\bigr)\geq\sum_{d=0}^{f_r-1}\log\bigl(1-\tfrac{d}{n_r}\bigr)\geq\sum_{d=0}^{f_r-1}\frac{d}{d-n_r}\geq\frac{1}{f_r-n_r}\sum_{d=0}^{f_r-1}d=\frac{(f_r-1)f_r}{2(f_r-n_r)}\geq\frac{-f_r^2}{2n_{r+1}},
\end{equation*}
where we have used $n_r-f_r=n_{r+1}$. From here we apply~\eqref{boundfr} and~\eqref{boundur} to get
\begin{equation*}
    \frac{f_r^2}{2n_{r+1}}\leq (1+Sr)^2\alpha_\epsilon^{2r}\frac{f^2}{2(n_0-Cf)}.
\end{equation*}
Now, since
\begin{equation*}
    \frac{f^2}{2(n_0-Cf)}=\frac{1}{2}\biggl(\frac{f}{\sqrt{n_0}}\biggr)^2\biggl(\frac{n_0}{n_0-Cf}\biggr)\longrightarrow\frac{\beta^2}{2}
\end{equation*}
as $n\to\infty$, we conclude that, when $n$ is large enough, it holds for all $r\in\mathbb N$ that
\begin{equation*}
    \log\prob\bigl(\lambda^{(r)}\in\SAT\bigr)\geq-\beta^2(1+Sr)^2\alpha_\epsilon^{2r},
\end{equation*}
and of course $(1+Sr)^2\alpha_\epsilon^r\to 0$ as $r\to\infty$, so $S^\prime\coloneqq\sup_{r\in\mathbb N}(1+Sr)^2\alpha_\epsilon^r<\infty$, thus
\begin{equation}
\label{lowerboundsummable}
    \sum_{r=1}^\infty (1+Sr)^2\alpha_\epsilon^{2r}\leq S^\prime\sum_{r=1}^\infty\alpha_\epsilon^r=S^\prime\frac{\alpha_\epsilon}{1-\alpha_\epsilon}<\infty,
\end{equation}
as required.

We now establish~\ref{L3}. From~\eqref{distofmr} it follows that $\lvert\mathcal C^{(r)}_0\rvert\sim\mathrm{Binomial}(m,p^{(r)}_0)$ for all $r\in\mathbb N$. Writing
\begin{equation*}
    \prob\bigl(\lvert\mathcal C^{(r)}_0\rvert=0\bigr)=\bigl(1-p^{(r)}_0\bigr)^m=\biggl(\biggl(1-\biggl(\frac{f_{r-1}(f_{r-1}-1)}{n_0}\biggr)\biggl(\frac{n_0^2}{n_{r-1}(n_{r-1}-1)}\bigg)\frac{1}{4n_0}\biggr)^{n_0}\biggr)^{m/n_0},
\end{equation*}
we find, using \eqref{computelimfr}, \eqref{computelimuranduR}, and $m/n_0\to\alpha$, that $\prob(\lvert\mathcal C^{(r)}_0\rvert=0)\to e^{-\tfrac{1}{4}\beta^2\alpha_\epsilon^{2(r-1)}\alpha}$ as $n\to\infty$. Taking the logarithm, we want to conclude that
\begin{equation*}
    \lim_{n\to\infty}\sum_{r=1}^\infty\log\prob\bigl(\lvert\mathcal C^{(r)}_0\rvert=0\bigr)1_{[R]}(r)=-\frac{1}{4}\beta^2\alpha\sum_{r=1}^\infty\alpha_\epsilon^{2(r-1)}\geq -\frac{1}{4}\beta^2\sum_{r=1}^\infty\alpha_\epsilon^{2r-1},
\end{equation*}
where the inequality follows from $\alpha\leq (1+\epsilon)\alpha=\alpha_\epsilon$, and the equality will again follow from an application of dominated convergence, once we give a uniform (over large enough~$n$) summable (in~$r$) lower bound on the terms $\log\prob(\lvert\mathcal C^{(r)}_0\rvert=0)$. Notice first of all that, using~\eqref{boundsumfk} and~\eqref{boundur},
\begin{equation*}
    p^{(r)}_0\leq\frac{C^2}{4}\biggl(\frac{f}{\sqrt{n_0}}\biggr)^2\biggl(\frac{n_0}{n_0-Cf-1}\biggr)^2\frac{1}{n_0}\longrightarrow 0\quad\text{as $n\to\infty$},
\end{equation*}
where the convergence is uniform in~$r$ (since the upper bound does not depend on~$r$). Thus, when $n$ is large enough, $p^{(r)}_0\leq 1/2$ for all $r\in\mathbb N$. Next, using now~\eqref{boundfr} and~\eqref{boundur}, we find that
\begin{equation*}
    mp^{(r)}_0\leq\frac{(1+Sr)^2\alpha_\epsilon^{2r}}{4}\biggl(\frac{f}{\sqrt{n_0}}\biggr)^2\biggl(\frac{n_0}{n_0-Cf-1}\biggr)^2\frac{m}{n_0},
\end{equation*}
where of course
\begin{equation*}
    \lim_{n\to\infty}\biggl(\frac{f}{\sqrt{n_0}}\biggr)^2\biggl(\frac{n_0}{n_0-Cf-1}\biggr)^2\frac{m}{n_0}=\beta^2\alpha\leq\beta^2\alpha_\epsilon,
\end{equation*}
so it follows that when $n$ is large enough, it holds for all $r\in\mathbb N$ that
\begin{equation}
\label{boundmpr0}
    mp^{(r)}_0\leq\tfrac{1}{2}\beta^2(1+Sr)^2\alpha_\epsilon^{2r-1}.
\end{equation}
Thus, when $n$ is large enough to satisfy both~\eqref{boundmpr0} and $p^{(r)}_0\leq\tfrac{1}{2}$ for all $r\in\mathbb N$, we get by using~\eqref{canonlogbound} that
\begin{equation*}
    \log\prob\bigl(\lvert\mathcal C^{(r)}_0\rvert=0\bigr)=m\log\bigl(1-p^{(r)}_0\bigr)\geq\frac{mp^{(r)}_0}{p^{(r)}_0-1}\geq -\beta^2(1+Sr)^2\alpha_\epsilon^{2r-1}
\end{equation*}
holds for all $r\in\mathbb N$, where we used $p^{(r)}_0\leq\tfrac{1}{2}$ in the denominator and~\eqref{boundmpr0} in the numerator, and~\eqref{lowerboundsummable} shows that this lower bound is summable, proving~\ref{L3}.

We now establish~\ref{L4}. Here we will not be able to apply dominated convergence, and it is for this reason that we stop the ``splitting'' process in~\eqref{maindecomp} after $R$ steps/rounds. If we were able to establish~\ref{L4} with (symbolically) $R=\infty$, then we could have completely bypassed a verification of~\ref{L1} and thus the need for Lemma~\ref{alphaisnull}. But alas, we proceed without the comfort of dominated convergence and with the need for a verification of~\ref{L1} and Lemma~\ref{alphaisnull}. Also, points~\ref{L1} through~\ref{L3} could all have been verified with $\alpha$ in place of~$\alpha_\epsilon$, and it is only for this point~\ref{L4} that we need the ``$\epsilon$-breathing room'' that $\alpha_\epsilon$ provides us over a more direct calculation using~$\alpha$. Finally, it is at this point that we need the exponent~$3/8$ in the definition of~$f_r$ to be greater than~$1/4$, and actually the precise definition of~$f_r$ comes into play, where only the first order asymptotics of~$f_r$ mattered for the other points. Thus, point~\ref{L4} is by far the most delicate out of the four.

First of all, if we define
\begin{equation*}
    \xi\coloneqq\max_{r\in [R]}\prob\bigl(\lvert\mathcal C^{(r)}_1\rvert>f_r\bigm\vert\lvert\mathcal C^{(r)}_0\rvert=0\bigr),
\end{equation*}
then taking the logarithm gives
\begin{align*}
    \sum_{r=1}^R\log\prob\bigl(\lvert\mathcal C^{(r)}_1\rvert\leq f_r\bigm\vert\lvert\mathcal C^{(r)}_0\rvert=0\bigr)&\geq R\min_{r\in [R]}\log\prob\bigl(\lvert\mathcal C^{(r)}_1\rvert\leq f_r\bigm\vert\lvert\mathcal C^{(r)}_0\rvert=0\bigr) \\
    &=R\log(1-\xi)\geq R\frac{\xi}{\xi-1}\geq\frac{c\log(n_0)\xi}{\xi-1},
\end{align*}
where we in the second to last inequality use~\eqref{canonlogbound}. Hence, if we can show that $\log(n_0)\xi\to 0$ as $n\to\infty$, then certainly $\xi\to 0$ as well, and thus
\begin{equation*}
    \lim_{n\to\infty}\frac{c\log(n_0)\xi}{\xi-1}=0,
\end{equation*}
from which~\ref{L4} follows. Now, from~\eqref{distofmr} we get from known results that in the conditional distribution given $\lvert\mathcal C^{(r)}_0\rvert=0$, $\lvert\mathcal C^{(r)}_1\rvert$ follows a Binomial distribution with parameters~$m$ and
\begin{equation*}
    p^{(r)}_{1\mid 0}\coloneqq\frac{p^{(r)}_1}{1-p^{(r)}_0}=\frac{f_{r-1}n_r}{n_{r-1}(n_{r-1}-1)-\tfrac{1}{4}f_{r-1}(f_{r-1}-1)}.
\end{equation*}
Now consider for a moment the expression
\begin{equation}
\label{neededforL4}
    \frac{mn_r}{n_{r-1}(n_{r-1}-1)-\tfrac{1}{4}f_{r-1}(f_{r-1}-1)}\leq\frac{mn_0}{(n_0-Cf-1)^2-\tfrac{1}{4}(Cf)^2}\longrightarrow\alpha\quad\text{as $n\to\infty$},
\end{equation}
where we have used~\eqref{boundsumfk} and~\eqref{boundur}, and the convergence is uniform in~$r$. Letting then $N^{(r)}\sim\mathrm{Binomial}(m,p^{(r)}_{1\mid 0})$ under~$\prob$, we see that
\begin{equation}
\label{meanNrupperbound}
    \mean[N^{(r)}]=mp^{(r)}_{1\mid 0}=f_{r-1}\frac{mn_r}{n_{r-1}(n_{r-1}-1)-\tfrac{1}{4}f_{r-1}(f_{r-1}-1)}\leq (1+\epsilon)\alpha f_{r-1}=\alpha_\epsilon f_{r-1},
\end{equation}
where the last inequality holds for large enough~$n$ independent of~$r$ thanks to~\eqref{neededforL4}. Furthermore, since $f_{r-1}=\lfloor\alpha_\epsilon^{r-1}f+(r-1)\alpha_\epsilon^{r-1}n_0^{3/8}\rfloor$, we get
\begin{equation*}
    \alpha_\epsilon f_{r-1}\leq\alpha_\epsilon^r f+r\alpha_\epsilon^r n_0^{3/8}-\alpha_\epsilon^r n_0^{3/8},
\end{equation*}
giving all together
\begin{equation}
\label{meannrbound}
    \alpha_\epsilon^r f+r\alpha_\epsilon^r n_0^{3/8}\geq\mean[N^{(r)}]+\alpha_\epsilon^r n_0^{3/8},
\end{equation}
uniformly in~$r$ for large enough~$n$. Now,
\begin{equation*}
    \prob\bigl(\lvert\mathcal C^{(r)}_1\rvert>f_r\bigm\vert\lvert\mathcal C^{(r)}_0\rvert=0\bigr)=\prob\bigl(N^{(r)}>f_r\bigr),
\end{equation*}
and since $N^{(r)}$ takes only integer values, rounding down makes no difference, i.e.\
\begin{equation*}
    \prob\bigl(N^{(r)}>f_r\bigr)=\prob\bigl(N^{(r)}>\alpha_\epsilon^r f+r\alpha_\epsilon^r n_0^{3/8}\bigr)\leq\prob\bigl(N^{(r)}>\mean[N^{(r)}]+\alpha_\epsilon^r n_0^{3/8}\bigr),
\end{equation*}
where we in the inequality have used~\eqref{meannrbound}, so it holds uniformly in~$r$ for large enough~$n$. We will now apply Chernoff's bound:
\begin{equation}
\label{chernoff}
    \prob\bigl(N>\mean[N]+\delta\mean[N]\bigr)\leq e^{-\tfrac{1}{3}\delta^2\mean[N]},
\end{equation}
when $N$ is binomially distributed and $0<\delta<1$. Taking in our case $N=N^{(r)}$ and
\begin{equation*}
    \delta=\frac{\alpha_\epsilon^r n_0^{3/8}}{\mean[N^{(r)}]},
\end{equation*}
we find that, as long as $r\in[R]$, \eqref{boundur} and~\eqref{lowerboundfr} give us
\begin{equation*}
    \mean[N^{(r)}]=f_{r-1}\frac{mn_r}{n_{r-1}(n_{r-1}-1)-\tfrac{1}{4}f_{r-1}(f_{r-1}-1)}\geq\frac{f}{n_0^{1/16}}\cdot\frac{m(n_0-Cf)}{n_0^2},
\end{equation*}
so using~\eqref{betapowerR} in addition to the above we get for all $r\in[R]$:
\begin{equation*}
    \delta=\frac{\alpha_\epsilon^r n_0^{3/8}}{\mean[N^{(r)}]}\leq\frac{1}{n_0^{1/16}}\cdot\frac{\sqrt{n_0}}{f}\cdot\frac{n_0}{m}\cdot\frac{n_0}{n_0-Cf}\longrightarrow 0\quad\text{as $n\to\infty$},
\end{equation*}
so naturally $\delta<1$ for all $r\in [R]$ for large enough~$n$. Hence, Chernoff's bound yields
\begin{equation*}
    \prob\bigl(N^{(r)}>\mean[N^{(r)}]+\alpha_\epsilon^r n_0^{3/8}\bigr)\leq\exp\biggl(\frac{-\alpha_\epsilon^{2r}n_0^{3/4}}{3\mean[N^{(r)}]}\biggr)
\end{equation*}
for all $r\in[R]$ when $n$ is large enough. We saw in~\eqref{meanNrupperbound} that $\mean[N^{(r)}]\leq f_{r-1}$ for all $r\in\mathbb N$ when $n$ is large enough, and $f_{r-1}\leq Cf$ for all $r,n\in\mathbb N$ by~\eqref{boundsumfk}. On the other hand, $\alpha_\epsilon^{2r}\geq\alpha_\epsilon^{2R}\geq n_0^{-1/8}$ for all $r\in[R]$ and $n\in\mathbb N$ by~\eqref{betapowerR}, so taken all together we get when $n$ is large enough:
\begin{equation*}
    \xi=\max_{r\in[R]}\prob\bigl(N^{(r)}>f_r\bigr)\leq\max_{r\in[R]}\exp\biggl(\frac{-\alpha_\epsilon^{2r}n_0^{3/4}}{3\mean[N^{(r)}]}\biggr)\leq\exp\biggl(\frac{-n_0^{5/8}}{3Cf}\biggr),
\end{equation*}
and of course $n_0^{5/8}/(3Cf)$ is asymptotic to $n_0^{1/8}/(3C\beta)$ as $n\to\infty$, readily yielding $\log(n_0)\xi\to 0$, completing the proof of~\ref{L4} and thus the entire proof of the lower bound.

\subsection{Upper bound}
We now give a proof of the upper bound
\begin{equation*}
    \limsup_{n\to\infty}\prob(\varphi_\mathcal L\in\SAT)\leq\exp\biggl(\frac{-\beta^2\alpha}{4(1-\alpha)}\biggr).
\end{equation*}
Taken together with the lower bound, this will complete the proof of the case $k=2$ in Theorem~\ref{mainthm} when $\alpha>0$. We are still assuming that $\varphi$ is a random $2$-CNF formula with $m=m(n)$ clauses and $n_0=n_0(n)$ variables, where $n_0\to\infty$ and $m/n_0\to\alpha$, and that $\lvert\mathcal L\rvert=f=f(n)$, where $f/\sqrt{n_0}\to\beta$. A priori we have $0<\alpha<1$ and $0\leq\beta\leq\infty$, but the case $\beta=0$ is trivial, so we assume without loss of generality that $\beta>0$. We will also assume that $\beta<\infty$, since the case $\beta=\infty$ follows. Indeed, let for a given $T>0$ $\psi$ denote a random $2$-CNF formula with $n_0$ variables and $\lfloor T\sqrt{n_0}\rfloor$ clauses. Then $f\geq\lfloor T\sqrt{n_0}\rfloor$ for large enough~$n$ if $\beta=\infty$, so by Lemma~\ref{probdecreaseinm} we get
\begin{equation*}
    \limsup_{n\to\infty}\prob(\varphi_\mathcal L\in\SAT)\leq\limsup_{n\to\infty}\prob(\psi_\mathcal L\in\SAT)\leq\exp\biggl(\frac{-T^2\alpha}{4(1-\alpha)}\biggr).
\end{equation*}
Letting $T\to\infty$ gives the desired result. Hence, assume $0<\beta<\infty$.

The method for proving the upper bound will be similar to the one for the lower bound. We ``reset'' the notation from the lower bound and begin anew. Let $0<\epsilon<1$ be given. We prove the bound
\begin{equation*}
    \limsup_{n\to\infty}\prob(\varphi_\mathcal L\in\SAT)\leq\exp\biggl(\frac{-\beta^2(1-\epsilon)\alpha}{4[1-(1-\epsilon)\alpha]}\biggr),
\end{equation*}
from which the desired bound follows by taking $\epsilon\to 0$. Put $\alpha_\epsilon\coloneqq (1-\epsilon)\alpha$.

As in the lower bound we assume without loss of generality that $\mathcal L=[n_0]\setminus[n_0-f]$ by Lemma~\ref{probdecreaseinc}. Let $K=\{ 0,1,2,\star\}$, and decompose $\varphi_\mathcal L$ into $\varphi_1$ and~$\varphi_2$ with corresponding sets $\mathcal C=(\mathcal C^{(1)}_k)_{k\in K}$ according to~\eqref{defnmk}. Define
\begin{gather*}
    n_1\coloneqq n_0-f,\quad f_1\coloneqq\bigl\lfloor\alpha_\epsilon f-\alpha_\epsilon n_0^{3/8}\bigr\rfloor,\quad w\coloneqq\bigl\lfloor n_0^{1/4}\bigr\rfloor, \\
    m_1\coloneqq\bigl\lceil m-3\alpha\bigl(f+n_0^{3/8}\bigr)\bigr\rceil,\quad\mathcal L_1\coloneqq [n_1]\setminus [n_1-(f_1-w)],
\end{gather*}
where $\lceil\cdot\rceil$ denotes the ceiling function (rounding up), and let $\lambda^{(1)}$ denote a random $1$-CNF formula with $f_1$ clauses and $n_1$ variables, and let $\Lambda^{(1)}$ denote the set of $f_1$ random literals defining~$\lambda^{(1)}$, and let finally $\varphi^{(1)}$ denote a random $2$-CNF formula with $m_1$ clauses and $n_1$ variables.

We get from~\eqref{phiLsatiff} that
\begin{align*}
    \prob(\varphi_\mathcal L\in\SAT)={}&\prob\bigl(\lvert\mathcal C^{(1)}_0\rvert=0,\,\varphi_1\in\SAT,\, (\varphi_2)_{\varphi_1}\in\SAT\bigr) \\
    \leq{}&\prob\bigl(\lvert\mathcal C^{(1)}_0\rvert=0,\,\lvert\mathcal C^{(1)}_1\rvert\geq f_1,\,\lvert\mathcal C^{(1)}_2\rvert\geq m_1,\,\varphi_1\in\SAT,\, (\varphi_2)_{\varphi_1}\in\SAT\bigr) \\
    &+\prob\bigl(\lvert\mathcal C^{(1)}_1\rvert<f_1\bigr)+\prob\bigl(\lvert\mathcal C^{(1)}_2\rvert<m_1\bigr).
\end{align*}
If $\mathfrak M$ denotes the set of partitions $(M_k)_{k\in K}$ of $[m]$ such that $\lvert M_0\rvert=0$, $\lvert M_1\rvert\geq f_1$, and $\lvert M_2\rvert\geq m_1$, we get
\begin{align*}
    \MoveEqLeft\prob\bigl(\lvert\mathcal C^{(1)}_0\rvert=0,\,\lvert\mathcal C^{(1)}_1\rvert\geq f_1,\,\lvert\mathcal C^{(1)}_2\rvert\geq m_1,\,\varphi_1\in\SAT,\, (\varphi_2)_{\varphi_1}\in\SAT\bigr) \\
    &=\sum_{M\in\mathfrak M}\prob\bigl(\varphi_1\in\SAT,\, (\varphi_2)_{\varphi_1}\in\SAT\bigm\vert\mathcal C=M\bigr)\prob(\mathcal C=M).
\end{align*}
For a given $M\in\mathfrak M$ we now consider for $k=1,2$ an auxiliary random $k$-CNF formula~$\psi_k$ with $\lvert M_k\rvert$ clauses and $n_1$ variables, such that $\psi_1$ and $\psi_2$ are independent. Then, by Lemma~\ref{distofphiandm},
\begin{equation*}
    \prob\bigl(\varphi_1\in\SAT,\, (\varphi_2)_{\varphi_1}\in\SAT\bigm\vert\mathcal C=M\bigr)=\prob(\psi_1\in\SAT,\, (\psi_2)_{\psi_1}\in\SAT).
\end{equation*}
The next step of this argument is a bit more involved than in the lower bound, since here, the number of \emph{distinct} clauses in~$\psi_1$ might be lower than~$\lvert M_1\rvert$ owing to duplicates, and thus, $\psi_1$ might fix fewer than $\lvert M_1\rvert$ variables in~$\psi_2$. This was of course also the case in the lower bound, but fixing the maximum number of variables only reduced the probability of satisfiability even further; here, we must deal with this issue proper. Let $\Lambda$ denote the set of random literals defining~$\psi_1$, i.e.\ $\psi_1=\min(\Lambda)$. Notice firstly that, since $\lvert M_1\rvert\geq f_1$, $\psi_1$ has more clauses than~$\lambda^{(1)}$, and hence
\begin{equation*}
    \prob\bigl(\lvert\Lambda\rvert<f_1-w\bigr)\leq\prob\bigl(\lvert\Lambda^{(1)}\rvert<f_1-w\bigr).
\end{equation*}
Using this, we now write
\begin{align*}
    \prob(\psi_1\in\SAT,\, (\psi_2)_{\psi_1}\in\SAT)\leq{}&\prob\bigl(\psi_1\in\SAT,\, (\psi_2)_{\psi_1}\in\SAT,\,\lvert\Lambda\rvert\geq f_1-w\bigr) \\
    &+\prob\bigl(\lvert\Lambda^{(1)}\rvert<f_1-w\bigr).
\end{align*}
Letting $\mathcal B$ denote the set of all satisfiable $1$-CNF formulas with at least $f_1-w$ (distinct) clauses and $n_1$ variables, we find that
\begin{align*}
    \prob\bigl(\psi_1\in\SAT,\, (\psi_2)_{\psi_1}\in\SAT,\,\lvert\Lambda\rvert\geq f_1-w\bigr)&=\sum_{h\in\mathcal B}\prob(\psi_1=h,\, (\psi_2)_h\in\SAT) \\
    &=\sum_{h\in\mathcal B}\prob(\psi_1=h)\prob((\psi_2)_h\in\SAT) \\
    &\leq\prob\bigl(\psi_1\in\SAT,\,\lvert\Lambda\rvert\geq f_1-w\bigr)\prob\bigl((\psi_2)_{\mathcal L_1}\in\SAT\bigr) \\
    &\leq\prob\bigl(\psi_1\in\SAT\bigr)\prob\bigl((\psi_2)_{\mathcal L_1}\in\SAT\bigr) \\
    &\leq\prob\bigl(\lambda^{(1)}\in\SAT\bigr)\prob\bigl(\psi^{(1)}_{\mathcal L_1}\in\SAT\bigr),
\end{align*}
where we in the first inequality use Lemma~\ref{probdecreaseinc} and in the final inequality use Lemma~\ref{probdecreaseinm}. This yields
\begin{equation*}
    \prob\bigl(\psi_1\in\SAT,\, (\psi_2)_{\psi_1}\in\SAT\bigr)\leq\prob\bigl(\lambda^{(1)}\in\SAT\bigr)\prob\bigl(\varphi^{(1)}_{\mathcal L_1}\in\SAT\bigr)+\prob\bigl(\lvert\Lambda^{(1)}\rvert<f_1-w\bigr),
\end{equation*}
and thus,
\begin{align*}
    \MoveEqLeft[3]\prob\bigl(\lvert\mathcal C^{(1)}_0\rvert=0,\,\lvert\mathcal C^{(1)}_1\rvert\geq f_1,\,\lvert\mathcal C^{(1)}_2\rvert\geq m_1,\,\varphi_1\in\SAT,\, (\varphi_2)_{\varphi_1}\in\SAT\bigr) \\
    \leq{}&\prob\bigl(\lvert\mathcal C^{(1)}_0\rvert=0,\,\lvert\mathcal C^{(1)}_1\rvert\geq f_1,\,\lvert\mathcal C^{(1)}_2\rvert\geq m_1\bigr) \\
    &\times\bigl[\prob\bigl(\lambda^{(1)}\in\SAT\bigr)\prob\bigl(\varphi^{(1)}_{\mathcal L_1}\in\SAT\bigr)+\prob\bigl(\lvert\Lambda^{(1)}\rvert<f_1-w\bigr)\bigr] \\
    \leq{}&\prob\bigl(\lvert\mathcal C^{(1)}_0\rvert=0\bigr)\prob\bigl(\lambda^{(1)}\in\SAT\bigr)\prob\bigl(\varphi^{(1)}_{\mathcal L_1}\in\SAT\bigr)+\prob\bigl(\lvert\Lambda^{(1)}\rvert<f_1-w\bigr).
\end{align*}
All in all, we get
\begin{align*}
    \prob(\varphi_\mathcal L\in\SAT)\leq{}&\prob\bigl(\lvert\mathcal C^{(1)}_0\rvert=0\bigr)\prob\bigl(\lambda^{(1)}\in\SAT\bigr)\prob\bigl(\varphi^{(1)}_{\mathcal L_1}\in\SAT\bigr) \\
    &+\prob\bigl(\lvert\mathcal C^{(1)}_1\rvert<f_1\bigr)+\prob\bigl(\lvert\mathcal C^{(1)}_2\rvert<m_1\bigr)+\prob\bigl(\lvert\Lambda^{(1)}\rvert<f_1-w\bigr).
\end{align*}
We now repeat the procedure on~$\varphi^{(1)}_{\mathcal L_1}$. Put again
\begin{equation*}
    R\coloneqq\bigl\lfloor c\log(n_0)\bigr\rfloor,\quad\text{where}\quad c\coloneqq\frac{1}{16\log(1/\alpha_\epsilon)},
\end{equation*}
and define for each $r\in\mathbb N$:
\begin{equation*}
    n_r\coloneqq n_0-f-\sum_{k=1}^{r-1}(f_k-w),\quad f_r\coloneqq\bigl\lfloor\alpha_\epsilon^r f-r\alpha_\epsilon^r n_0^{3/8}\bigr\rfloor,\quad m_r\coloneqq\biggl\lceil m-\frac{3}{1-\epsilon}(f+n_0^{3/8})\sum_{k=1}^r\alpha_\epsilon^k\biggr\rceil,
\end{equation*}
and finally $\mathcal L_r\coloneqq [n_r]\setminus[n_r-(f_r-w)]$, and let $\mathcal C^{(r)}_k$ denote the set~$\mathcal C_k$ from~\eqref{defnmk} with $\varphi^{(r-1)}$ in place of~$\varphi$ and $\mathcal L_{r-1}$ in place of~$\mathcal L$ for each $k\in K$, and denote finally by $\lambda^{(r)}$ a random $1$-CNF formula with $f_r$ clauses and $n_r$ variables, and by $\Lambda^{(r)}$ the set of random literals defining~$\lambda^{(r)}$, and finally by $\varphi^{(r)}$ a random $2$-CNF formula with $m_r$ clauses and $n_r$ variables. Redoing the argument above $R$ times, and finally bounding $\prob(\varphi^{(R)}_{\mathcal L_R}\in\SAT)\leq 1$, we get
\begin{align*}
    \prob(\varphi_\mathcal L\in\SAT)\leq{}&\prod_{r=1}^R\prob\bigl(\lvert\mathcal C^{(r)}_0\rvert=0\bigr)\prod_{r=1}^R\prob\bigl(\lambda^{(r)}\in\SAT\bigr) \\
    &+\sum_{r=1}^R\prob\bigl(\lvert\mathcal C^{(r)}_1\rvert<f_r\bigr)+\sum_{r=1}^R\prob\bigl(\lvert\mathcal C^{(r)}_2\rvert<m_r\bigr)+\sum_{r=1}^R\prob\bigl(\lvert\Lambda^{(r)}\rvert<f_r-w\bigr).
\end{align*}
The upper bound then follows after a verification of the following points:
\begin{enumerate}[label=(U\arabic*)]
    \item\label{U1} $\lim_{n\to\infty}\prod_{r=1}^R\prob\bigl(\lambda^{(r)}\in\SAT\bigr)=\exp\biggl(\frac{-\beta^2}{4}\sum_{r=1}^\infty\alpha_\epsilon^{2r}\biggr)$,
    \item\label{U2} $\lim_{n\to\infty}\prod_{r=1}^R\prob\bigl(\lvert\mathcal C^{(r)}_0\rvert=0\bigr)\leq\exp\biggl(\frac{-\beta^2}{4}\sum_{r=1}^\infty\alpha_\epsilon^{2r-1}\biggr)$,
    \item\label{U3} $\lim_{n\to\infty}\sum_{r=1}^R\prob\bigl(\lvert\mathcal C^{(r)}_1\rvert<f_r\bigr)=0$,
    \item\label{U4} $\lim_{n\to\infty}\sum_{r=1}^R\prob\bigl(\lvert\mathcal C^{(r)}_2\rvert<m_r\bigr)=0$,
    \item\label{U5} $\lim_{n\to\infty}\sum_{r=1}^R\prob\bigl(\lvert\Lambda^{(r)}\rvert<f_r-w\bigr)=0$.
\end{enumerate}

We initially note that, by Lemma~\ref{distofphiandm},
\begin{equation}
\label{distofmrupper}
    \bigl(\lvert\mathcal C^{(r)}_0\rvert,\lvert\mathcal C^{(r)}_1\rvert,\lvert\mathcal C^{(r)}_2\rvert,\lvert\mathcal C^{(r)}_\star\rvert\bigr)\sim\mathrm{Multinomial}\bigl(m_{r-1},(p^{(r)}_0,p^{(r)}_1,p^{(r)}_2,p^{(r)}_\star)\bigr),
\end{equation}
where we now have for $r\geq 2$
\begin{equation*}
    p^{(r)}_0=\frac{(f_{r-1}-w)(f_{r-1}-w-1)}{4n_{r-1}(n_{r-1}-1)},\quad p^{(r)}_1=\frac{(f_{r-1}-w)n_r}{n_{r-1}(n_{r-1}-1)},\quad p^{(r)}_2=\frac{n_r(n_r-1)}{n_{r-1}(n_{r-1}-1)},
\end{equation*}
and lastly $p^{(r)}_\star=1-p^{(r)}_0-p^{(r)}_1-p^{(r)}_2$. The expressions for~$p^{(1)}_k$, $k\in\{ 0,1,2,\star\}$, are the same, except we do not subtract~$w$ from~$f$ in this case.

We of course still have
\begin{equation}
\label{computelimfrupper}
    \lim_{n\to\infty}\frac{f_r}{\sqrt{n_0}}=\alpha_\epsilon^r\beta
\end{equation}
for all $r\in\mathbb N$, so it follows immediately from Lemma~\ref{fewduplicateliterals} that
\begin{equation*}
    \sum_{r=1}^R\prob\bigl(\lvert\Lambda^{(r)}\rvert<f_r-w\bigr)\leq\sum_{r=1}^R\frac{4\alpha_\epsilon^{2r}\beta^2}{w}\leq 4\beta^2\frac{R}{w}\longrightarrow 0\quad\text{as $n\to\infty$},
\end{equation*}
proving~\ref{U5}.

Putting $C\coloneqq 1/(1-\alpha_\epsilon)$ and $C^\prime\coloneqq 6C/(1-\epsilon)$, we have
\begin{equation}
    f+\sum_{k=1}^{r-1}(f_k-w)\leq\sum_{k=0}^{r-1}f_k\leq\sum_{k=0}^{r-1}\alpha_\epsilon^k f\leq Cf
\end{equation}
for all $r\in\mathbb N$, and therefore $n_0-Cf\leq n_r\leq n_0$. We also immediately find $m_r\leq m$, and $f/n_0^{3/8}\to\infty$ as $n\to\infty$, so $n_0^{3/8}\leq f$ for large enough~$n$, and after this point
\begin{equation*}
    \frac{3}{1-\epsilon}\bigl(f+n_0^{3/8}\bigr)\sum_{k=1}^r\alpha_\epsilon^r\leq\frac{3C}{1-\epsilon}\bigl(f+n_0^{3/8}\bigr)\leq C^\prime f,
\end{equation*}
yielding $m_r\geq m-C^\prime f$ and hence
\begin{equation}
\label{computelimurgr}
    \lim_{n\to\infty}\frac{n_r}{n_0}=1,\quad\text{and}\quad\lim_{n\to\infty}\frac{m_r}{n_0}=\alpha
\end{equation}
for all $r\in\mathbb N$. Of course, we still have
\begin{equation}
\label{betapowerRupper}
    \frac{1}{n_0^{1/16}}\leq\alpha_\epsilon^R\leq\frac{1}{\alpha_\epsilon n_0^{1/16}}
\end{equation}
by the same argument as in~\eqref{betapowerR}. Now, for a lower bound on $f_r$ we notice that when $r\in [R]$ we have $r\alpha_\epsilon^r n_0^{3/8}\leq R\alpha_\epsilon^r n_0^{3/8}$, and
\begin{equation*}
    \frac{\tfrac{1}{2}\alpha_\epsilon^rf}{R\alpha_\epsilon^r n_0^{3/8}}=\frac{f}{2Rn_0^{3/8}}\longrightarrow\infty\quad\text{as $n\to\infty$},
\end{equation*}
where the convergence is uniform in~$r$, so we have $r\alpha_\epsilon^r n_0^{3/8}\leq\tfrac{1}{2}\alpha_\epsilon^rf-1$ for all $r\in[R]$ when $n$ is large enough, and thus
\begin{equation}
\label{lowerboundfrupper}
    f_r=\bigl\lfloor\alpha_\epsilon^r f-r\alpha_\epsilon^r n_0^{3/8}\bigr\rfloor\geq\frac{\alpha_\epsilon^r f}{2}\geq\frac{f}{2n_0^{1/16}}
\end{equation}
for all $r\in[R]$ when $n$ is large enough, where the last inequality comes from~\eqref{betapowerRupper}. Further, since $f/n_0^{1/16}$ is asymptotic to $\beta n_0^{7/16}$ as $n\to\infty$, we see that
\begin{equation*}
    w\leq n_0^{1/4}\leq\frac{f}{4n_0^{1/16}}\leq\frac{f_r}{2}
\end{equation*}
for all $r\in[R]$ when $n$ is large enough, using in the end~\eqref{lowerboundfrupper}, and from these inequalities follows
\begin{equation}
\label{boundfrminusw}
    f_r-w\geq\tfrac{1}{2}f_r.
\end{equation}

For~\ref{U1} we see that, similarly to~\ref{L2}, since $\lambda^{(r)}$ is a random $1$-CNF formula with $f_r$ clauses and $n_r$ variables, Lemma~\ref{random1sat} together with~\eqref{computelimfrupper} and~\eqref{computelimurgr} gives
\begin{equation*}
    \lim_{n\to\infty}\prob\bigl(\lambda^{(r)}\in\SAT\bigr)=e^{-\tfrac{1}{4}\beta^2\alpha_\epsilon^{2r}}
\end{equation*}
for all $r\in\mathbb N$, and to complete the proof of~\ref{U1} we need only give a uniform (in large~$n$) summable (in~$r$) lower bound on $\log\prob(\lambda^{(r)}\in\SAT)$ to be able to apply dominated convergence. By the exact same argument as in the proof of~\ref{L2} we get
\begin{equation*}
    \log\prob\bigl(\lambda^{(r)}\in\SAT\bigr)\geq\frac{-f_r^2}{2n_{r+1}},
\end{equation*}
and we also have
\begin{equation*}
    \lim_{n\to\infty}\frac{f^2}{2(n_0-Cf)}=\frac{\beta^2}{2},
\end{equation*}
so using $f_r\leq\alpha_\epsilon^r f$ and $n_{r+1}\geq n_0-Cf$, we conclude that
\begin{equation*}
    \frac{-f_r^2}{2n_{r+1}}\geq\alpha_\epsilon^{2r}\frac{-f^2}{2(n_0-Cf)}\geq-\alpha_\epsilon^{2r}\beta^2,
\end{equation*}
where the last inequality holds for large enough~$n$ independent of~$r$.

Now for~\ref{U2} we get from~$\eqref{distofmrupper}$ that $\lvert\mathcal C^{(r)}_0\rvert\sim\mathrm{Binomial}(m_{r-1},p^{(r)}_0)$, so similarly to the argument in the proof of~\ref{L3}:
\begin{equation*}
    \prob\bigl(\lvert\mathcal C^{(r)}_0\rvert=0\bigr)=\bigl(1-p^{(r)}_0\bigr)^{m_{r-1}}\longrightarrow e^{-\tfrac{1}{4}\beta^2\alpha_\epsilon^{2(r-1)}\alpha}\leq e^{-\tfrac{1}{4}\beta^2\alpha_\epsilon^{2r-1}}
\end{equation*}
for all $r\in\mathbb N$, and to complete the proof of~\ref{U2} we need only give a uniform (in large~$n$) summable (in~$r$) lower bound on $\log\prob(\lvert\mathcal C^{(r)}_0\rvert=0)$. To this end, we again notice that
\begin{equation*}
    p^{(r)}_0\leq\frac{f^2}{4(n_0-Cf-1)^2}\longrightarrow 0
\end{equation*}
uniformly in~$r$, so $p^{(r)}_0\leq\tfrac{1}{2}$ for all $r\in\mathbb N$ when $n$ is large enough, and
\begin{equation*}
    m_{r-1}p^{(r)}_0\leq\alpha_\epsilon^{2(r-1)}\frac{mf^2}{4(u-Cf-1)^2}\longrightarrow\frac{1}{4}\alpha_\epsilon^{2(r-1)}\alpha\beta^2\leq\frac{1}{4}\alpha_\epsilon^{2r-1}\beta^2,
\end{equation*}
so $m_{r-1}p^{(r)}_0\leq\tfrac{1}{2}\alpha_\epsilon^{2r-1}\beta^2$ for all $r\in\mathbb N$ when $n$ is large enough, and hence we conclude using~\eqref{canonlogbound} that
\begin{equation*}
    \log\prob\bigl(\lvert\mathcal C^{(r)}_0\rvert=0\bigr)=m_{r-1}\log\bigl(1-p^{(r)}_0\bigr)\geq\frac{m_{r-1}p^{(r)}_0}{p^{(r)}_0-1}\geq-\alpha_\epsilon^{2r-1}\beta^2
\end{equation*}
for all $r\in\mathbb N$ when $n$ is large enough.

Now for~\ref{U3}. We get from~\eqref{distofmrupper} that $\lvert\mathcal C^{(r)}_1\rvert\sim\mathrm{Binomial}(m_{r-1},p^{(r)}_1)$, so we see that
\begin{equation*}
    \mean\bigl[\lvert\mathcal C^{(r)}_1\rvert\bigr]=
    \begin{dcases}
        \frac{m_{r-1}(f_{r-1}-w)n_r}{n_{r-1}(n_{r-1}-1)},&\text{if $r\geq 2$}, \\
        \frac{mfn_1}{n_0(n_0-1)},&\text{if $r=1$},
    \end{dcases}
\end{equation*}
and
\begin{equation*}
    \frac{m_{r-1}n_r}{n_{r-1}(n_{r-1}-1)}\geq\frac{(m-C^\prime f)(n_0-Cf)}{n_0^2}\longrightarrow\alpha,
\end{equation*}
so $m_{r-1}n_r/n_{r-1}(n_{r-1}-1)\geq\alpha_\epsilon$ for all $r\in\mathbb N$ when $n$ is large enough. Hence, since
\begin{equation*}
    f_{r-1}-w=\bigl\lfloor\alpha_\epsilon^{r-1}f-(r-1)\alpha_\epsilon^{r-1}n_0^{3/8}\bigr\rfloor-\bigl\lfloor n_0^{1/4}\bigr\rfloor\geq\alpha_\epsilon^{r-1}f-(r-1)\alpha_\epsilon^{r-1}n_0^{3/8}-2n_0^{1/4},
\end{equation*}
and of course $f\geq f-2n_0^{1/4}$ for the case $r=1$, we get
\begin{equation*}
    \mean\bigl[\lvert\mathcal C^{(r)}_1\rvert\bigr]\geq\alpha_\epsilon^r f-r\alpha_\epsilon^r n_0^{3/8}+\alpha_\epsilon^r n_0^{3/8}-2\alpha_\epsilon n_0^{1/4}
\end{equation*}
for all $r\in\mathbb N$ when $n$ is large enough, which is the same as the second inequality in
\begin{equation}
\label{meancr1bound}
    f_r\leq\alpha_\epsilon^r f-r\alpha_\epsilon^r n_0^{3/8}\leq\mean\bigl[\lvert\mathcal C^{(r)}_1\rvert\bigr]-\alpha_\epsilon^r n_0^{3/8}+2\alpha_\epsilon n_0^{1/4}.
\end{equation}
Now, if $r\in [R]$, then $\alpha_\epsilon^r\geq n_0^{-1/16}$ by~\eqref{betapowerRupper}, and of course $2n_0^{1/4}\leq\tfrac{1}{2}n_0^{5/16}$ for large enough~$n$, so in that case
\begin{equation*}
    -\alpha_\epsilon^r n_0^{3/8}+2\alpha_\epsilon n_0^{1/4}\leq-n_0^{5/16}+2n_0^{1/4}\leq -\tfrac{1}{2}n_0^{5/16}
\end{equation*}
for large enough~$n$. We now get from~\eqref{meancr1bound} and the above that
\begin{equation*}
    \prob\bigl(\lvert\mathcal C^{(r)}_1\rvert<f_r\bigr)\leq\prob\bigl(\lvert\mathcal C^{(r)}_1\rvert<\mean\bigl[\lvert\mathcal C^{(r)}_1\rvert\bigr]-\tfrac{1}{2}n_0^{5/16}\bigr)
\end{equation*}
for all $r\in[R]$ when $n$ is large enough. Next, we wish to apply the following Chernoff bound:
\begin{equation*}
    \prob\bigl(N<\mean[N]-\delta\mean[N]\bigr)\leq e^{-\tfrac{1}{2}\delta^2\mean[N]}
\end{equation*}
when $N$ is binomially distributed and $0<\delta<1$. Taking in our case $N=\lvert\mathcal C^{(r)}_1\rvert$ and
\begin{equation*}
    \delta=\frac{n_0^{5/16}}{2\mean\bigl[\lvert\mathcal C^{(r)}_1\rvert\bigr]}\leq\frac{n_0^{5/16}}{2\alpha_\epsilon (f_{r-1}-w)}\leq\frac{n_0^{5/16}}{\alpha_\epsilon f_{r-1}}\leq\frac{2n_0^{3/8}}{\alpha_\epsilon f}\longrightarrow 0,
\end{equation*}
where the first inequality comes from the initial bound on $\mean[\lvert\mathcal C^{(r)}_1\rvert]$ and holds for all $r\in\mathbb N$ when $n$ is large enough, the second comes from~\eqref{boundfrminusw} and holds for all $r\in[R]$ when $n$ is large enough, and the final comes from~\eqref{lowerboundfrupper} and holds for all $r\in[R]$ when $n$ is large enough. Thus, for all $r\in[R]$ it holds for large enough~$n$ that $\delta<1$, so Chernoff's bound gives
\begin{equation*}
    \prob\bigl(\lvert\mathcal C^{(r)}_1\rvert<\mean\bigl[\lvert\mathcal C^{(r)}_1\rvert\bigr]-\tfrac{1}{2}n_0^{5/16}\bigr)\leq\exp\biggl(\frac{-n_0^{5/8}}{8\mean\bigl[\lvert\mathcal C^{(r)}_1\rvert\bigr]}\biggr).
\end{equation*}
Lastly, we note that
\begin{equation*}
    \mean\bigl[\lvert\mathcal C^{(r)}_1\rvert\bigr]\leq f\frac{mn_0}{(n_0-Cf-1)^2},
\end{equation*}
and $mn_0/(n_0-Cf-1)^2\to\alpha$ as $n\to\infty$, so for large enough~$n$ it holds for all $r\in\mathbb N$ that $\mean[\lvert\mathcal C^{(r)}_1\rvert]\leq 2\alpha f$, and using this yields
\begin{equation*}
    \exp\biggl(\frac{-n_0^{5/8}}{8\mean\bigl[\lvert\mathcal C^{(r)}_1\rvert\bigr]}\biggr)\leq\exp\biggl(\frac{-n_0^{5/8}}{16\alpha f}\biggr).
\end{equation*}
Putting it all together, we get
\begin{equation*}
    \sum_{r=1}^R \prob\bigl(\lvert\mathcal C^{(r)}_1\rvert<f_r\bigr)\leq R\exp\biggl(\frac{-n_0^{5/8}}{16\alpha f}\biggr)\longrightarrow 0\quad\text{as $n\to\infty$},
\end{equation*}
as desired.

Now finally for~\ref{U4}. We see for all $r\in\mathbb N$ that, since $\lvert\mathcal C^{(r)}_2\rvert$ can only attain integer values,
\begin{equation*}
    \prob\bigl(\lvert\mathcal C^{(r)}_2\rvert<m_r\bigr)=\prob\Bigl(\lvert\mathcal C^{(r)}_2\rvert<m-\frac{3}{1-\epsilon}\bigl(f+n_0^{3/8}\bigr)\sum_{k=1}^r\alpha_\epsilon^k\Bigr),
\end{equation*}
and
\begin{equation*}
    m-\frac{3}{1-\epsilon}\bigl(f+n_0^{3/8}\bigr)\sum_{k=1}^r\alpha_\epsilon^k\leq m_{r-1}-\frac{3}{1-\epsilon}\bigl(f+n_0^{3/8}\bigr)\alpha_\epsilon^r,
\end{equation*}
so
\begin{equation*}
    \prob\Bigl(\lvert\mathcal C^{(r)}_2\rvert<m-\frac{3}{1-\epsilon}\bigl(f+n_0^{3/8}\bigr)\sum_{k=1}^r\alpha_\epsilon^k\Bigr)\leq\prob\Bigl(m_{r-1}-\lvert\mathcal C^{(r)}_2\rvert>\frac{3\alpha_\epsilon^r}{1-\epsilon}\bigl(f+n_0^{3/8}\bigr)\Bigr).
\end{equation*}
Now, from~\eqref{distofmrupper} we find that $m_{r-1}-\lvert\mathcal C^{(r)}_2\rvert=\lvert\mathcal C^{(r)}_0\rvert+\lvert\mathcal C^{(r)}_1\rvert+\lvert\mathcal C^{(r)}_\star\rvert$, and
\begin{equation*}
    \lvert\mathcal C^{(r)}_0\rvert+\lvert\mathcal C^{(r)}_1\rvert+\lvert\mathcal C^{(r)}_\star\rvert\sim\mathrm{Binomial}\bigl(m_{r-1},1-p^{(r)}_2\bigr).
\end{equation*}
Put for convenience $N^{(r)}\coloneqq\lvert\mathcal C^{(r)}_0\rvert+\lvert\mathcal C^{(r)}_1\rvert+\lvert\mathcal C^{(r)}_\star\rvert$, and notice that for $r\geq 2$, as $n_r=n_{r-1}-(f_{r-1}-w)$,
\begin{equation*}
    p^{(r)}_2=\frac{n_r(n_r-1)}{n_{r-1}(n_{r-1}-1)}=1-(f_{r-1}-w)\frac{2n_{r-1}-(f_{r-1}-w)-1}{n_{r-1}(n_{r-1}-1)},
\end{equation*}
giving us
\begin{equation*}
    1-p^{(r)}_2=(f_{r-1}-w)\frac{2n_{r-1}-(f_{r-1}-w)-1}{n_{r-1}(n_{r-1}-1)}\leq\frac{2}{n_{r-1}-1}f_{r-1}\leq\frac{2}{n_0-Cf-2}f_{r-1},
\end{equation*}
and of course, for $r=1$ this is the last inequality in
\begin{equation*}
    1-p^{(1)}_2=1-\frac{n_1(n_1-1)}{n_0(n_0-1)}=\frac{2n_0-f-1}{n_0(n_0-1)}f\leq\frac{2}{n_0-1}f\leq\frac{2}{n_0-Cf-2}f,
\end{equation*}
which holds, so since $2m/(n_0-Cf-2)\to 2\alpha$, we find that
\begin{equation*}
    \mean[N^{(r)}]=m_{r-1}\bigl(1-p^{(r)}_2\bigr)\leq\frac{2m}{n_0-Cf-2}f_{r-1}\leq 3\alpha f_{r-1}=\frac{3\alpha_\epsilon}{1-\epsilon}f_{r-1}
\end{equation*}
holds for all $r\in\mathbb N$ when $n$ is large enough, and since $f_{r-1}\leq\alpha_\epsilon^{r-1}f-(r-1)\alpha_\epsilon^{r-1}n_0^{3/8}$, we get
\begin{equation}
\label{boundonmeansumm}
    \mean[N^{(r)}]\leq\frac{3\alpha_\epsilon^r}{1-\epsilon}f-(r-1)\frac{3\alpha_\epsilon^r}{1-\epsilon}n_0^{3/8}
\end{equation}
uniformly in~$r$ when $n$ is large enough. Now, continuing our calculations from before and using~\eqref{boundonmeansumm}, we get
\begin{align*}
    \prob\Bigl(m_{r-1}-\lvert\mathcal C^{(r)}_2\rvert>\frac{3\alpha_\epsilon^r}{1-\epsilon}\bigl(f+n_0^{3/8}\bigr)\Bigr)&=\prob\Bigl(N^{(r)}>\frac{3\alpha_\epsilon^r}{1-\epsilon}f+\frac{3\alpha_\epsilon^r}{1-\epsilon}n_0^{3/8}\Bigr) \\
    &\leq\prob\Bigl(N^{(r)}>\mean[N^{(r)}]+r\frac{3\alpha_\epsilon^r}{1-\epsilon}n_0^{3/8}\Bigr)
\end{align*}
for all $r\in\mathbb N$ as long as $n$ is large enough. We are now in a position to apply Chernoff's bound~\eqref{chernoff}, which will require a lower bound on $\mean[N^{(r)}]$. For $r\geq 2$ we find that
\begin{equation*}
    1-p^{(r)}_2=\frac{2n_{r-1}-(f_{r-1}-w)-1}{n_{r-1}(n_{r-1}-1)}(f_{r-1}-w)\geq\frac{2(n_0-Cf)-f-1}{2n_0^2}f_{r-1}\geq\frac{n_0-Cf}{2n_0^2}f_{r-1},
\end{equation*}
where the first inequality holds when $r\in[R]$ and $n$ is large enough due to~\eqref{boundfrminusw}, the last inequality holds for large enough~$n$ since of course $n_0-Cf\geq f+1$ at some point, and by verification the inequality also holds for $r=1$. Further, if $r\in[R]$ and $n$ is large enough we get $f_{r-1}\geq f/(2n_0^{1/16})$ from~\eqref{lowerboundfrupper}, and thus
\begin{equation*}
    \mean[N^{(r)}]=m_{r-1}(1-p^{(r)}_2)\geq\frac{(m-C^\prime f)(n_0-Cf)}{2n_0^2}\cdot\frac{f}{2n_0^{1/16}}\geq\frac{\alpha_\epsilon f}{4n_0^{1/16}}
\end{equation*}
for all $r\in[R]$ when $n$ is large enough, where the last inequality holds because
\begin{equation*}
    \lim_{n\to\infty}\frac{(m-C^\prime f)(n_0-Cf)}{n_0^2}=\alpha.
\end{equation*}
Now, using the above lower bound on $\mean[N^{(r)}]$, we get for the following choice of $\delta$:
\begin{equation*}
    \delta=\frac{3r\alpha_\epsilon^r n_0^{3/8}}{(1-\epsilon)\mean[N^{(r)}]}\leq\frac{12Rn_0^{7/16}}{(1-\epsilon)f}\longrightarrow 0\quad\text{as $n\to\infty$},
\end{equation*}
and hence $\delta<1$ for all $r\in[R]$ when $n$ is large enough. Chernoff's bound~\eqref{chernoff} gives
\begin{equation*}
    \prob\Bigl(N^{(r)}>\mean[N^{(r)}]+r\frac{3\alpha_\epsilon^r}{1-\epsilon}n_0^{3/8}\Bigr)\leq\exp\biggl(\frac{-3r^2 \alpha_\epsilon^{2r}n_0^{3/4}}{(1-\epsilon)^2\mean[N^{(r)}]}\biggr).
\end{equation*}
From~\eqref{boundonmeansumm} it follows that $\mean[N^{(r)}]\leq (3\alpha_\epsilon^r f)/(1-\epsilon)$, and inserting this yields the first inequality in:
\begin{equation*}
    \exp\biggl(\frac{-3r^2 \alpha_\epsilon^{2r}n_0^{3/4}}{(1-\epsilon)^2\mean[N^{(r)}]}\biggr)\leq\exp\biggl(\frac{-r^2\alpha_\epsilon^r n_0^{3/4}}{(1-\epsilon)f}\biggr)\leq\exp\biggl(\frac{-n_0^{11/16}}{f}\biggr),
\end{equation*}
and the second inequality follows from~\eqref{betapowerRupper} and the fact that $r\geq 1$ and $1-\epsilon\leq 1$. The above inequality holds for all $r\in[R]$ when $n$ is large enough. Putting it all together, we conclude that
\begin{equation*}
    \sum_{r=1}^R\prob\bigl(\lvert\mathcal C^{(r)}_2\rvert<m_r\bigr)\leq R\exp\biggl(\frac{-n_0^{11/16}}{f}\biggr)\longrightarrow 0\quad\text{as $n\to\infty$},
\end{equation*}
proving~\ref{U4} and hence also finishing the proof of the case $k=2$ in Theorem~\ref{mainthm} when $\alpha>0$.

\subsection{The random mixed 1- and 2-SAT problem}
In this section, we give a proof of Theorem~\ref{mainthm2}, which follows quite readily from the part of Theorem~\ref{mainthm} that we have just proved. Remember that $n_0\to\infty$, $f/\sqrt{n_0}\to\beta$, and $m/n_0\to\alpha$, where $0\leq\beta\leq\infty$ and $0<\alpha<1$. We are considering a random $2$-CNF formula~$\varphi$ with $m$ clauses and $n_0$ variables and a random $1$-CNF formula~$\lambda$ with $f$ clauses and $n_0$ variables, such that $\varphi$ and~$\lambda$ are independent. Let $\Lambda=\{ L_1,L_2,\dots,L_f\}$ denote the set of the random literals defining~$\lambda$. From~\eqref{splitoffunitclauses} we get
\begin{equation*}
    \varphi\land\lambda\in\SAT\iff\lambda\in\SAT\quad\text{and}\quad\varphi_\Lambda\in\SAT.
\end{equation*}
Let $\mathcal L_0\coloneqq [n_0]\setminus [n_0-f]$ so that $\lvert\mathcal L_0\rvert=f\geq\lvert\Lambda\rvert$, and let $\mathcal B$ denote the set of all satisfiable $1$-CNF formulas with $n_0$ variables and at most $f$ clauses. We get from the above and Lemma~\ref{probdecreaseinc} that
\begin{align*}
    \prob(\varphi\land\lambda\in\SAT)&=\prob(\lambda\in\SAT,\varphi_\Lambda\in\SAT) \\
    &=\sum_{h\in\mathcal B}\prob(\varphi_\Lambda\in\SAT\mid\lambda=h)\prob(\lambda=h) \\
    &=\sum_{h\in\mathcal B}\prob(\varphi_h\in\SAT)\prob(\lambda=h) \\
    &\geq\prob(\varphi_{\mathcal L_0}\in\SAT)\prob(\lambda\in\SAT) \\
    &\longrightarrow\exp\biggl(\frac{-\beta^2\alpha}{4(1-\alpha)}\biggr)\exp\biggl(\frac{-\beta^2}{4}\biggr)\quad(\text{as $n\to\infty$}) \\
    &=\exp\biggl(\frac{-\beta^2}{4(1-\alpha)}\biggr),
\end{align*}
where the convergence comes from the first part of Theorem~\ref{mainthm} and Lemma~\ref{random1sat} respectively.

Now, to get a corresponding upper bound on $\prob(\varphi\land\lambda\in\SAT)$ we need a lower bound on~$\lvert\Lambda\rvert$, i.e.\ we need to bound the number of duplicates in the random literals defining~$\lambda$. We get from Lemma~\ref{fewduplicateliterals} that $\lvert\Lambda\rvert\geq f-\lfloor n_0^{1/4}\rfloor$ with probability tending towards~$1$. Now put $\mathcal L_1\coloneqq[n_0]\setminus[n_0-(f-\lfloor n_0^{1/4}\rfloor)]$ so that $\lvert\mathcal L_1\rvert=f-\lfloor n_0^{1/4}\rfloor$, and we repeat the argument from above: let $\mathcal B$ denote the set of all satisfiable $1$-CNF formulas with $n_0$ variables and at least $f-\lfloor n_0^{1/4}\rfloor$ clauses. Then
\begin{align*}
    \prob(\varphi\land\lambda\in\SAT)={}&\prob\bigl(\lambda\in\SAT,\,\varphi_\Lambda\in\SAT,\,\lvert\Lambda\rvert\geq f-\bigl\lfloor n_0^{1/4}\bigr\rfloor\bigr) \\
    &+\prob\bigl(\lambda\in\SAT,\,\varphi_\Lambda\in\SAT,\,\lvert\Lambda\rvert<f-\bigl\lfloor n_0^{1/4}\bigr\rfloor\bigr),
\end{align*}
where of course the second term vanishes, and using Lemma~\ref{probdecreaseinc}:
\begin{align*}
    \prob\bigl(\lambda\in\SAT,\,\varphi_\Lambda\in\SAT,\,\lvert\Lambda\rvert\geq f-\bigl\lfloor u^{1/4}\bigr\rfloor\bigr)&=\sum_{h\in\mathcal B}\prob(\varphi_h\in\SAT)\prob(\lambda=h) \\
    &\leq\prob(\varphi_{\mathcal L_1}\in\SAT)\sum_{h\in\mathcal B}\prob(\lambda=h) \\
    &\leq\prob(\varphi_{\mathcal L_1}\in\SAT)\prob(\lambda\in\SAT) \\
    &\longrightarrow\exp\biggl(\frac{-\beta^2}{4(1-\alpha)}\biggr)\quad\text{as $n\to\infty$},
\end{align*}
using in the end Theorem~\ref{mainthm}, since $\lvert\mathcal L_1\rvert/\sqrt{n_0}\to\beta$, and Lemma~\ref{random1sat}. This completes the proof of Theorem~\ref{mainthm2}, except for the case $\alpha=0$, which may be proved as follows: let $\psi$ be a random $2$-CNF formula with $\lfloor\epsilon n_0\rfloor$ clauses and $n_0$ variables. Then by Lemma~\ref{probdecreaseinm},
\begin{equation*}
    \liminf_{n\to\infty}\prob(\varphi\land\lambda\in\SAT)\geq\liminf_{n\to\infty}\prob(\psi\land\lambda\in\SAT)=\exp\biggl(\frac{-\beta^2}{4(1-\epsilon)}\biggr),
\end{equation*}
and
\begin{equation*}
    \limsup_{n\to\infty}\prob(\varphi\land\lambda\in\SAT)\leq\limsup_{n\to\infty}\prob(\lambda\in\SAT)=\exp\biggl(\frac{-\beta^2}{4}\biggr),
\end{equation*}
so taking $\epsilon\to 0$ yields the desired result.

\subsection{Sublinear number of binary clauses}

The only thing missing from a proof of Theorem~\ref{mainthm} in the case $k=2$ is the verification of
\begin{equation*}
    \lim_{n\to\infty}\prob(\varphi_\mathcal L\in\SAT)=e^{-(\gamma/2)^2},
\end{equation*}
where $\varphi$ is a random $2$-CNF formula with $m=m(n)$ clauses and $n_0=n_0(n)$ variables satisfying $m\to\infty$ and $m/n_0\to 0$, and $\mathcal L\subseteq\pm[n_0]$ is a consistent set of literals with $\lvert\mathcal L\rvert=f=f(n)$ such that $f\sqrt{m}/n_0\to\gamma$. We may assume without loss of generality that $0<\gamma<\infty$, as the cases $\gamma=0$ and $\gamma=\infty$ then follow by an application of Lemma~\ref{probdecreaseinm}, as we have shown in similar situations previously.

As usual we assume without loss of generality that $\mathcal L=[n_0]\setminus [n_0-f]$, so we get
\begin{equation*}
    \prob(\varphi_\mathcal L\in\SAT)=\prob\bigl(\lvert\mathcal C_0\rvert=0,\,\varphi_1\in\SAT,\, (\varphi_2)_{\varphi_1}\in\SAT\bigr)
\end{equation*}
from~\eqref{phiLsatiff}, and from Lemma~\ref{distofphiandm}:
\begin{equation*}
    \mean[\lvert\mathcal C_0\rvert]=m\frac{f(f-1)}{4n_0(n_0-1)}=\biggl(\frac{1}{2}\cdot\frac{f\sqrt{m}}{n_0}\biggr)^2\frac{f-1}{f}\cdot\frac{n_0}{n_0-1}\longrightarrow\biggl(\frac{\gamma}{2}\biggr)^2,
\end{equation*}
meaning that $\lvert\mathcal C_0\rvert$ is asymptotically Poisson-distributed with mean $(\gamma/2)^2$, yielding
\begin{equation*}
    \prob(\varphi_\mathcal L\in\SAT)\leq\prob(\lvert\mathcal C_0\rvert=0)\longrightarrow e^{-(\gamma/2)^2}.
\end{equation*}
This leaves the lower bound. By the exact same arguments as those leading up to~\eqref{lowerbdfirst}, we get
\begin{equation*}
    \prob(\varphi_\mathcal L\in\SAT)\geq\prob(\lambda\in\SAT)\prob(\psi_{\mathcal L_1}\in\SAT)\prob(\lvert\mathcal C_0\rvert=0)\prob\bigl(\lvert\mathcal C_1\rvert\leq f_1\bigm\vert\lvert\mathcal C_0\rvert=0\bigr),
\end{equation*}
where
\begin{equation*}
    f_1\coloneqq\bigl\lfloor 2\gamma\sqrt{m}+m^{3/8}\bigr\rfloor,
\end{equation*}
and $\lambda$ is a random $1$-CNF formula with $f_1$ clauses and $n_1\coloneqq n_0-f$ variables, $\psi$ is a random $2$-CNF formula with $m$ clauses and $n_1$ variables, and lastly $\mathcal L_1\coloneqq [n_1]\setminus [n_1-f_1]$. Of course, we still have $\prob(\lvert\mathcal C_0\rvert=0)\to e^{-(\gamma/2)^2}$, so we need to show that the other factors tend towards~$1$. We notice that
\begin{equation*}
    \frac{f}{n_0}=\frac{1}{\sqrt{m}}\cdot\frac{f\sqrt{m}}{n_0}\longrightarrow 0,
\end{equation*}
since $m\to\infty$, so $n_1/n_0\to 1$, and $f_1/\sqrt{m}\to 2\gamma$, thus
\begin{equation*}
    \frac{f_1}{\sqrt{n_1}}=\frac{f_1}{\sqrt{m}}\sqrt{\frac{m}{n_0}}\sqrt{\frac{n_0}{n_1}}\longrightarrow 0,
\end{equation*}
and with this Lemma~\ref{random1sat} gives $\prob(\lambda\in\SAT)\to 1$, and Lemma~\ref{alphaisnull} gives $\prob(\psi_{\mathcal L_1}\in\SAT)\to 1$. Next, $\lvert\mathcal C_1\rvert$ has a binomial distribution with parameters~$m$ and
\begin{equation*}
    p_{1\mid 0}\coloneqq\frac{f(n_0-f)}{n_0(n_0-1)-\tfrac{1}{4}f(f-1)}
\end{equation*}
given $\lvert\mathcal C_0\rvert=0$, again by Lemma~\ref{distofphiandm}. Let $N\sim\mathrm{Binom}(m,p_{1\mid 0})$, and notice that
\begin{equation*}
    \mean[N]=\frac{mf(n_0-f)}{n_0(n_0-1)-\tfrac{1}{4}f(f-1)}=\frac{f\sqrt{m}}{n_0}\cdot\frac{n_0-f}{n_0}\cdot\frac{n_0^2}{n_0(n_0-1)-\frac{1}{4}f(f-1)}\sqrt{m},
\end{equation*}
so since
\begin{equation*}
    \lim_{n\to\infty}\biggl[\frac{f\sqrt{m}}{n_0}\cdot\frac{n_0-f}{n_0}\cdot\frac{n_0^2}{n_0(n_0-1)-\frac{1}{4}f(f-1)}\biggr]=\gamma,
\end{equation*}
we conclude that
\begin{equation*}
    \tfrac{1}{2}\gamma\sqrt{m}\leq\mean[N]\leq 2\gamma\sqrt{m}
\end{equation*}
when $n$ is large enough. In particular we have
\begin{equation*}
    \delta=\frac{m^{3/8}}{\mean[N]}\leq\frac{2}{\gamma m^{1/8}}<1
\end{equation*}
for large enough $n$, so applying Chernoff's bound~\eqref{chernoff} yields:
\begin{align*}
    \prob\bigl(\lvert\mathcal C_1\rvert>f_1\bigm\vert\lvert\mathcal C_0\rvert=0\bigr)&=\prob(N>f_1) \\
    &=\prob\bigl(N>2\gamma\sqrt{m}+m^{3/8}\bigr) \\
    &\leq\prob\bigl(N>\mean[N]+m^{3/8}\bigr) \\
    &\leq\exp\biggl(\frac{-m^{3/4}}{3\mean[N]}\biggr) \\
    &\leq\exp\biggl(\frac{-m^{1/4}}{6\gamma}\biggr)\longrightarrow 0
\end{align*}
as $n\to\infty$, finishing the proof.

\subsection{Random 3-SAT}

We are able to prove Theorem~\ref{mainthm} directly in the case $k=3$ (without assuming $\alpha>0$). The proof in this section will follow along the same lines as before, but it is a bit more notationally heavy than in the case $k=2$, as there are now, in addition to $0$-, $1$-, and $2$-clauses, also $3$-clauses appearing, but it is mathematically more elementary, since the decomposition~\eqref{phiLsatiff3} only occurs three times, so there are no infinite series to deal with. We seek to prove that
\begin{equation*}
    \lim_{n\to\infty}\prob(\varphi_\mathcal L\in\SAT)=e^{-(\gamma/2)^3},
\end{equation*}
where $\varphi$ is a random $3$-CNF formula with $m=m(n)$ clauses and $n_0=n_0(n)$ variables, $\mathcal L\subseteq\pm[n_0]$ is a consistent set of literals with $\lvert\mathcal L\rvert=f=f(n)$, and $m\to\infty$, $m/n_0\to\alpha$, and $fm^{1/3}/n_0\to\gamma$ as $n\to\infty$, where $0\leq\alpha<3.145$ and $0\leq\gamma\leq\infty$. Assume without loss of generality that $\mathcal L=[n_0]\setminus [n_0-f]$ and $0<\gamma<\infty$.

From~\eqref{phiLsatiff3} we have
\begin{equation*}
    \prob(\varphi_\mathcal L\in\SAT)=\prob\bigl(\lvert\mathcal C^{(1)}_0\rvert=0,\,\varphi_1\in\SAT,\, (\varphi_2\land\varphi_3)_{\varphi_1}\in\SAT\bigr),
\end{equation*}
and from Lemma~\ref{distofphiandm} it follows that $\lvert\mathcal C^{(1)}_0\rvert$ is Binomially distributed with mean
\begin{equation*}
    \mean\bigl[\lvert\mathcal C^{(1)}_0\rvert\bigr]=m\frac{f(f-1)(f-2)}{8n_0(n_0-1)(n_0-2)}\longrightarrow\frac{\gamma^3}{8},
\end{equation*}
so $\lvert\mathcal C^{(1)}_0\rvert$ is asymptotically Poisson-distributed with mean $(\gamma/2)^3$, and thus
\begin{equation*}
    \lim_{n\to\infty}\prob\bigl(\lvert\mathcal C^{(1)}_0\rvert=0\bigr)=e^{-(\gamma/2)^3},
\end{equation*}
which immediately yields the correct upper bound for the limit $\lim_{n\to\infty}\prob(\varphi_\mathcal L\in\SAT)$.

For the lower bound we apply the same line of reasoning as in the lower bound in the case $k=2$. Put $K\coloneqq\{ 0,1,2,3,\star\}$ and $\mathcal C\coloneqq (\mathcal C^{(1)}_k)_{k\in K}$. Put further
\begin{equation*}
    f_1\coloneqq\bigl\lfloor\gamma^2 m^{1/3}+m^{1/5}\bigr\rfloor\quad\text{and}\quad m_1\coloneqq\bigl\lfloor 2\gamma m^{2/3}+m^{2/5}\bigr\rfloor.
\end{equation*}
We have
\begin{align*}
    \MoveEqLeft\prob\bigl(\lvert\mathcal C^{(1)}_0\rvert=0,\,\varphi_1\in\SAT,\, (\varphi_2\land\varphi_3)_{\varphi_1}\in\SAT\bigr) \\
    &\geq\prob\bigl(\lvert\mathcal C^{(1)}_0\rvert=0,\,\lvert\mathcal C^{(1)}_1\rvert\leq f_1,\,\lvert\mathcal C^{(1)}_2\rvert\leq m_2,\,\varphi_1\in\SAT,\, (\varphi_2\land\varphi_3)_{\varphi_1}\in\SAT\bigr) \\
    &=\sum_{M\in\mathfrak M}\prob\bigl(\varphi_1\in\SAT,\, (\varphi_2\land\varphi_3)_{\varphi_1}\in\SAT\bigm\vert\mathcal C=M\bigr)\prob(\mathcal C=M),
\end{align*}
where $\mathfrak M$ denotes the collection of all partitions $M=(M_k)_{k\in K}$ of $[m]$ satisfying $\lvert M_0\rvert=0$, $\lvert M_1\rvert\leq f_1$, and $\lvert M_2\rvert\leq m_1$. By the exact same argument as in~\eqref{transfermethod} we get
\begin{equation*}
    \prob\bigl(\varphi_1\in\SAT,\, (\varphi_2\land\varphi_3)_{\varphi_1}\in\SAT\bigm\vert\mathcal C=M\bigr)\geq\prob\bigl(\lambda^{(1)}\in\SAT\bigr)\prob\bigl((\psi^{(1)}\land\varphi^{(1)})_{\mathcal L_1}\in\SAT\bigr),
\end{equation*}
where $\lambda^{(1)}$ is a random $1$-CNF formula with $f_1$ clauses and $n_1\coloneqq n_0-f$ variables, $\psi^{(1)}$ is a random $2$-CNF formula with $m_1$ clauses and $n_1$ variables, and $\varphi^{(1)}$ is a random $3$-CNF formula with $m$ clauses and $n_1$ variables such that (the clauses of) $\psi^{(1)}$ and~$\varphi^{(1)}$ are independent, and finally $\mathcal L_1\coloneqq [n_1]\setminus [n_1-f_1]$. Hence,
\begin{equation}
\label{decomp3}
\begin{aligned}
    \prob(\varphi_\mathcal L\in\SAT)\geq{}&\prob\bigl(\lambda^{(1)}\in\SAT\bigr)\prob\bigl((\psi^{(1)}\land\varphi^{(1)})_{\mathcal L_1}\in\SAT\bigr) \\
    &\times\prob\bigl(\lvert\mathcal C^{(1)}_0\rvert=0,\,\lvert\mathcal C^{(1)}_1\rvert\leq f_1,\,\lvert\mathcal C^{(1)}_2\rvert\leq m_1\bigr).
\end{aligned}
\end{equation}
Lemma~\ref{distofphiandm} again tells us that both $\lvert\mathcal C^{(1)}_1\rvert$ and $\lvert\mathcal C^{(1)}_2\rvert$ are binomially distributed with parameters $m$ and $p^{(1)}_1$ or $p^{(1)}_2$ respectively, where
\begin{equation*}
    p^{(1)}_1=\frac{3f(f-1)(n_0-f)}{4n_0(n_0-1)(n_0-2)},\quad\text{and}\quad p^{(1)}_2=\frac{3f(n_0-f)(n_0-f-1)}{2n_0(n_0-1)(n_0-2)}.
\end{equation*}
A quick calculation shows that $\mean[\lvert\mathcal C^{(1)}_1\rvert]$ is asymptotic to $\tfrac{3}{4}\gamma^2 m^{1/3}$ and $\mean[\lvert\mathcal C^{(1)}_2\rvert]$ is asymptotic to $\tfrac{3}{2}\gamma m^{2/3}$ as $n\to\infty$. Thus, we can apply Chernoff's bound~\eqref{chernoff} to receive
\begin{equation*}
    \prob\bigl(\lvert\mathcal C^{(1)}_1\rvert>f_1\bigr)\leq\prob\bigl(\lvert\mathcal C^{(1)}_1\rvert>\mean\bigl[\lvert\mathcal C^{(1)}_1\rvert\bigr]+m^{1/5}\bigr)\leq\exp\biggl(\frac{-m^{2/5}}{3\mean\bigl[\lvert\mathcal C^{(1)}_1\rvert\bigr]}\biggr)\leq\exp\biggl(\frac{-m^{1/15}}{3\gamma^2}\biggr),
\end{equation*}
where the first and last inequality holds for large enough~$n$, so that $\lvert\mathcal C^{(1)}_1\rvert\leq f_1$ w.h.p. By a similar argument we see that $\lvert\mathcal C^{(1)}_2\rvert\leq m_1$ w.h.p. This means that
\begin{equation*}
    \lim_{n\to\infty}\prob\bigl(\lvert\mathcal C^{(1)}_0\rvert=0,\,\lvert\mathcal C^{(1)}_1\rvert\leq f_1,\,\lvert\mathcal C^{(1)}_2\rvert\leq m_1\bigr)=e^{-(\gamma/2)^3},
\end{equation*}
so it remains only to show that the other factors in~\eqref{decomp3} tend towards~$1$ as $n\to\infty$. We notice that
\begin{equation*}
    \frac{f}{n_0}=\frac{fm^{1/3}}{n_0}\cdot\frac{1}{m^{1/3}}\longrightarrow 0,\quad\text{so}\quad\frac{n_1}{n_0}\longrightarrow 1,\quad\text{and thus}\quad\frac{f_1}{\sqrt{n_1}}=\frac{f_1}{\sqrt{m}}\sqrt{\frac{m}{n_0}}\sqrt{\frac{n_0}{n_1}}\longrightarrow 0.
\end{equation*}
Lemma~\ref{random1sat} thus gives $\prob(\lambda^{(1)}\in\SAT)\to 1$.

Looking at the final factor, let for each $k\in\{0,1,2,\star\}$ $\mathcal C^{(2)}_k(2)$ denote the set of $j\in [m_1]$ for which the $j$'th clause of $\psi^{(1)}$ becomes a $k$-clause when fixing the variables dictated by $\mathcal L_1$ (cf.~\eqref{defnmk}), and let $\mathcal C^{(2)}_k(3)$ denote the corresponding set for $\varphi^{(1)}$. Put
\begin{equation*}
    \lvert\mathcal C^{(2)}_k\rvert\coloneqq\lvert\mathcal C^{(2)}_k(2)\rvert+\lvert\mathcal C^{(2)}_k(3)\rvert,
\end{equation*}
where the two terms are seen to be independent. As before we are able to make the following decomposition:
\begin{align*}
    \prob\bigl((\psi^{(1)}\land\varphi^{(1)})_{\mathcal L_1}\in\SAT\bigr)\geq{}&\prob\bigl(\lambda^{(2)}\in\SAT\bigr)\prob\bigl((\psi^{(2)}\land\varphi^{(2)})_{\mathcal L_2}\in\SAT\bigr) \\
    &\times\prob\bigl(\lvert\mathcal C^{(2)}_0\rvert=0,\,\lvert\mathcal C^{(2)}_1\rvert\leq f_2,\,\lvert\mathcal C^{(2)}_2\rvert\leq m_2\bigr),
\end{align*}
where $n_2\coloneqq n_1-f_1$,
\begin{equation*}
    f_2\coloneqq\bigl\lfloor m^{1/5}\bigr\rfloor,\quad m_2\coloneqq\bigl\lfloor 3\gamma m^{2/3}+m^{2/5}\bigr\rfloor,\quad\text{and}\quad\mathcal L_2\coloneqq [n_2]\setminus [n_2-f_2],
\end{equation*}
and $\lambda^{(2)}$ is a random $1$-CNF formula with $f_2$ clauses and $n_2$ variables, $\psi^{(2)}$ is a random $2$-CNF formula with $m_2$ clauses and $n_2$ variables, and $\varphi^{(2)}$ is a random $3$-CNF formula with $m$ clauses and $n_2$ variables such that (the clauses of) $\psi^{(2)}$ and~$\varphi^{(2)}$ are independent.

As before $f_2/\sqrt{n_2}\to 0$, so Lemma~\ref{random1sat} gives $\mathbb P(\lambda^{(2)}\in\SAT)\to 1$. Looking at Lemmas~\ref{distofphiandm}, a quick calculation shows that $\mean[\lvert\mathcal C^{(2)}_0(2)\rvert]$ is asymptotic to $\tfrac{1}{2}\gamma^5\alpha^{4/3}n_0^{-2/3}\to 0$, and $\mean[\lvert\mathcal C^{(2)}_0(3)\rvert]$ is asymptotic to $\frac{1}{8}\gamma^6\alpha^2 n_0^{-1}\to 0$, and it follows that $\lvert\mathcal C^{(2)}_0\rvert=0$ w.h.p.\ by an application of Markov's inequality:
\begin{equation*}
    \prob\bigl(\lvert\mathcal C^{(2)}_0\rvert>0\bigr)\leq\mean\bigl[\lvert\mathcal C^{(2)}_0\rvert\bigr]\longrightarrow 0.
\end{equation*}
Next, we find that $\mean[\lvert\mathcal C^{(2)}_1(2)\rvert]\to 2\gamma^3 \alpha$ and $\mean[\lvert\mathcal C^{(2)}_1(3)\rvert]$ is asymptotic to $\tfrac{3}{4}\gamma^4 \alpha^{5/3}n_0^{-1/3}\to 0$, so $\lvert\mathcal C^{(2)}_1\rvert\leq f_2$ w.h.p.\ again by Markov's inequality:
\begin{equation*}
    \prob\bigl(\lvert\mathcal C^{(2)}_1\rvert>f_2\bigr)\leq\frac{\mean\bigl[\lvert\mathcal C^{(2)}_1\rvert\bigr]}{f_2}\longrightarrow 0.
\end{equation*}
Lastly, $\mean[\lvert\mathcal C^{(2)}_2(2)\rvert]$ is asymptotic to $2\gamma m^{2/3}$ and $\mean[\lvert\mathcal C^{(2)}_2(3)\rvert]$ is asymptotic to $\tfrac{3}{2}\gamma^2 \alpha m^{1/3}$, so $\mean[\lvert\mathcal C^{(2)}_2\rvert]\leq 3\gamma m^{2/3}$ for large enough~$n$, and Chernoff's inequality~\eqref{chernoff} yields
\begin{equation*}
    \prob\bigl(\lvert\mathcal C^{(2)}_2\rvert>m_2\bigr)\leq\prob\bigl(\lvert\mathcal C^{(2)}_2\rvert>\mean\bigl[\lvert\mathcal C^{(2)}_2\rvert\bigr]+m^{2/5}\bigr)\leq\exp\biggl(\frac{-m^{4/5}}{3\mean\bigl[\lvert\mathcal C^{(2)}_2\rvert\bigr]}\biggr)\leq\exp\biggl(\frac{-m^{2/15}}{9\gamma^2}\biggr),
\end{equation*}
and thus $\lvert\mathcal C^{(2)}_2\rvert\leq m_2$ w.h.p. Hence, we are only missing a verification of the fact that $(\psi^{(2)}\land\varphi^{(2)})_{\mathcal L_2}$ is satisfiable w.h.p. We make the final decomposition:
\begin{equation*}
    \prob\bigl((\psi^{(2)}\land\varphi^{(2)})_{\mathcal L_2}\in\SAT\bigr)\geq\prob\bigl(\psi^{(3)}\land\varphi^{(3)}\in\SAT\bigr)\prob\bigl(\lvert\mathcal C^{(3)}_0\rvert=0,\,\lvert\mathcal C^{(3)}_1\rvert=0,\,\lvert\mathcal C^{(3)}_2\rvert\leq m_3\bigr),
\end{equation*}
where $n_3\coloneqq n_2-f_2$ and $m_3\coloneqq\lfloor 4\gamma m^{2/3}+m^{2/5}\rfloor$, $\psi^{(3)}$ is a random $2$-CNF formula with $m_3$ clauses and $n_3$ variables, and $\varphi^{(3)}$ is a random $3$-CNF formula with $m$ clauses and $n_3$ variables such that (the clauses of) $\psi^{(3)}$ and~$\varphi^{(3)}$ are independent (this corresponds to taking $f_3\coloneqq 0$).

We see that $\mean[\lvert\mathcal C^{(3)}_0(2)\rvert]$ is asymptotic to $\tfrac{3}{4}\gamma\alpha^{16/15}n_0^{-14/15}\to 0$, and $\mean[\lvert\mathcal C^{(3)}_0(3)\rvert]$ is asymptotic to $\tfrac{1}{8}\alpha^{8/5}n_0^{-7/5}\to 0$, so $\lvert\mathcal C^{(3)}_0\rvert=0$ w.h.p. Similarly $\mean[\lvert\mathcal C^{(3)}_1(2)\rvert]$ is asymptotic to $3\gamma\alpha^{13/15}n_0^{-2/15}\to 0$, and $\mean[\lvert\mathcal C^{(3)}_1(3)\rvert]$ is asymptotic to $\tfrac{3}{4}\alpha^{7/5}n_0^{-3/5}\to 0$, so again $\lvert\mathcal C^{(3)}_1\rvert=0$ w.h.p. Finally, $\mean[\lvert\mathcal C^{(3)}_2(2)\rvert]$ is asymptotic to $3\gamma m^{2/3}$ and $\mean[\lvert\mathcal C^{(3)}_2(3)\rvert]$ is asymptotic to $\tfrac{3}{2}\alpha m^{1/5}$, so $\mean[\lvert\mathcal C^{(3)}_2\rvert]\leq 4\gamma m^{2/3}$ for large enough~$n$, and by Chernoff's inequality we find that $\lvert\mathcal C^{(3)}_2\rvert\leq m_3$ w.h.p.

The final step in the proof is to show that the mixed formula $\psi^{(3)}\land\varphi^{(3)}$ is satisfiable w.h.p. Since $m/n_0\to\alpha$, $m_3/m^{2/3}\to 4\gamma$, and $n_3/n_0\to 1$, it follows that $m_3\leq\lfloor 10^{-6}n_3\rfloor$ for large enough~$n$. It further holds that $m/n_3\to\alpha<3.145$, so $m\leq\lfloor 3.145 n_3\rfloor$ for large enough~$n$. Let $\Phi$ denote a random mixed CNF formula with $n_3$ variables, $\lfloor 10^{-6}n_3\rfloor$ 2-clauses, and $\lfloor 3.145 n_3\rfloor$ 3-clauses. It follows from Lemma~\ref{probdecreaseinm} that
\begin{equation*}
    \liminf_{n\to\infty}\prob\bigl(\psi^{(3)}\land\varphi^{(3)}\in\SAT\bigr)\geq\liminf_{n\to\infty}\prob(\Phi\in\SAT).
\end{equation*}
From the first lines in the proof of Theorem~1 in~\cite{achlioptas00} it follows that $\liminf_{n\to\infty}\prob(\Phi\in\SAT)>0$, and it then follows from Theorem~2 from~\cite{AKKK01} that $\Phi$ is satisfiable w.h.p., completing the proof of our Theorem~\ref{mainthm}.

\subsection{The random mixed 1- and 3-SAT problem}

The final part is the proof of Theorem~\ref{mainthm2} in the case $k=3$. Thus, we again have $n_0\to\infty$, and we let~$\varphi$ denote a random $3$-CNF formula with $m$ clauses and $n_0$ variables, such that $m/n_0\to\alpha$ for some $0\leq\alpha<3.145$, and let~$\lambda$ denote a random $1$-CNF formula with $f$ clauses and $n_0$ variables, such that $f/\sqrt{n_0}\to\beta$ for some $0\leq\beta\leq\infty$, where $\varphi$ and $\lambda$ are independent. Then we immediately have
\begin{equation*}
    \limsup_{n\to\infty}\prob(\varphi\land\lambda\in\SAT)\leq\limsup_{n\to\infty}\prob(\lambda\in\SAT)=e^{-(\beta/2)^2}
\end{equation*}
from Lemma~\ref{probdecreaseinm} and Lemma~\ref{random1sat}.

To obtain the corresponding lower bound, we first note from~\eqref{splitoffunitclauses} that
\begin{equation*}
    \varphi\land\lambda\in\SAT\iff\lambda\in\SAT\quad\text{and}\quad\varphi_\lambda\in\SAT.
\end{equation*}
Let $\mathcal L\coloneqq [n_0]\setminus [n_0-f]$ and note that $\lambda$ will always have at most $f=\lvert\mathcal L\rvert$ distinct clauses. Thus, letting $\mathcal B$ denote the set of all satisfiable $1$-CNF formulas with $n_0$ variables and at most $f$ (distinct) clauses, we get from the above and Lemma~\ref{probdecreaseinc} that
\begin{align*}
    \prob(\varphi\land\lambda\in\SAT)&=\prob(\lambda\in\SAT,\varphi_\lambda\in\SAT) \\
    &=\sum_{h\in\mathcal B}\prob(\varphi_\lambda\in\SAT\mid\lambda=h)\prob(\lambda=h) \\
    &=\sum_{h\in\mathcal B}\prob(\varphi_h\in\SAT)\prob(\lambda=h) \\
    &\geq\prob(\varphi_\mathcal L\in\SAT)\prob(\lambda\in\SAT) \\
    &\longrightarrow e^{-(\beta/2)^2}\quad\text{as $n\to\infty$},
\end{align*}
where we in the end use Lemma~\ref{random1sat} and Theorem~\ref{mainthm}. This concludes the proof of Theorem~\ref{mainthm2} and thus the entire article.

\newpage
\phantomsection
\addcontentsline{toc}{section}{References}
\printbibliography

\end{document}